\documentclass[11pt]{article}%
\usepackage{fullpage,graphicx,psfrag,amsmath,amsfonts,verbatim}
\usepackage[small,bf]{caption}
\usepackage{amsmath, amssymb, amsthm, amsfonts, amsxtra, bbm, mathrsfs,
stmaryrd}
\usepackage{algorithm}
\usepackage{algorithmic}
\usepackage[CJKbookmarks=true, bookmarksnumbered=true, bookmarksopen=true,
bookmarks=false, colorlinks=true, citecolor=red, linkcolor=blue,
anchorcolor=red, urlcolor=blue ]{hyperref}
\usepackage{enumitem, hyperref, listings}
\usepackage{graphicx,epstopdf,caption,subcaption}
\usepackage{tikz-cd}
\usepackage{amsmath}
\usepackage{amsfonts}
\usepackage{amssymb}
\usepackage{mathtools}
\usepackage{graphicx}%
\providecommand{\U}[1]{\protect\rule{.1in}{.1in}}
\theoremstyle{plain}
\newtheorem{theorem}{Theorem}
\newtheorem{lemma}{Lemma}

\newtheorem{assumption}{Assumption}
\newtheorem{prop}{Proposition}
\newtheorem{remark}{Remark}
\newtheorem{corollary}{Corollary}
\theoremstyle{definition}

\newtheorem{definition}{Definition}[section]

\def\EE{\mathbb{E}}

\def\PP{\mathbb{P}}

\def\RR{\mathbb{R}}

\def\ZZ{\mathbb{Z}}

\def\calF{\mathcal{F}}

\def\calL{\mathcal{L}}
\def\calM{\mathcal{M}}

\def\calP{\mathcal{P}}

\def\calS{\mathcal{S}}

\def\calW{\mathcal{W}}

\def\bI{\boldsymbol{I}}
\def\bJ{\boldsymbol{J}}

\def\bX{\mathbf{X}}

\def\bpi{\boldsymbol{\pi}}
\def\bPi{\boldsymbol{\Pi}}
\def\balpha{\boldsymbol{\alpha}}
\def\bbeta{\boldsymbol{\beta}}
\def\bmu{\boldsymbol{\mu}}
\def\bnu{\boldsymbol{\nu}}
\def\bx{\mathbf{x}}
\def\calMM{\boldsymbol{\calM}}

\def\1{\mathbbm{1}}

\begin{document}

\title{Distributionally Robust Martingale Optimal Transport}
\author{Zhengqing Zhou, Jose H. Blanchet and Peter W. Glynn  \thanks{Zhengqing Zhou is with the Department of Mathematics, Stanford University. Jose H. Blanchet and Peter W. Glynn are with the Department of Management Science \& Engineering, Stanford University.  \newline
Email:  \href{zqzhou@stanford.edu}{zqzhou@stanford.edu},
\href{jose.blanchet@stanford.edu}{jose.blanchet@stanford.edu}, \href{glynn@stanford.edu}{glynn@stanford.edu}.}}
\maketitle

\begin{abstract}
We study the problem of bounding path-dependent expectations (within any finite time horizon $d$) over the class of discrete-time martingales whose marginal
distributions lie within a prescribed tolerance of a given collection of
benchmark marginal distributions. This problem is a
relaxation of the martingale optimal transport (MOT) problem and is motivated
by applications to super-hedging in financial markets. We show that the empirical version of our relaxed MOT problem can be approximated within
$O\left(  n^{-1/2}\right)$ error where $n$ is the number of samples of each
of the individual marginal distributions (generated independently) and using a suitably constructed finite-dimensional linear programming problem. 


\end{abstract}

\section{Introduction}

Martingale Optimal Transport (MOT), an optimization problem that is motivated by  model-free pricing on exotic options in mathematical finance, has been widely studied in the past few years \cite{model_indep_bound_Beiglbock2013, MOT_continuous_Dolinsky2014, stochasticControl_MOT_galichon2014,  robustPriceBound_Hobson2015, MOT_beiglbock2016, multi_MOT_Lim2016, completeDualMOT_Touzi2017, OPTandSkorokhodEmbedding_Beiglbock2017, multi_MOT_Lim2016, MOT_generalDimension_Lim2019, NUTZ_multiMOT_2020}. MOT can be viewed as 
the standard optimal transport problem with a martingale constraints on the underlying measure. Despite the theoretical appeal, the computational analysis of MOT is challenging especially for the multi-period case. For example, in general, as far as we understand, there is no convergence rate analysis in the existing literature. Motivated by the computational challenges of MOT, we introduce a new formulation of MOT, which we shall refer to as $\varepsilon$-distributionally robust martingale optimal transport problem ($\varepsilon$-DRMOT). Our relaxation recovers the MOT problem by setting $\varepsilon=0$. 

In contrast to the standard MOT, our formulation allow the marginal distributions at each time to be misspecified by a tolerance error controlled by $\varepsilon$. From a modeling standpoint, in the finance setting, for instance, this relaxation may capture that available option prices may not be sufficient to recover the marginal distributions of the implicit market (risk neutral) measure or that there are price distortions, such as bid-ask spreads, which make it impossible to realistically impose a specific marginal distribution. Moreover, there are other potential applications, apart from finance, in which martingale-like constraints are natural. For example, linear regression models, which are widely used in machine learning can also be interpreted basically as a martingale constraint. In these machine learning settings, allowing for misspecification in the marginal distributions is common and therefore our proposed framework can be used to inform distributionally robust formulations involving decisions constrained by conditional information. Recent efforts in this area, involving optimal transport and distributional uncertainty regions which are informed by conditional information, are discussed in \cite{jose2020DRO_conditional_estimation}. Another active research area that could benefit from our development corresponds to multi-stage distributionally robust optimization. A significant challenge in this area is enforcing adaptivity of the adversary. Using martingales (which are linear constraints) to enforce adaptivity is both convenient and natural in many settings.

In addition to introducing $\varepsilon$-DRMOT, we also develop  numerical discretization schemes which can be used to approximate $\varepsilon$-DRMOT in the multi-period setting. In order to provide a precise summary of our contributions, let us first introduce the  mathematical description of $\varepsilon$-DRMOT in the following section.

\subsection{The mathematical formulation and its motivation}
We start by describing the set of $d$-step martingale measures defined on a product space $\Omega = \bigtimes_{i=1}^d\Omega_i$ that is endowed with the Borel $\sigma$-field, where $\Omega_i\subseteq\RR$ for $1\leq i\leq d$. For a given measurable space $S$, let $\calP(S)$ denote the set of probability measures on $S$. Also, let $\bX:\Omega \rightarrow \Omega$ be the mapping for which $\bX(\omega) = \omega$ (this is also called the canonical embedding) and write $\bX = (X_1,\cdots, X_d)$. Given $\pi\in\mathcal{P}\left(  \Omega\right)$,  we define $\PP_{\pi}(\cdot)$ via $\PP_{\pi}(\bX\in\cdot) = \pi(\cdot)$ and let $\EE_{\pi}(\cdot)$ be its associated expectation operator. Moreover, we 
 adapt the convention to use lowercase alphabet when we write the expectation in terms of integral. For instance,  $\EE_{\pi}\left[f(\bX)\right] = \int f(\bx)\pi(d\bx)$.  Then, the space of $d$-step \textit{martingale measures} on $\Omega$ is given by 
\begin{equation*}
\mathcal{M}\left(\Omega\right)  =\left\lbrace \pi
\in\mathcal{P}\left(\Omega\right)  :\mathbb{E}%
_{\pi}\left[ X_{k+1}|\calF_k\right]  =X_k, \pi-\textrm{a.e. for } 1\leq k\leq d-1\right\rbrace,
\end{equation*}
where $\calF_k = \sigma (X_1,\cdots, X_k)$ is the canonical filtration ($1\leq k\leq d-1$). For any $\pi\in\calP(\Omega)$, let $\pi_{i}$ be its 
$i$-th marginal measure (which is supported on $\Omega_{i}$), so that $\pi_i(\cdot)=\PP_{\pi}(\bX_i\in\cdot)$. We say that the joint probability measure $\pi$ is \textit{consistent with the sequence of marginals} $\bmu = (\mu_1,\cdots,\mu_d)\in\bigtimes_{i=1}^d\calP(\Omega_i)$ if $\pi_i = \mu_i$ for $1\leq i\leq d$, and then we write $\pi\in\Pi(\bmu)$. In other words, $\Pi(\boldsymbol{\mu})$ denote the set of probability measures on $\Omega$ such that its marginal measures are exactly $\mu_1,\cdots, \mu_d$ - in this order. Then, \begin{equation*}
    \calM_{\bmu}(\Omega) := \calM(\Omega)\cap\Pi(\bmu)
\end{equation*}
is the set of martingale measures consistent with $\bmu$. It is known in \cite{strassen1965} that $\calM_{\bmu}(\Omega)\neq\varnothing$ if and only if
\begin{equation*}
    \mu_1\preceq_c\cdots\preceq_c\mu_d,
\end{equation*}
where $\preceq_c$ denotes convex ordering (reviewed in Section \ref{sect.DRMOT_compact}). Given the convex ordering, $\calM_{\bmu}(\Omega)$ is typically not a singleton. For a payoff function $f:\Omega \rightarrow \RR$, we are then naturally led to the study of the worst case upper bound
\begin{equation}\label{def.standard_MOT}
    \sup_{\pi\in\calM_{\bmu}(\Omega)}\EE_{\pi}\left[f(\bX)\right].
\end{equation}
The above upper bound recovers the standard MOT formulation discussed in, for example, \cite{model_indep_bound_Beiglbock2013, robustPriceBound_Hobson2015,completeDualMOT_Touzi2017}.

When the $\mu_i$'s fail to satisfy the convex ordering, it is appropriate to consider a relaxation of \eqref{def.standard_MOT}. The violation of convex ordering is very common. For instance, even if $\bmu$ admits the convex ordering, the empirical measures induced by each marginals typically fail to satisfy the convex ordering. In order to model a relaxation of \eqref{def.standard_MOT}, we begin with introducing the notation $\mathcal{W}\left(  \mu,\nu \right)$ for
the Wasserstein distance between the (Borel) probability measures $\mu$ and
$\nu$ with support in $\mathbb{R}^n$. Precisely, if $\Pi\left(  \mu,\nu \right)  $
denotes the space of probability measures with support in $\mathbb{R}^n\times\RR^n$ and
with marginals $\mu$ and $v$, respectively then
\begin{equation*}
\mathcal{W}\left(  \mu,v\right)  =\inf\left\lbrace \int\|\bx-\mathbf{y}\|_1\pi\left(  dx,dy\right)  :\pi\in\Pi\left(  \mu,\nu \right)
\right\rbrace, \label{Wd}%
\end{equation*}
where $\|\mathbf{z}\|_1 = \sum_{i=1}^n|z_i|$ is the $l_1$ norm for $\mathbf{z} = (z_1,\cdots, z_n)$. The distance between $\pi\in\calP(\Omega)$ and $\pi(\bmu)$ is given by
\begin{equation*}
    d\left(\pi, \Pi(\bmu)\right) = \inf \left\lbrace \calW(\pi, \pi'): \pi'\in\Pi(\bmu) \right\rbrace.
\end{equation*}
For $\varepsilon>0$, this becomes the final ingredient that leads to the consideration of the $\varepsilon$-relaxed distributional uncertainty set of martingale measures
\begin{equation}\label{def.uncertainty_set}
\mathcal{M}_{\boldsymbol{\mu}}\left(  \varepsilon
\right)  =\left\lbrace\pi\in\mathcal{M}\left(\Omega\right)
:d\left(\pi, \Pi(\boldsymbol{\mu})\right) \leq\varepsilon\right\rbrace.
\end{equation}
Compare to the standard MOT formulation in \eqref{def.standard_MOT}, our formulation in \eqref{def.uncertainty_set} does not require the marginals of $\pi$ to exactly match $\bmu$, but they are close to each other based on the distance function $d(\cdot, \cdot)$. To be specific, for any $\pi\in\calM_{\bmu}(\varepsilon)$, we have 
\begin{equation*}
    \sum_{i=1}^d \calW\left(\mu_i, \pi_i\right)\leq\varepsilon. 
\end{equation*}
Hence, the marginal measures of $\pi\in\calM_{\bmu}(\varepsilon)$ lie within a prescribed size of uncertainty $\varepsilon$ of the given marginals $\mu_1,\cdots, \mu_d$. For notational simplicity, without further specifying the marginals, we will use $\calM(\varepsilon)$ to denote $\calM_{\bmu}(\varepsilon)$ throughout the paper. The uncertainty set $\calM(\varepsilon)$ in \eqref{def.uncertainty_set} is carefully designed in the sense that it enjoys nice geometry structure. It is straight forward to check that  $\Pi(\boldsymbol{\mu})$ is closed (under the topology that is induced by Wassertstein distance) and convex. Thus, by standard functional analysis, there exists a probability measure $\pi'\in\Pi(\boldsymbol{\mu})$, such that $d(\pi, \Pi(\boldsymbol{\mu})) = \calW(\pi, \pi')$. 



Now, we provide the $\varepsilon$-DRMOT formulation. Given a non-negative function $f$, we consider

\begin{equation}\label{eq.worstcase}
I\left(\varepsilon\right)
=\sup_{\pi\in\calM(\varepsilon)}\mathbb{E}_{\pi}\left[f\left( \bX\right)\right].
\end{equation}
Note that $I\left(0\right)  $
recovers the standard MOT formulation. Just as in the conventional MOT, $I\left(
\varepsilon\right)$ can be interpreted as the cost of super-hedging a path
dependent financial security that pays $f\left(  \mathbf{X}\right)$, where $\mathbf{X}$ is the underlying (stochastic) price process.

The assumption that $\varepsilon=0$ in the standard MOT problem is motivated
by a setting in which call options are known for every strike price and for
every maturity time $1\leq i\leq d$. By differentiating the call option
prices with respect to the strike, we can recover the marginal distributions
$\mu_{i}$ exactly. In practice, however, call options are not traded at every
strike. So, although $\varepsilon>0$ will often be small, it typically will
not be exactly equal to zero. In fact, in our main computational results, we
will be able to send $\varepsilon\rightarrow0$, so this relaxation and our
results can also be used to estimate $I\left(0\right)$.

There are alternative relaxations that can be handled with the methods that we
study. For example, one may consider a benchmark measure $\mu\in
\mathcal{P}\left(\Omega\right)  $, parametrically
calibrated based on both private and public information and then consider a
relaxation of the form $\mathcal{W}\left(  \mu,\pi\right)  \leq\varepsilon$. Our formulation has the advantage that the misspecification size only depends on marginal information, so it may be easier to calibrate.

\subsection{The main goal and our technical contributions}\label{sec.main_contribution}

Our main goal in this paper is to provide a discretization scheme for
computing $I\left(\varepsilon\right)
$ by obtaining independent empirical versions, $\widehat{\mu}_{i}^{(n)} $ of
the $\mu_{i}$'s (each of them constructed with $n$ samples)\ and using these
to compute $I\left(\varepsilon
\right)  $ within an error of order $O_P\left(  n^{-1/2}\right)  $ (i.e. with
canonical sample complexity). In order to achieve this goal, we obtain various results which form the core of our technical contributions. The roadmap towards our goal proceeds as follows.

A) We show a strong duality result (Theorem \ref{thm.strongdual}), which is crucial to develop approximation and structural results for $I\left(
\varepsilon\right)$. For example, as a consequence of Theorem \ref{thm.strongdual}, we show
that $I\left(\cdot\right)  $ is
concave and therefore continuous in its domain of finiteness.

B) Suppose we can observe i.i.d samples $X_{i}^{(1)},
X_{i}^{(2)}, \cdots$ from $\mu_{i}$ for each $1\leq i\leq d$. Let  $\widehat{\mu}_{i}^{(n)}$ denote the
empirical distribution function,
\begin{equation*}
\label{defn.empirical_distribution}\widehat{\mu}_{i}^{(n)}(dx):=\frac{1}%
{n}\sum_{j=1}^{n}\mathbbm{1}(X_{i}^{(j)}\in dx).
\end{equation*}
Let $\boldsymbol{\widehat\mu}^n$ denote $\left(\widehat{\mu}_{1}^{(n)},\cdots, \widehat{\mu}_{d}^{(n)}\right)$. Using part A), we show (in Theorem \ref{thm.distance_n}) that
\[
I\left(\varepsilon\right)
=I_n\left(\varepsilon\right) + O_P\left(  n^{-1/2}\right)  ,
\]
with explicit constants in the error terms. Here $I_n(\varepsilon)$ is defined by
\begin{equation*}
    I_n(\varepsilon) := \sup_{\pi\in\calM_n(\varepsilon)} \EE_{\pi}\left[f(\bX)\right],
\end{equation*}
where $\calM_n(\varepsilon):= \left\lbrace \pi\in\calM(\Omega): d\left(\pi, \Pi(\widehat\bmu^n)\right)\leq\varepsilon\right\rbrace $.

C) Building on A) and B), we provide a discretization scheme for computing
$I_{n}\left(\varepsilon\right)  $. Note that
this is necessary because, even if the $\widehat{\mu}_{i}^{\left(  n\right) }
$'s are discrete, the set $\calM_{n}\left(
\varepsilon\right)  $ is still infinite-dimensional. The discretization 
involves partitioning each set $\Omega_{i}$ in {$O$}$\left(
n^{1/2}\right)  $ equally sized intervals in order to maintain an error of
order {$O$}$\left(  n^{-1/2}\right)  $. This leads to a finite-dimensional linear program (LP) with {$O$}$\left(  n^{3d/2}\right)
$ variables, see Theorem \ref{thm.discretize_approximation}.

D) Finally, we generalize our model to non-compact marginals. We use a technique based on Skorokhod's embedding to show that the non-compact problem can essentially be reduced to the compact problem that we have discussed in A), B) and C). Hence we can prove parallel results to B) and C); see Theorem \ref{thm.strongdual_general}, Theorem \ref{thm.distance_n_general} and Theorem \ref{thm.non_compact_final_step}.

\subsection{Related literature}
As mentioned earlier, our problem formulation is closely related to the
standard MOT problem, which was introduced in
\cite{model_indep_bound_Beiglbock2013} in order to obtain model-free bounds
for option pricing and super-hedging; see also, for connections to the
Skorokhod embedding problem \cite{robustPriceBound_Hobson2015}, in addition to
\cite{completeDualMOT_Touzi2017}, for general duality results and insights on
the structure of the optimal solution, and \cite{statsSuperhedge_obloj2018}
for a statistical approach to robust hedging also based on the MOT.
Traditionally, these results involve two period martingales 
\cite{multi_MOT_Lim2016, MOT_generalDimension_Lim2019,MOT_beiglbock2016}, there are also results on multi-period martingales \cite{NUTZ_multiMOT_2020}. These contributions are given typically in the context of discrete martingales, but there are also results for continuous-time martingales \cite{MOT_continuous_Dolinsky2014, stochasticControl_MOT_galichon2014, OPTandSkorokhodEmbedding_Beiglbock2017}. Most of the literature in MOT focuses on studying strong duality or regularity properties of the dual optimal solutions \cite{MOT_beiglbock2016, completeDualMOT_Touzi2017,dualattain}, leading to challenging technical issues since the dual theory of MOT is fundamentally different from the standard optimal transport.

Despite the rich literature in the theory of MOT, computational methods for MOT are under exploration. The numerical implementation of MOT is challenging because the martingale coupling can be easily destroyed given empirical marginals. For the two-period MOT, numerical approximation with parametric rate of convergence have been studied in \cite{guo_computational_2019}. However, the theoretical guarantee of numerical methods for multi-period MOT has not been well studied. As far as we understand, in the multi-period setting, only consistency as sample size increases has been proved in \cite{guo_computational_2019, backhoffveraguas_stability_2019}. The work of  \cite{Obloj_robustPricingMultiple_2019} discusses a neural network approach to optimize the dual problem, although no rates of convergence are given. Most relaxations of the MOT problem involve relaxing the martingale constraints (e.g., \cite{guo_computational_2019}), we take a different modeling approach, relaxing instead the marginal constraints.

The main technical difficulty lies in the insufficient understanding of the dual variables in multi-period case. By instead relaxing the marginal constraint in the uncertainty set, implying that we do not need to match the empirical marginals exactly we are able to overcome some of the technical difficulties in the past literature and yet preserving the subtle structure imposed by the martingale constraint.

Our model formulation connects to the literature on distributionally robust optimization (DRO). Motivated by the seminal paper \cite{dupuis2000robust}, Robust optimization and robust control under model uncertainty has been extensively studied \cite{hansen2001robust, hansen2008robustness, DROwithMomentUncertainty_Ye2010} in the past two decades. The idea of modeling distributional robustness has also been applied to different areas, such as finance
\cite{pflug2014multistage} and statistics \cite{huber2004robust}.
In short, DRO is a class of games in which the optimizer chooses a decision against an adversary which chooses a probability model within a feasible set, known as the distributional uncertainty set. In recent years, non-parametric uncertainty sets based on optimal transport have been considered \cite{dataDrivenDRO_tractableFormulation_Kuhn2018, DROwithWasserstein_GAO_2016,DR_model_risk_Jose_2016, RWP_Jose2016,SGM_DRO_f-divergence_Duchi2016,
stats_RO_EL_Duchi2016}. It has also been shown that DRO estimators can recover a wide range of classical regularization estimators in statistics \cite{SOS_inference_Jose2016, DROgroupLasso_Jose2017} and machine learning \cite{duchi2016variance, LearningPerformance_DRO_Duchi2018, sinha2018certifying}. We believe that our results will find applications in DRO area as well. In particular, as mentioned earlier in the Introduction, in the setting of multi-stage distributionally robust stochastic programming.

\subsection{Organization}

The rest of the paper is organized as follows. In Section \ref{sect.DRMOT_compact}, we provide explicit statements of the technical results described in parts A)-C), where the marginals are assumed to have compact supports. In Section \ref{sect.DRMOT_non_compact}, we discuss how to generalize our formulation to allow non-compact marginals following the results in D).
The results are presented in Section \ref{sect.DRMOT_compact} and Section \ref{sect.DRMOT_non_compact} without their proofs; our goal is to provide a precise description of the roadmap outlined in contributions
A) to D). 

The complete proofs of the results corresponding to contributions A) to D) are
given in Section \ref{sect.proof_strong_dual} to Section \ref{sect.proofs_non_compact}, respectively. Each of these sections, in
turn, contain various technical results whose proofs are related to various appendices inside these sections.

\section{Distributionally Robust Martingale Optimal Transport: Compact Domains}\label{sect.DRMOT_compact}

Throughout this paper, we impose the following two assumptions.

\begin{assumption}
\label{aspn.convexorder} The $d$ marginal measures $\mu_{1}, \cdots, \mu_{d}$ satisfy the convex ordering denoted by $$\mu_1\preceq_c\cdots\preceq_c\mu_d.$$ Here, any two integrable probability measures $\mu$ and
$\nu$ are in convex ordering \emph{(}$\mu\preceq_{c}\nu$\emph{)} if and only if
\begin{equation*}
\int\phi(x)\mu(dx) \leq\int\phi(x)\nu(dx)
\end{equation*}
holds for arbitrary convex function $\phi$ with linear growth at infinity.
\end{assumption}
In view of Assumption \ref{aspn.convexorder}, by the standard results in \cite{strassen1965}, there exists a martingale measure in $\Pi(\bmu)$.
\begin{assumption}
\label{aspn.cost_function} The payoff function $f(\cdot)$ is
continuous on $\mathbb{R}^{d}$.
\end{assumption}
In this section, we further assume compact supports for each marginals. In section \ref{sect.DRMOT_non_compact}, we will discuss how to generalize our framework to the non-compact case. 

\begin{assumption}
\label{aspn.compact_support} $\Omega_{i}$ \emph{(}the support of $\mu_i$\emph{)} is a compact interval of $\mathbb{R}$
for each $1\leq i\leq d$, and $\Omega_{1}\subseteq\Omega_{2}\subseteq
\cdots\subseteq\Omega_{d}$.
\end{assumption}

\subsection{Strong Duality and Continuity in the Relaxation Parameter}
Before stating Theorem \ref{thm.strongdual}, we introduce the following notation system for simplicity. For a generic $d$-tuple $v=(v_1,\cdots, v_d)\in\RR^d$, let $v_{i:j}:=(v_i,\cdots, v_j)$ for $1\leq i\leq j\leq d$. For a generic product space $S_1\times\cdots\times S_m$, let $S_{i:j}:= S_i\times\cdots\times S_j$ for $1\leq i\leq j\leq d$. For a generic sequence of probability measures $(\varphi_1,\cdots, \varphi_d)$, let $\varphi_{i:j}:= (\varphi_i,\cdots,\varphi_j)$ for $1\leq i\leq j\leq d$. For any space $S\subset \RR^n$, let $C(S)$ be the set of continuous functions that supporting on $S$. 
\begin{theorem}
\label{thm.strongdual} Let Assumption \ref{aspn.compact_support} hold. Suppose that $\mathcal{M}_{\boldsymbol{\nu}}(\varepsilon
)\ne\varnothing$ for some  $\boldsymbol{\nu}\in\bigtimes_{i=1}^d\calP(\Omega_i)$ and $\varepsilon>0$, then the dual representation
of
\begin{equation*}
I_{\boldsymbol{\nu}}(\varepsilon) = \sup_{\pi\in\mathcal{M}%
_{\boldsymbol{\nu}}(\varepsilon)}\mathbb{E}_{
\pi}\left[  f(\mathbf{X})\right]
\end{equation*}
is:
\begin{align*}
J_{\boldsymbol{\nu}}(\varepsilon) := \emph{inf} \quad &
\gamma\varepsilon+ \sum_{k=1}^{d}\int_{\Omega_{k}}\beta_{k}(x_{k})\nu
_{k}(dx_{k})\\
\emph{subject to} \quad &  
\gamma\in\RR_{\geq0},\quad \boldsymbol{\alpha}\in\bigtimes_{i=1}^{d-1}C(\Omega_{1:i}),\quad \boldsymbol{\beta}\in\bigtimes_{i=1}^d C(\Omega_i);\\
&  H(\gamma,\boldsymbol{\alpha},\boldsymbol{\beta})(\mathbf{x},\mathbf{x}') \leq 0, \quad \emph{for all } \mathbf{x}, \mathbf{x'} \in \Omega.
\end{align*}
Here $\boldsymbol{\alpha}= (\alpha_1,\cdots,\alpha_{d-1})$, where $\alpha_{i}(\cdot)\in C\left(\Omega
_{1:i}\right)$. $\boldsymbol{\beta} = (\beta_1,\cdots,\beta_d)$, where $\beta_i(\cdot)\in C(\Omega_i)$. The function $H$ is defined by
\begin{equation*}
    H(\gamma,\boldsymbol{\alpha},\boldsymbol{\beta})(\mathbf{x},\mathbf{x}')= f(\mathbf{x}') + \sum_{k=1}^{d-1}\alpha_{k}
(\bx_{1:k}^{\prime})(x_{k+1}^{\prime}- x_{k}^{\prime
})-\gamma\|\bx - \bx'\|_1-\sum_{k=1}^{d}\beta_{k}(x_{k}),
\end{equation*}
where $\mathbf{x} = (x_1,\cdots, x_d), \mathbf{x}'=(x_1',\cdots, x_d')$.
Moreover, we have the following strong duality
\[
I_{\boldsymbol{\nu}}(\varepsilon) = J_{\boldsymbol{\nu}}(\varepsilon).
\]
and there exist a primal optimizer  $\pi^{\emph{DRO}}\in\mathcal{M}_{\boldsymbol{\nu}}(\varepsilon)$ for $I_{\boldsymbol{\nu}
}(\varepsilon)$.
\end{theorem}
We shall emphasize the connection of the above dual formulation to robust super-hedging. By setting $\bx' = \bx$, the constraints $H(\gamma, \balpha, \bbeta)(\bx, \bx)\leq 0$ recover the exact super-hedging description in \cite{model_indep_bound_Beiglbock2013}:
\begin{equation}\label{ieq.super_hedging}
    \sum_{k=1}^d \beta_k(x_k) - \sum_{k=1}^{d-1}\alpha_{k}
(\bx_{1:k})(x_{k+1}- x_{k}) \geq f(\bx), \textrm{ for all } \bx\in\Omega
\end{equation}
where we can construct a portfolio consisting of vanilla options $\beta_k(x_k)$ and the risky asset traded (via the self-financing strategy $\{\alpha_k(\bx_{1:k})\}_{k=1}^{d-1}$) over time, (the left-hand side of \eqref{ieq.super_hedging}) to super-replicate the payoff $f(\bx)$. In contrast to the standard MOT, we allow $\bx'\neq \bx$ in the dual constraints, which leads to new super-hedging strategy that is robust to model misspecification in the marginals. To be precise, we can interpret the dual constraints in Theorem \ref{thm.strongdual} by
\begin{equation*}
\sum_{k=1}^{d}\left[\beta_{k}(x_{k}) + \gamma |x_k' - x_k|\right] - \sum_{k=1}^{d-1}\alpha_{k}
(\bx_{1:k}^{\prime})(x_{k+1}^{\prime}- x_{k}^{\prime
})\geq  f(\mathbf{x}'), \textrm{ for all } \bx, \bx'\in\Omega. 
\end{equation*}
Hence, 
\begin{equation*}
    \sum_{k=1}^d \widetilde\beta(x_k') - \sum_{k=1}^{d-1}\alpha_k(\bx'_{1:k})(x_{k+1}' - x_k') \geq f(\bx'), \textrm{ for all } \bx'\in\Omega,
\end{equation*}
where $\widetilde\beta(x_k') := \inf_{x_k\in\Omega_k}\left\lbrace \beta_k(x_k) + \gamma|x_k' - x_k|\right\rbrace$, for $1\leq k\leq d$. Then, even the marginals are misspecified within the distributional uncertainty set in Theorem \ref{thm.strongdual}, we can still super-replicate the payoff $f(\bx')$ by a portfolio consisting of vanilla options $\widetilde\beta(x_k')$ and investments in the risky asset according to the self-financing strategy $\{\alpha_k(\bx_{1:k}^{\prime})\}_{k=1}^{d-1}$.

Below is a direct Corollary of Theorem \ref{thm.strongdual}, which will be useful in the proof of one of our main results, namely, Theorem \ref{thm.distance_n}.

\begin{corollary}
\label{cor.dual_reformulation} Let Assumption \ref{aspn.compact_support} hold. Suppose that $\mathcal{M}_{\boldsymbol{\nu}}(\varepsilon
)\ne\varnothing$ for some $\boldsymbol{\nu}\in\bigtimes_{i=1}^d \calP(\Omega_i)$ and  $\varepsilon>0$, we have 
\begin{equation}
I_{\boldsymbol{\nu}}(\varepsilon)=\inf_{\left(  \gamma, \boldsymbol{\alpha},   \boldsymbol{\beta} \right)  \in\mathcal{S}_{\boldsymbol{\nu}_{1:d-1}}}%
\gamma\varepsilon+\int_{\Omega_{d}}F\left(  x_d;\gamma,\boldsymbol{\alpha} ,\boldsymbol{\beta} \right)
\nu_{d}(dx_{d}), \label{eq.dual_reformulation}
\end{equation}
where 
\begin{equation*}
\begin{aligned} F\left(x_d; \gamma, \boldsymbol{\alpha} ,\boldsymbol{\beta}\right) : = \sup_{ \bx_{1:d-1}\in\Omega_{1:d-1},   \mathbf{x}'\in\Omega}\left\lbrace f\left(\mathbf{x}'\right)- \sum_{k=1}^{d-1}\beta_k(x_k)-\gamma\sum_{k=1}^{d}|x_k-x_k'| + \right. \\ \left. \sum_{k=1}^{d-1}\alpha_k(\mathbf{x}_{1:k}')(x_{k+1}'-x_k')\right\rbrace \end{aligned}
\end{equation*}
and
\begin{align*}
\mathcal{S}_{\boldsymbol{\nu}_{1:d-1}} &  :=\left\{  \left(
\gamma,\boldsymbol{\alpha} ,\boldsymbol{\beta}\right)  :\gamma\geq0,\alpha_{k}\in C(\Omega_{1:k}),\beta_{k}\in C(\Omega_{k}), \int_{\Omega_{k}}\beta_{k}(x)\nu_{k}(dx)=0,\forall 1\leq k\leq
d-1\right\}.
\end{align*}
\end{corollary}
The above Corollary is a direct consequence of Theorem \ref{thm.strongdual}, we sketch its proof as follows. 
\begin{proof}[Proof of Corollary \ref{cor.dual_reformulation}]
	Follow the proof of Theorem \ref{thm.strongdual}, we obtain the following result (see \eqref{strongdual_I>=J}):
	\begin{equation}\label{cor1.keyexpression}
	I_{\boldsymbol{\nu}}(\varepsilon) = \inf_{\mbox{$\begin{subarray}{c}\gamma \geq 0, \alpha_k\in C(\Omega_{1:k}), \\
			\beta_k\in C(\Omega_k), 1\leq k\leq d-1\end{subarray}$}} \gamma\varepsilon + \sum_{k=1}^{d-1}\int_{\Omega_k}\beta_{k}(x)\nu_k(dx) + \int_{\Omega_d}
	F\left(x_d; \gamma, \boldsymbol{\alpha}, \boldsymbol{\beta}\right)\nu_d(dx_d).
	\end{equation}
	By the definition of $F$, we observe that the objective function $I_{\boldsymbol{\nu}}(\varepsilon)$ does not change if we shift any $\beta_k(\cdot)$ to $\beta_k(\cdot) + \lambda_k$ by an arbitrary constant $\lambda_k$. Hence we can properly select constants $\lambda_k$ $(1\leq k\leq d-1)$ to enforce the constraints
	$
	\int_{\Omega_k}\beta_k(x)\nu_k(dx) = 0, \textrm{ for all }1\leq k\leq d-1.
	$
	Plugging the above constraints into \eqref{cor1.keyexpression}, we conclude the desired result.
\end{proof}

\begin{remark}\label{remark.symmetry}
By symmetry, we can replace $\nu_{d}$ in \eqref{eq.dual_reformulation} by any
other $\nu_{k}$ $(1\leq k\leq d-2)$. To be precise,
\begin{equation*}
I_{\boldsymbol{\nu}}(\varepsilon)=\inf_{\left(  \gamma, \boldsymbol{\alpha},   \boldsymbol{\beta} \right)  \in\mathcal{S}_{\boldsymbol{\nu}_{-k}}}%
\gamma\varepsilon+\int_{\Omega_{d}}F\left(x_k;\gamma,\boldsymbol{\alpha} ,\boldsymbol{\beta} \right)
\nu_{k}(dx_{k}),
\end{equation*}
where 
\begin{equation*}
\begin{aligned} F\left(x_k; \gamma, \boldsymbol{\alpha} ,\boldsymbol{\beta}\right) : = \sup_{ \bx_{-k}\in\Omega_{-k},   \mathbf{x}'\in\Omega}\left\lbrace f\left(\mathbf{x}'\right)- \sum_{1\leq j\leq d, j\neq k}\beta_j(x_j)-\gamma\sum_{j=1}^{d}|x_j-x_j'| + \right. \\ \left. \sum_{j=1}^{d-1}\alpha_j(\mathbf{x}_{1:j}')(x_{j+1}'-x_j')\right\rbrace, \end{aligned}
\end{equation*}
$\bx_{-k} := (x_1,\cdots, x_{k-1}, x_{k+1},\cdots, x_d)$, $\Omega_{-k}:= \bigtimes_{j\neq k}\Omega_j$ and
\begin{align*}
\mathcal{S}_{\boldsymbol{\nu}_{-k}} &  :=\bigg\{  \left(
\gamma,\boldsymbol{\alpha} ,\boldsymbol{\beta}\right)  :\gamma\geq0; \alpha_{j}\in C(\Omega_{1:j}) \emph{ for all } 1\leq j\leq d-1;\beta_{j}\in C(\Omega_{j}),\\
&\qquad\qquad\qquad\qquad\qquad\qquad\qquad\qquad\qquad\int_{\Omega_{j}}\beta_{j}(x)\nu_{j}(dx)=0,\emph{ for all } j\neq k \bigg\}.
\end{align*}
\end{remark}

Next we study the continuity in the relaxation parameter $\varepsilon$ in
$\varepsilon$-DRMOT. This result is relevant because it shows that we can use our relaxation as an asymptotically correct (as $\varepsilon \rightarrow 0$) approximation to the standard MOT problem.

\begin{prop}
\label{prop.cotinuity_epsilon} Given $\boldsymbol{\nu} \in \bigtimes_{i=1}^d\calP(\Omega_i)$, let $R_{\boldsymbol{\nu}} :=
\inf\{\varepsilon\geq 0 : \mathcal{M}_{\boldsymbol{\nu}}(\varepsilon)\neq\varnothing\}$. Then,  $I_{\boldsymbol{\nu}}(\varepsilon)$ is both concave and continuous as a function of $\varepsilon\in(R_{\boldsymbol{\nu}}, +\infty)$.
Moreover, if $\nu_{1}\preceq_{c}\cdots\preceq_{c}\nu_{d}$, $I_{\boldsymbol{\nu}}(\varepsilon)$ is a continuous function of
$\varepsilon\in[0, +\infty)$.
\end{prop}
The detailed proof of Proposition \ref{prop.cotinuity_epsilon} is deferred to Section \ref{sect.proof_consistency}. Finally, we investigate the feasibility of $I_{n}(\varepsilon)$. The proof is deferred to Section \ref{sect.proof_feasiblity}.

\begin{lemma}
\label{lemma.feasibility} Let Assumption \ref{aspn.convexorder} and \ref{aspn.compact_support} hold. Then, for any $\varepsilon>0$ and $\delta\in(0,1)$,
there exist universal constants $C_{1}, C_{2}>0$, such that for any $n\geq N(\varepsilon,\delta):=
\frac{d^{2}}{C_{2}}\log\left(\frac{d C_{1}}{\delta}\right)\varepsilon^{-2}$, 
$I_{n}(\varepsilon)$
is feasible with probability at least $1-\delta$.
\end{lemma}

\subsection{Sample Complexity}

In this section, we prove the claimed result in part B) of Section \ref{sec.main_contribution}, i.e., the convergence of $I_{n}(\varepsilon) $ to
$I(\varepsilon)$.

\begin{theorem}
\label{thm.distance_n} Let Assumption \ref{aspn.convexorder} and
\ref{aspn.compact_support} hold. Suppose that $|f|$ is bounded by $D$ on $\Omega$. Then, for any $\varepsilon>0$, $\delta\in(0,1)$ and
$n\geq N\left(  \frac{\varepsilon}{2},\frac{\delta}{2d}\right)  $ \emph{(}defined in
Lemma \ref{lemma.feasibility}\emph{)} , with probability at least $1-\delta$, we have
\begin{equation*}
\left|  I_{n}(\varepsilon) - I(\varepsilon)\right|  \le\frac{4dD}{\varepsilon}\sqrt
{\frac{\log\left(  2C_{1}d/\delta\right)  }{C_{2}n}},
\end{equation*}
where the universal constants $C_{1}, C_{2}>0$ are defined in Lemma
\ref{lemma.feasibility}.
\end{theorem}
The above concentration result shows that $I_n(\varepsilon)$, as an estimator to $I(\varepsilon)$, is asymptotically consistent. Specifically, the gap $|I_n(\varepsilon) - I(\varepsilon)|$ is of order $O\left(n^{-1/2}\right)$ with high probability. Moreover, the dependency of the gap on $d$, the number of steps, is $\widetilde{O}(d)$. Here, $a_n = \widetilde{O}(b_n)$ means $a_n = O(b_n)$ up to a poly-log factor of $b_n$.  
The detailed proof of Theorem \ref{thm.distance_n} is deferred to Section \ref{sect.proof_theorem_distance_n}.

Our result can be further generalized to the case that each marginal $\mu_i$ is $k$-dimensional probability measure (Hence, the underlying martingale is a $k$-dimensional, $d$ steps martingale). For each $1\leq i\leq d$, assume that $\Omega_i$ is a compact subset of $\RR^k$ ($k\geq2$), and $\mu_i\in\calP(\Omega_i)$. Moreover, The underlying probability measure in  $\Pi(\bmu)$ is now a martingale measure in $\mathbb{R}^{k\times d}$. The following result (parallel to Theorem \ref{thm.distance_n}) shows that $|I_n(\varepsilon) - I(\varepsilon)|$ is of order $O\left(n^{-1/k}\right)$ with high probability, and the gap also has $\widetilde{O}(d)$ dependency on $d$.

\begin{theorem}
\label{thm.general_distance_n} 
Given $k\geq 2$. For all $1\leq i\leq d$, let $\Omega_i$ be compact subset of $\RR^k$ and $\mu_i\in\calP(\Omega_i)$. Assume that $\mu_1\preceq_c\cdots\preceq_c\mu_d$, and $|f|$ is bounded by $D$ on $\Omega$. Then, for any $\varepsilon>0$, $\delta\in(0,1)$ and
$n\geq \frac{2^kd^k}{C_2'\varepsilon^k} \log\left(\frac{2d^2C_1'}{\delta}\right) $, with probability at least $1-\delta$, we have 
\begin{equation*}
\left|  I_{n}(\varepsilon) - I(\varepsilon)\right|  \leq\frac{4dD}{\varepsilon}\left(\frac{\log(2C_1'd/\delta)}{C_2'}\right)^{1/k}\frac{1}{n^{1/k}},
\end{equation*}
where $C_1', C_2'$ are universal constants.
\end{theorem}
The proof of Theorem \ref{thm.general_distance_n} is essentially the same as the proof of Theorem \ref{thm.distance_n}, while the only difference is we need to apply the Wasserstein concentration bound for probability measures in general dimensions. We sketch the proof of Theorem \ref{thm.general_distance_n} in Section \ref{sect.proof_theorem_distance_n}.

\subsection{Finite-Dimensional Linear Programming (LP) Reduction} \label{subsec.discretization}

Although in the formulation of $I_n(\varepsilon)$, we replace each $\mu_i$ by its empirical version $\widehat{\mu}^{(n)}_i$, the computation of $I_n(\varepsilon)$ is still not feasible. This fact can be viewed in different perspectives. In view of the primal formulation, we allow the candidate measures be any martingale measures that support on $\Omega$ with each marginals sufficiently close to the empirical marginals, so $I_n(\varepsilon)$ is by default an infinite-dimensional LP. From the perspective of dual formulation (Theorem \ref{thm.strongdual}), we  optimize the problem over functional space, which is also computationally intractable. Therefore, we need to discretize the support of the candidate measures in the Wasserstein ball, or equivalently, discretize the support of dual variables in the dual representation. In this section, we propose an uniform grid discretization, and prove the approximation error induced by such discretization.

We assume in this section that each $\Omega_{i}$ (the support of $\mu_{i}$) is
a closed interval $[a_{i}, b_{i}]$ in $\mathbb{R}$ with Lebesgue measure $l_{i}:=b_{i} -
a_{i}$. As a direct consequence of Assumption \ref{aspn.compact_support}, $l_1\leq\cdots\leq l_d$. For each
$i\in\left\lbrace 1,\cdots,d \right\rbrace $, denoted by $\Omega_{i}^{(N)}$,
the $N$-fold discretization of $\Omega_{i} = [a_{i}, b_{i}]$ is defined by%

\begin{equation}\label{eq.discrete_support}
\Omega_{i}^{(N)} = \left\lbrace a_{i} + \frac{kl_i}{N}: k=0,
1,\cdots,N\right\rbrace.
\end{equation}
Moreover, let $\Omega^{(N)}$ denote the Cartesian product $\bigtimes_{i=1}^d\Omega_i^{(N)}$. For technical convenience, we define the following relaxation of martingale measure.

\begin{definition}
($\delta$-martingale probability measure) A probability measure $\pi$ on
$\mathbb{R}^{d}$ is a $\delta$-martingale probability measure if
\[
\left|  \mathbb{E}_{\pi}\left[  X_{k+1}|\mathcal{F}_{k}\right]
- X_{k}\right|  \leq\delta, \quad\pi- \text{a.e. for } 1\leq k\leq d-1, 
\]
where $\mathcal{F}_{k}=\sigma(X_{1},\cdots,X_{k})$
is the canonical filtration. Moreover, we define $\calM_{\boldsymbol{\mu}}(\delta)$ to be the set of $\delta$-martingale probability measures that
contained in $\Pi(\boldsymbol{\mu})$.
\end{definition}

It is straight forward to see that, a $\delta$-martingale is close to a true martingale when $\delta$ is small. 
Let $\mathcal{M}\left( \Omega^{(N)};\delta\right)  $ be the
set of all the $\delta-$martingale measures that supported on $\Omega^{(N)}$, now we define
\begin{equation*}
I_{n,N}(\varepsilon,\delta) := \sup
_{\pi\in\calM_{n, N}(\varepsilon,\delta
)}\mathbb{E}_{\pi}\left[  f(\bX)\right],
\end{equation*}
where 
\begin{equation*}
\mathcal{M}_{n, N}( \varepsilon, \delta) :=
\left\lbrace \pi\in\mathcal{M}\left(\Omega^{(N)}
;\delta\right)  :
d\left(\pi, \Pi\left(\boldsymbol{\widehat\mu}^{n}\right)\right)\leq\varepsilon\right\rbrace.
\end{equation*}
It is easy to see that $I_{n,N}(\varepsilon, \delta) $
is computable. Precisely, we can solve this problem via its dual formulation. Suppose that we have samples $X_i^{(1)},\cdots, X_i^{(n)}\sim\mu_i$ for $1\leq i\leq d$. Then, 
\begin{align*}
    I_{n,N}(\varepsilon, \delta) = \text{inf} \quad &
\gamma\varepsilon+ \frac{1}{n}\sum_{i=1}^{d}\sum_{j=1}^n\beta_{i,j}\\
\text{subject to} \quad &  
\gamma\in\RR_{\geq0},\quad \alpha_{k_1,\cdots,k_i}\in\RR,\quad \beta_{i,j}\in\RR \textrm{ for all } 1\leq j\leq n, 0\leq k_i\leq N, 1\leq i\leq d;\\
&  f\left(a_1+\frac{k_1l_1}{N},\cdots, a_d+\frac{k_dl_d}{N}\right) +\sum_{i=1}^{d-1}\alpha_{k_1,\cdots,k_i}\left(a_{i+1}-a_i+\frac{k_{i+1}l_{i+1}-k_il_i}{N}\right)\\
& +\sum_{i=1}^{d-1}\delta|\alpha_{k_1,\cdots, k_i}|-\gamma\sum_{i=1}^d\left|X_i^{(j_i)} - \left(a_i + \frac{k_il_i}{N}\right)\right|-\sum_{i=1}^d\beta_{i, j_i}\leq 0,\\
&\textrm{where } 1\leq j_i\leq n, 0\leq k_i\leq N, 1\leq i\leq d.
\end{align*}
Thus, $I_{n,N}(\varepsilon, \delta) $ is a finite-dimensional LP, which can be solved by generic LP Algorithm. The following lemma allows us to construct a discrete 
$\delta$-martingale measure that arbitrarily close to a given martingale measure under Wasserstein distance, the proof is
deferred to Section \ref{sect.proof_discretize_lemma}.

\begin{lemma}
\label{lemma.discretize} Let Assumption \ref{aspn.compact_support} hold. Suppose that $\Omega_i = [a_i, b_i]$ and $l_i = b_i - a_i$ for all $1\leq i\leq d$. Let $l:=\max_{1\leq i\leq d}l_i = l_d$. Then, for any 
$\varepsilon>0$, any martingale probability measure $\pi$ on
$\Omega$ and integer $N$, there exists a
$l/N$-martingale probability measure $\pi^{(N)}$, such that

\begin{enumerate}
\item[\emph{(1)}] $\pi^{(N)}$'s $i$-th marginal $\pi^{(N)}_{i}$ is supported on $\Omega_{i}^{(N)}$ \emph{(}defined in \eqref{eq.discrete_support}\emph{)} for all
$1\leq i\leq d$;

\item[\emph{(2)}] $\mathcal{W}\left(  \pi,\pi^{(N)}\right)  \leq
\frac{l\sqrt{d}}{N}$, i.e., $\pi$ and $\pi^{(N)}$ are close under
Wasserstein distance.
\end{enumerate}
\end{lemma}

Before stating our main result in this section, we impose the following strict dispersion assumption which is slightly stronger than Assumption \ref{aspn.disperse}. This is an natural assumption following Assumption \ref{aspn.convexorder} and \ref{aspn.compact_support}, because the existence of martingale measure automatically enforced the support of each marginal measure to be enlarged from the intial step to the last step.

\begin{assumption}
\label{aspn.disperse}
\[
\tau:= \min_{1\leq i\leq d-1}\left\lbrace a_{i} - a_{i+1}\right\rbrace
\wedge\min_{1\leq i\leq d-1}\left\lbrace b_{i+1} - b_{i}\right\rbrace >0,
\]
i.e, the support of marginals become more and more dispersed over time.
\end{assumption}
Next we show that, under mild condition on the payoff function $f$, we can select $\varepsilon_N$ and $\delta_N$ properly, such that $I_{n,N}(\varepsilon_N,\delta_N)$ can be arbitrarily close to $I_n(\varepsilon)$ as $N$ increase.

\begin{theorem}
\label{thm.discretize_approximation} Suppose that $\Omega_i = [a_i, b_i]$ and $l_i = b_i - a_i$ for all $1\leq i\leq d$. Let $l:=\max_{1\leq i\leq d}l_i = l_d$. If $f$ is $L$-Lipschitz which absolute value is bounded by $D$ on $\Omega$, and Assumption \ref{aspn.disperse} hold. Then, for any $\varepsilon>0$, $\delta\in(0,1)$,
with probability at least $1-\delta$,  we have
\begin{equation}
\label{eq.discretize_approximation}\left|  I_n(
\varepsilon) - I_{n,N}\left(\varepsilon+
\frac{ld^{1/2}}{N}, \frac{l}{N}\right)  \right|
\leq\frac{c}{\varepsilon}\cdot\frac{1}{N}
\end{equation}
holds when $n\geq N(\varepsilon,\delta) = \frac{d^{2}}{C_{2}}\log\left(
\frac{dC_{1}}{\delta}\right)  \varepsilon^{-2}$, where the universal constants $C_{1}, C_{2}>0$ are defined in Lemma
\ref{lemma.feasibility}, the constant $c$ depends on $L, D, d, l$ and $\tau$.
\end{theorem}
The proof of Theorem \ref{thm.discretize_approximation} is deferred to Section \ref{sect.proof_thm_discrete_approx}. Finally, by combining Theorem \ref{thm.distance_n} and \ref{thm.discretize_approximation}, we can use $I_{n, N}(\varepsilon,\delta)$ to approximate $I(\varepsilon)$. To be clear, for $n$ sufficiently large, if $N = n^{1/2}$, with high probability we have 
\begin{equation*}
\left|  I_{n,N}\left(\varepsilon+
\frac{ld^{1/2}}{N}, \frac{l}{N}\right)  - I(\varepsilon)\right|  \lesssim \frac{1}{\varepsilon n^{1/2}}.
\end{equation*}
Here the notation $a_n\lesssim b_n$ means $a_n \leq C b_n$ for some positive constant $C$ the does not depend on the sample size $n$.

\section{Distributionally Robust Martingale Optimal Transport: Non Compact Domains}\label{sect.DRMOT_non_compact}

In this section, we formulate and solve the $\varepsilon$-DRMOT when the marginal distribution may not have compact support. Consider $d$ probability measures $\mu_1,\cdots,\mu_d\in\calP(\RR)$, in addition to Assumption \ref{aspn.convexorder}, we assume further that:
\begin{assumption}\label{aspn.moment}
There exists a constant $\gamma>0$, such that $\sup_{1\leq i\leq d}\EE_{\mu_i}[e^{\gamma|X|^2}] < \infty$. Moreover, $\EE_{\mu_i}X=0$ for all $1\leq i\leq d$. 
\end{assumption}
We remark that the square-exponential moment condition can be achieved by a wide range of probability measures (e.g., Sub-Gaussian measure). Meanwhile, the mean zero assumption is mostly for technical convenience. 

For each $\pi\in\calP(\RR^d)$, let $(X_1,\cdots,X_d)$ be the random process which the underlying law is $\pi$.  Be consistent with the previous discussion, let $\pi_i$ denote the $i$-th marginal distribution of $\pi$, i.e., $X_i\sim\pi_i$ ($1\leq i\leq d$). In particular, let $\calM_0(\RR^d)$ denote the set of martingale measures on $\RR^d$ that centered at zero (i.e., expectation of each marginals are zero). 

Now, we consider a set of distribution $\calMM_{\boldsymbol{\mu}}(\varepsilon)$ which is defined by
\begin{equation}\label{eq.uncertaintyset}
\calMM_{\boldsymbol{\mu}}(\varepsilon) := \left\lbrace \pi\in\calM_0(\RR^d):  d\left(\pi, \Pi(\boldsymbol{\mu})\right)\leq\varepsilon \right\rbrace,
\end{equation}
where $\varepsilon >0$ is prefixed. Compare to the definition in \eqref{def.uncertainty_set}, we allow the candidate measure $\pi\in\calMM_{\bmu}(\varepsilon)$ to support on $\RR^d$. For notational convenience, without further specify the marginals, we use $\calMM(\varepsilon)$ to denote $\calMM_{\bmu}(\varepsilon)$ throughout the rest of the paper. Given the above formulation, we aim to compute the following distributionally robust worst case expectation, which is a slight modification of the problem \eqref{eq.worstcase}:
\begin{equation}\label{eq.worstcase_generalSupport}
\bI(\varepsilon) := 
\sup_{\pi\in\calMM(\varepsilon)}\EE_{\pi}\left[f(\bX)\right],
\end{equation}
Note that \eqref{eq.worstcase_generalSupport} is natural extension to \eqref{eq.worstcase}. To estimate the worst case expectation in \eqref{eq.worstcase_generalSupport}, the key step is to find a good compactification of  $\calMM(\varepsilon)$, then we can leverage the tools that we developed in Section \ref{sect.DRMOT_compact}, where the support of each marginals are assumed to be compact. For any compact intervals $\Lambda_i$ $(1\leq i\leq d)$, let $\Lambda$ denote $ \bigtimes_{i=1}^d\Lambda_i$. Then, we define
\begin{equation*}
\calMM^{\Lambda}(\varepsilon):= \calMM(\varepsilon)\cap\calP(\Lambda),
\end{equation*}
The corresponding distributionally robust worst case expectation is 
\begin{equation}\label{eq.worstcase_compactSupport}
\bI^{\Lambda}(\varepsilon) := 
\sup_{\pi\in\calMM^{\Lambda}(\varepsilon)}\EE_{ \pi}\left[f(\bX)\right].
\end{equation}
We will show that the problem in \eqref{eq.worstcase_generalSupport} can be approximate by the problem in  \eqref{eq.worstcase_compactSupport}, by a proper choice of compact domain $\Lambda$. Hence, we only need to consider the case that each marginal of the measures in the uncertainty set has compact support. Such approximation allows us to leverage the theory that we built in Section \ref{sect.DRMOT_compact}.
In order to connect  \eqref{eq.worstcase_generalSupport} and \eqref{eq.worstcase_compactSupport}, we introduce the following lemma.
\begin{lemma}\label{lemma.skorokhod_embedding}
Let $X$ be a random variable satisfying $\EE X = 0$, $\EE |X| \geq c >0$ and  $\EE \left[e^{t|X|^2}\right] \leq C$ for some $t >0$ and $C>0$. By Skorokhod's embedding, there exists a stopping time $T$ such that $B_{T}\stackrel{d}{=}X$, where $(B_t)_{t\geq 0}$ is a standard Brownian motion. Under the same Brownian motion, given $\delta>0$, there exists a stopping time $T'$ and $k:= k(\delta, t, c, C)>0$, such that $|B_{T'}|\leq |B_{T}|\wedge k$ and $\EE|B_T - B_{T'}|\leq \delta$.
\end{lemma}
By Lemma \ref{lemma.skorokhod_embedding}, for any random variable $X$, we are able to construct a new bounded random variable $X'$, such that both $X$ and $X'$ can be embedded to a same Brownian motion, and their $L_1$ distance can be arbitrarily small. 

Observe that if for any $\pi\in\calMM(\varepsilon)$, we can find a $\widehat\pi\in\calMM^{\Lambda}(\varepsilon')$ (for some compact domain $\Lambda$ and $\varepsilon'\approx\varepsilon$), such that $\pi$ and $\widehat\pi$ are close under Wasserstein distance. Then, the object value in \eqref{eq.worstcase_generalSupport} and \eqref{eq.worstcase_compactSupport} should be close to each other. Now, for any martingale $(X_1,\cdots, X_d)\sim\pi$, we can first embed this martingale to a Brownian motion via an increasing sequence of stopping times. Next,  we use the Lemma \ref{lemma.skorokhod_embedding} to construct another increasing sequence of stopping times, such that the resulting process under the same Brownian motion has compact support, and its underling measure is close to $\pi$ under Wasserstein distance. Following this, we obtain the next lemma.

\begin{lemma}\label{lemma.reduce_to_compact}
Let Assumption \ref{aspn.convexorder} and \ref{aspn.moment} hold. Given $\varepsilon > 0$,  there exists $C'$ \emph{(}depends on $\bmu$ and $\gamma$\emph{)} so that for each $\delta>0$, the compact domain $\Lambda^{\delta} :=\bigtimes_{i=1}^d \left[-iC'\sqrt{\log(1/\delta)}, iC'\sqrt{\log(1/\delta)}\right]$ has the following property: 
For any $\pi\in\calMM(\varepsilon)$, there exists a probability measure $\widehat{\pi}\in \calMM^{\Lambda^{\delta}}(\varepsilon+\delta)$ satisfying  $\calW\left(\pi,\widehat{\pi}\right)\leq\delta$.
\end{lemma}
The proof of Lemma \ref{lemma.skorokhod_embedding} and \ref{lemma.reduce_to_compact} are deferred to Section \ref{proof.skorokhod_embedding}. As a direct consequence of Lemma \ref{lemma.reduce_to_compact}, we get the following feasibility result. 
\begin{corollary}\label{cor.feasibility_reduce_to_compact}
Let Assumptions \ref{aspn.convexorder} and \ref{aspn.moment} hold. Then, for any $\varepsilon > 0$, we have $\bI^{\Lambda^{\delta}}(\varepsilon)$ is feasible for $\delta\in(0,\varepsilon)$, where $\Lambda^{\delta}$ is constructed in Lemma \ref{lemma.reduce_to_compact}. 
\end{corollary}
\begin{proof}[Proof of Corollary \ref{cor.feasibility_reduce_to_compact}]
Based on Assumptions \ref{aspn.convexorder} and \ref{aspn.moment}, there exists a martingale measure $\pi\in\calMM(0)$. Thus, for any $\delta \in (0, \varepsilon)$, we can apply Lemma \ref{lemma.reduce_to_compact} to construct a compact domain $\Lambda^{\delta}$, such that there exists a probability measure $\pi'\in \calMM^{\Lambda^{\delta}}(\delta)\subset \calMM^{\Lambda^{\delta}}(\varepsilon)$.
\end{proof}

Similar to Theorem \ref{thm.strongdual}, we establish the following strong duality result for the optimization problem  \eqref{eq.worstcase_compactSupport}.

\begin{theorem}\label{thm.strongdual_general} 
Let $\Lambda_1,\cdots, \Lambda_d$ be compact subsets of $\RR$, and  $\boldsymbol{\nu}\in\bigtimes_{i=1}^d\calP(\RR)$ satisfies Assumption \ref{aspn.moment}. Suppose that $\calMM_{\boldsymbol{\nu}}^{\Lambda}(\varepsilon)\ne\varnothing$ for some $\varepsilon>0$. Then, the dual representation of 
\begin{equation*}
\bI_{\boldsymbol{\nu}}^{\Lambda}(\varepsilon) := 
\sup_{\pi\in\calMM_{\boldsymbol{\nu}}^{\Lambda}(\varepsilon)}\EE_{\pi}\left[f(\mathbf{X})\right]
\end{equation*}	
is the following:
\begin{align*}
\bJ_{\boldsymbol{\nu}}^{\Lambda}(\varepsilon)  := \emph{inf} \quad &  \gamma\varepsilon +  \sum_{k=1}^d\int_{\RR}\beta_k(x_k)\nu_k(dx_k)\\
\emph{subject to} \quad & \gamma \in\RR_{\geq 0},\quad \eta\in\RR, \quad \boldsymbol{\alpha}\in\bigtimes_{i=1}^{d-1}C(\Lambda_{1:i}),\quad \boldsymbol{\beta}\in\bigtimes_{i=1}^d C_b(\RR); \\
& \widetilde{H}(\gamma, \eta, \boldsymbol{\alpha}, \boldsymbol{\beta})(\mathbf{x}, \mathbf{x}')\leq 0, \quad \emph{for all } \mathbf{x}\in\RR^d, \textrm{ }\mathbf{x}' \in \Lambda=\bigtimes_{i=1}^d\Lambda_i.
\end{align*} 
Here $\boldsymbol{\alpha}= (\alpha_1,\cdots,\alpha_{d-1})$, where $\alpha_{i}(\cdot)\in C\left(\Lambda
_{1:i}\right)$ for $1\leq i\leq d$. $\boldsymbol{\beta} = (\beta_1,\cdots,\beta_d)$, where $\beta_i(\cdot)\in C_b(\RR)$ for $1\leq i\leq d$. The function $\widetilde H$ is defined by
\begin{align*}
    \widetilde H(\gamma, \eta, \boldsymbol{\alpha},\boldsymbol{\beta})(\mathbf{x},\mathbf{x}')= &f(\mathbf{x}') + \sum_{k=1}^{d-1}\alpha_k(\mathbf{x}_{1:k}')(x_{k+1}' - x_k')-\gamma\|\bx - \bx'\|_1 - \sum_{k=1}^{d}\beta_k(x_k)  + \eta x_1',
\end{align*}
where $\mathbf{x} = (x_1,\cdots, x_d), \mathbf{x}'=(x_1',\cdots, x_d')$.
Moreover, we have the strong duality holds
\[
\bI_{\boldsymbol{\nu}}^{\Lambda}(\varepsilon) = \bJ_{\boldsymbol{\nu}}^{\Lambda}(\varepsilon).
\]
and there exist a primal optimizer $\pi^{\emph{DRO}}\in\calMM_{\boldsymbol{\nu}}^{\Lambda}(\varepsilon)$ for $\bI_{\boldsymbol{\nu}}^{\Lambda}(\varepsilon)$. 
\end{theorem}
Since we are restricting the support of $\pi\in\calMM_{\bnu}(\varepsilon)$ to a compact domain $\Lambda$, the proof of Theorem \ref{thm.strongdual_general} is essentially the same as what we have in the proof of Theorem \ref{thm.strongdual}. Details are deferred to Section \ref{sect.strong_dual_general_proof}. Moreover, we also obtain the following corollary that is parallel to Corollary. The proof of Corollary \ref{cor.dual_representation_general} is omitted, as it simply repeats the arguments in the proof of Corollary \ref{cor.dual_reformulation}.

\begin{corollary}\label{cor.dual_representation_general}
Let $\Lambda_1,\cdots, \Lambda_d$ be compact subsets of $\RR$ and $\Lambda = \bigtimes_{i=1}^d \Lambda_i$. Let  $\boldsymbol{\nu}\in\bigtimes_{i=1}^d\calP(\RR)$ satisfies Assumption \ref{aspn.moment}, and  $\calMM_{\boldsymbol{\nu}}^{\Lambda}(\varepsilon)\ne\varnothing$ for some $\varepsilon > 0$. Then, 
\begin{align*}
&\bI_{\boldsymbol{\nu}}^{\Lambda}(\varepsilon) = 
\inf_{\left(\gamma, \eta,  \boldsymbol{\alpha}, \boldsymbol{\beta}\right)\in\calS^{\Lambda}_{\boldsymbol{\nu}_{1:d-1}}}\gamma\varepsilon + \int_{\RR}
F\left(x_d; \gamma,\eta, \boldsymbol{\alpha}, \boldsymbol{\beta}\right)\nu_d(dx_d),
\end{align*}	
where 
\begin{equation*}
\begin{aligned}
F\left(x_d; \gamma, \eta, \boldsymbol{\alpha}, \boldsymbol{\beta}\right) : = \sup_{\mbox{$\begin{subarray}{c} \bx_{1:d-1}\in\RR^{d-1},\mathbf{x}'\in\Lambda_{1:d}\end{subarray}$}}\left\lbrace f\left(\mathbf{x}'\right)- \sum_{k=1}^{d-1}\beta_k(x_k)-\gamma\sum_{k=1}^{d}|x_k-x_k'| + \right. \\ 
\left.  \sum_{k=1}^{d-1}\alpha_k(\mathbf{x}'_{1:k})(x_{k+1}'-x_k') + \eta x_1'\right\rbrace,
\end{aligned}
\end{equation*}
and 
\begin{align*}
S^{\Lambda}_{\boldsymbol{\nu}_{1:d-1}} := \left\lbrace \left(\gamma, \eta,  \boldsymbol{\alpha}, \boldsymbol{\beta}\right) : \gamma\geq 0, \eta\in\RR, \alpha_k\in C(\Lambda_{1:k}), \beta_k\in C_b(\RR), \right.\\
\left. \int_{\RR}\beta_k(x)\nu_k(dx) = 0, \emph{ for all }1\leq k\leq d-1 \right\rbrace.
\end{align*} 
\end{corollary}
As a consequence of Corollary \ref{cor.dual_representation_general}, we can prove the Lipschitz property of the map $\varepsilon \rightarrow I^{\Lambda}(\varepsilon)$ under suitable conditions. Details are deferred to Section \ref{sect.distance_n_general}.

\begin{prop}\label{prop.Lip_general_support}
Let Assumption \ref{aspn.cost_function} hold. let $\Lambda = \bigtimes_{i=1}^d\Lambda_i$, where $\Lambda_1,\cdots, \Lambda_d$ are compact intervals in $\RR$. Define $R = \inf\left\lbrace r\geq 0: \calMM^{\Lambda}(r) \neq \varnothing \right\rbrace$. Then, the map \begin{align*}
\left(R, +\infty\right) &\longrightarrow \RR\\
\varepsilon &\longmapsto \bI^{\Lambda}(\varepsilon)
\end{align*}
is $\frac{\sup_{\bx\in\Lambda}f(\bx)}{R}$-Lipschitz.
\end{prop}
By Lemma \ref{lemma.reduce_to_compact} and Proposition \ref{prop.Lip_general_support}, suppose that $f$ is $L$-Lipschitz, for $\delta < \varepsilon$, we have 
\begin{align*}
\bI(\varepsilon) &\leq L\delta + \bI^{\Lambda^{\delta}}(\varepsilon+\delta)\leq L\delta + \frac{L\max_{\bx\in\Lambda^{\delta}}|f(\bX)|}{\varepsilon}\delta +  \bI^{\Lambda^{\delta}}(\varepsilon) \leq L\left(1 + \frac{\max_{\bx\in\Lambda^{\delta}}|f(\bx)|}{\varepsilon}\right)\delta + \bI(\varepsilon),
\end{align*}
which gives
\begin{align}
\left|\bI(\varepsilon) - \bI^{\Lambda^{\delta}}(\varepsilon)\right|&\leq L\left(1 + \frac{\sup_{\bx\in\Lambda^{\delta}}|f(\bX)|}{\varepsilon}\right)\delta\leq L\left(1 + \frac{|f(\mathbf{0})| + 2Ld^2C'\sqrt{\log(1/\delta)}}{\varepsilon}\right)\delta. \label{ieq.I_IOmega_delta_bound}
\end{align}
Similar to Section \ref{sect.DRMOT_compact}, we define the empirical version of $\bI^{\Lambda}(\varepsilon)$ in \eqref{eq.worstcase_compactSupport} by
\begin{equation*}
    \bI^{\Lambda}_n(\varepsilon) := \sup_{\pi\in\calMM^{\Lambda}_n(\varepsilon)} \EE_{\pi} \left[f(\bX)\right],
\end{equation*}
where $\calMM^{\Lambda}_n(\varepsilon) = \left\lbrace \pi\in\calM_0(\Lambda): d\left(\pi, \Pi\left(\widehat\bmu^n\right)\right)\leq\varepsilon \right\rbrace $. The following lemma investigates the feasibility of $\bI^{\Lambda^{\delta}}_n(\varepsilon)$. The proof is similar to Lemma \ref{lemma.feasibility}, details can be found in Section \ref{sect.distance_n_general}.

\begin{lemma}\label{lemma.feasibility_non_compact}
Let Assumptions \ref{aspn.convexorder} and
\ref{aspn.moment} hold. Then, for any $\varepsilon>0$ and $\delta\in(0,\varepsilon/2)$,
there exist universal constants $C_{3}, C_{4}>0$ such that $\bI^{\Lambda^{\delta}}_{n}(\varepsilon)$
is feasible with probability at least $1-\delta'$ when $n\geq N'(\varepsilon,\delta'):= \frac{4d^{2}}{C_{4}}\log\left(\frac{dC_{3}}{\delta'}\right)  \varepsilon^{-2}$.
\end{lemma}
Now, we are able to provide the sample complexity result for $\bI^{\Lambda^{\delta}}_n$, which can be viewed as a parallel version of Theorem \ref{thm.distance_n}. Please find the complete proof in Section \ref{sect.distance_n_general}
\begin{theorem}\label{thm.distance_n_general}
Let Assumptions \ref{aspn.convexorder} and \ref{aspn.moment} hold. Suppose that $f$ is an $L$-Lipschitz function. Then, given $\varepsilon>0$ and  $\delta<\frac{\varepsilon}{4}$, for any $\delta'\in(0,1)$ and $n\geq N'\left(\frac{\varepsilon}{2},\frac{\delta'}{2d}\right)$ \emph{(}defined in Lemma \ref{lemma.feasibility_non_compact}\emph{)} , with probability at least $1-\delta'$ we have
\begin{equation*}
\left|\bI_n^{\Lambda^{\delta}}(\varepsilon) - \bI^{\Lambda^{\delta}}(\varepsilon)\right|\le \frac{dB}{\varepsilon}\sqrt{\frac{\log\left(2C_3d/\delta'\right)}{C_4n}},
\end{equation*}
where the universal constants $C_3, C_4>0$ are defined in Lemma \ref{lemma.feasibility_non_compact}, and $B$ is a constant depending on $L$ and $\bmu$.
\end{theorem}

Next, we approximate $\bI_n^{\Lambda^{\delta}}(\varepsilon)$ by proper discretization. Recalled that $\Lambda^{\delta} = \bigtimes_{i=1}^d\Lambda^{\delta}_i$, where $\Lambda^{\delta}_i = [-iC'\sqrt{\log(1/\delta)}, iC'\sqrt{\log(1/\delta)}]$ and $C'$ is a constant depends on $\bmu$ and $\gamma$ (defined in Assumption \ref{aspn.moment}). Similar to Section \ref{subsec.discretization}, the $N$-fold discretization of $\Lambda_i^{\delta}$ is defined by 
\begin{equation*}
\Lambda_{i}^{\delta, N} := \left\lbrace -iC'\sqrt{\log(1/\delta)} + k\cdot\frac{2iC'\sqrt{\log(1/\delta)}}{N}: k=0,1,\cdots,N \right\rbrace.
\end{equation*}
Let $\Lambda^{\delta,N}$ be the Cartesian product $\bigtimes_{i=1}^d\Lambda_i^{\delta, N}$. Denoted by  
$\calM_0\left(\Lambda^{\delta, N};\tau\right)$  the set of $\tau-$martingale measures that supported on  $\Lambda^{\delta, N}$, and each marginal has expectation zero. We introduce the  following finite-dimensional LP:
\begin{equation*}\label{def.finite_LP_general}
\bI_{n,N}^{\Lambda^{\delta}}(\varepsilon, \tau)  := 
\sup_{\pi\in\calMM^{\Lambda^{\delta}}_{n, N}(\varepsilon,\tau)}\EE_{\pi}\left[f(\bX)\right],
\end{equation*}
where the uncertainty set is defined by 
\begin{equation*}
\calMM^{\Lambda^{\delta}}_{n, N}( \varepsilon, 
\tau) := \left\lbrace \pi\in\calM_0\left(\Lambda^{\delta, N};\tau\right):  d\left(\pi, \Pi\left(\boldsymbol{\widehat\mu}^{(n)}\right)\right)\leq\varepsilon\right\rbrace.
\end{equation*}
The following result illustrates how to use  $\bI_{n,N}^{\Lambda^{\delta}}(\varepsilon, \tau)$ to approximate  $\bI^{\lambda^{\delta}}_n(\varepsilon)$. 
\begin{theorem}\label{thm.non_compact_final_step}
Let Assumption \ref{aspn.convexorder} and \ref{aspn.moment} hold. Let $\Lambda^{\delta} = \bigtimes_{i=1}^d[-iC'\sqrt{\log(1/\delta)}, iC'\sqrt{\log(1/\delta)}]$, where $C'$ is a constant depends on $\bmu$ and $\gamma$. Suppose that $f$ is a $L$-Lipschitz function. Then, given $\varepsilon >0$, $\delta'\in(0,1)$,  with probability at least $1-\delta'$,
\begin{equation*}
\left|\bI_n^{\Lambda^{\delta}}(\varepsilon) - \bI^{\Lambda^{\delta}}_{n, N}\left(\varepsilon + \frac{2d^{3/2}C'\sqrt{\log(1/\delta)}}{N},  \frac{2dC'\sqrt{\log(1/\delta)}}{N}\right)\right|\leq\frac{c\sqrt{\log(1/\delta)}}{\varepsilon}\cdot\frac{1}{N}
\end{equation*}
holds when $n\geq N'(\varepsilon,\delta')$ \emph{(}defined in Lemma \ref{lemma.feasibility_non_compact}\emph{)}. Here $c$ is a constant depending only on $\bmu$, $L$, $\gamma$, and $d$. 
\end{theorem}
Proof of Theorem \ref{thm.non_compact_final_step} is deferred to Section \ref{sect.proof_discretization_approximation_general}. Finally, by taking $\delta = n^{-1/2}$ and $N = n^{1/2}$, combining the results in Theorem \ref{thm.distance_n_general}, Theorem \ref{thm.non_compact_final_step} and \eqref{ieq.I_IOmega_delta_bound}, we have 
\begin{equation*}
   \left| \bI^{\Lambda^{\delta}}_{n, N}\left(\varepsilon + \frac{2d^{3/2}C'\sqrt{\log(1/\delta)}}{N},  \frac{2dC'\sqrt{\log(1/\delta)}}{N}\right) - \bI(\varepsilon)\right|\lesssim \frac{\sqrt{\log(n)}}{\varepsilon n^{1/2}}.
\end{equation*}

\section{Acknowledgement}
Support is acknowledged from NSF grants 2118199, 1915967, 1820942, 1838576 and AFOSR MURI 19RT1056 and the China Merchants Bank. 


\section{Proof of Theorem \ref{thm.strongdual}}\label{sect.proof_strong_dual}

We invoke the following standard min-max theorem of decision theory \cite{sion1958general} to prove
the Theorem \ref{thm.strongdual}.

\begin{lemma}
[Sion's minimax theorem, \cite{sion1958general}]\label{sion} Let $X$ be a
compact convex subset of a linear topological space and $Y$ a convex subset of
a linear topological space. If $f$ is a real-valued function on $X\times Y$ with

\begin{enumerate}
\item[\emph{(1)}] $f(x,\cdot)$ lower semi-continuous and quasi-convex on $Y$,
$\forall x\in X$.

\item[\emph{(}2)] $f(\cdot,y)$ upper semi-continuous and quasi-concave on $X$,
$\forall y\in Y$.
\end{enumerate}

then,
\[
\max_{x\in X}\inf_{y\in Y}f(x,y)=\inf_{y\in Y}\max_{x\in X}f(x,y)
\]

\end{lemma}

\begin{proof}[Proof of Theorem \ref{thm.strongdual}]
Define $\bx = (x_1,\cdots, x_d)$ and $\bx' = (x_1',\cdots, x_d')$. Consider the following function
	\begin{align*}
	\calL\left(\bpi,\gamma,\boldsymbol{\alpha}, \boldsymbol{\beta} \right) &=\int_{\Omega\times\Omega} \left[f\left(\bx'\right)+\sum_{k=1}^{d-1}\alpha_k(\bx'_{1:k})(x_{k+1}'-x_k')-\sum_{k=1}^{d}\beta_k(x_k)\right]\bpi(d\bx, d\bx') \\
	&\qquad +\gamma\left(\varepsilon - \int_{\Omega\times\Omega} \|\bx - \bx'\|_1 \bpi(d\bx, d\bx')\right) + \sum_{k=1}^d\int_{\Omega_k}\beta_k(x_k)\nu_k(dx_k)
	\end{align*}
	with
	\begin{align*}
	&\bpi \in \bPi := \{\text{Borel probability measures on $\Omega\times\Omega$}\};\\
	&\gamma \ge 0;\\
	&\boldsymbol{\alpha}= (\alpha_1,\cdots,\alpha_{d-1}), \textrm{ where } \alpha_k(\cdot)\in C\left(\Omega_{1:k}\right)\textrm{ for } 1\leq k\leq d-1;\\
	&\boldsymbol{\beta}= (\beta_1,\cdots,\beta_{d}), \textrm{ where } \beta_k(\cdot)\in C(\Omega_k)\textrm{ for } 1\leq k\leq d.
	\end{align*}
	Note that by Assumption \ref{aspn.compact_support}  we have $\Omega =\bigtimes_{i=1}^d\Omega_i$ is a compact subset of $\RR^d$. Thus,  $\Omega\times\Omega$ endowed with the usual metric is a compact Polish Space. As a result, we have $\bPi$ is compact with respect to the weak topology. Moreover, since $f, \left\lbrace\alpha_k\right\rbrace_{k=1}^{d-1}, \left\lbrace\beta_k\right\rbrace_{k=1}^{d}$ are continuous functions and the support of $\bpi$ is bounded, we have $\calL$ is continuous in $\bpi$.
	Finally, equip each $C(\Omega_{1:k})$ ($1\leq k\leq d-1$), $C(\Omega_k)$ ($1\leq k\leq d$) the uniform (with sup-norm) topology and $\gamma\in \RR_{\geq0}$ the Euclidean topology, we have that $\calL$ is also continuous in $\gamma$, $\alpha_k$ ($1\leq k\leq d-1$) and $\beta_k$ ($1\leq k\leq d$). Finally, $\calL$ is affine in $\left(\bpi,\gamma, \boldsymbol{\alpha}, \boldsymbol{\beta}\right)$, now we can invoke the Sion's minimax theorem (Lemma \ref{sion}) to obtain
	
	\begin{align}\label{eq.sion-minimax}
	\sup_{\bpi \in \bPi} \inf_{\mbox{$\begin{subarray}{c}\balpha\in C_{\balpha}(\Omega), \bbeta\in C_{\bbeta}(\Omega),\\
			\gamma\ge 0\end{subarray}$}} \calL\left(\bpi,\gamma,\boldsymbol{\alpha}, \boldsymbol{\beta} \right) =  \inf_{\mbox{$\begin{subarray}{c}\balpha\in C_{\balpha}(\Omega), \bbeta\in C_{\bbeta}(\Omega),\\
			\gamma\ge 0\end{subarray}$}}\sup_{\bpi \in \bPi} \calL\left(\bpi,\gamma,\boldsymbol{\alpha}, \boldsymbol{\beta} \right),
	\end{align}
	where ${C}_{\boldsymbol{\alpha}}(\Omega) = \bigtimes_{k=1}^{d-1} C(\Omega_{1:k})$, ${C}_{\boldsymbol{\beta}}(\Omega) = \bigtimes_{k=1}^{d} C(\Omega_{k})$.
	Let $\pi$ be the marginal of $\bpi$ on $\bx'$ (i.e., $\pi$ is the projection of $\bpi$ to its last $d$ marginals), we can rewrite the function $\calL$ by
	\begin{align*}
	\calL\left(\bpi,\gamma,\boldsymbol{\alpha}, \boldsymbol{\beta} \right) &=\int_{\Omega} f(\bx')\pi(d\bx') +  \int_{\Omega} \left[\sum_{k=1}^{d-1}\alpha_k(\bx'_{1:k})(x_{k+1}'-x_k')\right]\pi(d\bx') \\
	&+ \sum_{k=1}^d\left[\int_{\Omega_k}\beta_k(x_k)\nu_k(dx_k) - \int_{\Omega}\beta_k(x_k)\bpi(d\bx,d\bx')\right]\\
	&+\gamma\left(\varepsilon - \int_{\Omega\times\Omega} \|\bx - \bx'\|_1 \bpi(d\bx, d\bx')\right).
	\end{align*}
	Given any feasible $\bpi\in\bPi$, consider the inner infimum of the LHS in \eqref{eq.sion-minimax}, prevent it from being $-\infty$ we have
	\begin{equation}{\label{mart_alpha}}
	\int_{\Omega} \alpha_k(\bx'_{1:k})(x_{k+1}'-x_k')\pi(d\bx') = 0, \qquad  \forall \alpha_k(\cdot)\in C(\Omega_{1:k}), \quad \forall1\leq k\leq d-1,
	\end{equation}
	\begin{equation}\label{margin_beta}
	\int_{\Omega_k}\beta_k(x_k)\nu_k(dx_k) - \int_{\Omega}\beta_k(x_k)\bpi(d\bx,d\bx') = 0, \quad \beta_k(\cdot)\in C(\Omega_k), \quad \forall 1\leq k\leq d,
	\end{equation}
	and
	\begin{equation}\label{robust_gamma}
	\int_{\Omega\times\Omega} \|\bx-\bx'\|_1 \bpi(d\bx, d\bx')\leq \varepsilon.
	\end{equation}
	Let $\calF_k'=\sigma\left(X_1',\cdots,X_k'\right)$, $1\leq k\leq d-1$ be the canonical filtration generated by $\{X'_k\}_{k=1}^d$. Since $C(\Omega_{1:k})$ is dense in the space of $\calF_k'$-measurable integrable function on  $\Omega_{1:k}$, by the compactness of $\Omega_{1:k}$ and Dominance Convergence Theorem, the equation \eqref{mart_alpha} implies
	\begin{equation}\label{eq.mart_constraint}
	\EE_{ \pi}\left[X'_{k+1}|\calF_k'\right] = X_k' \quad \pi - a.s. \quad\forall 1\leq k\leq d-1.
	\end{equation}
	Hence $\{X_k'\}_{k=1}^d$ is a martingale under the probability measure $\pi$, which implies $\pi_1\preceq_{c}\cdots\preceq_{c}\pi_d$, where $\pi_k$ is the $k$-th marginal measure of $\pi$ ($1\leq k\leq d$).
	
	Now we turn to the equation \eqref{margin_beta}. For any $1\leq k\leq d$, note that any indicator function of Borel measurable set in $\Omega_k$ can be approximated by functions in $C(\Omega_k)$, similar to the previous argument, by the compactness of $\Omega_k$ and Dominated Convergence Theorem, we have $\bpi_k = \nu_k$ ($1\leq k\leq d$), where $\bpi_k$ is the $k$-th marginal measure of $\bpi$ (note that $\bpi$ is a $2d$-dimensional measure).
	
	Let $\bpi_{1:d}$ denote the projection of $\bpi$ to its first $d$ marginals. From the above discussions, we have $\bpi_{1:d}\in\Pi(\bnu)$. Note that $\pi$ is the projection of $\bpi$ to its last $d$ marginals. Thus,
	by inequality \eqref{robust_gamma} and the definition of Wasserstein distance, we obtain
	\begin{align*}
	d\left(\pi, \Pi(\bnu)\right)\leq \calW(\pi, \bpi_{1:d}) = \int_{\Omega\times\Omega}\|\bx-\bx'\|_1\bpi(d\bx, d\bx')\leq \varepsilon.
	\end{align*}
	Base on the above arguments, given $\bpi\in\bPi$, if the inner infimum of the LHS in \eqref{eq.sion-minimax} $\neq -\infty$, we have $\pi\in \calM_{\boldsymbol{\nu}}(\varepsilon)$  (note that $\pi$ is the projection of $\bpi$ on $\bx'$). Conversely, for any $\pi\in\calM_{\bnu}(\varepsilon)$, there exists a probability measure $\pi'\in\Pi(\bnu)$, such that $\calW(\pi, \pi')\leq \varepsilon$. Let $\bpi$ be the coupling of $\pi$ and $\pi'$ that attaining the infimum in the definition Wasserstein distance. It is straight forward to check that equation \eqref{mart_alpha}, \eqref{margin_beta} and inequality \eqref{robust_gamma} are satisfied. Thus, the inner infimum of the LHS in \eqref{eq.sion-minimax} $\neq-\infty$ for $\bpi$. 
	
	Hence, for $\pi\in\bPi$ such that the inner infimum of the LHS in \eqref{eq.sion-minimax} $\neq-\infty$, 
	\begin{align*}
	&\quad \inf_{\mbox{$\begin{subarray}{c}\balpha\in C_{\balpha}(\Omega), \bbeta\in C_{\bbeta}(\Omega),\\
			\gamma\ge 0\end{subarray}$}} \calL\left(\bpi,\gamma,\balpha, \bbeta\right) \\
	& = \inf_{\mbox{$\begin{subarray}{c}\balpha\in C_{\balpha}(\Omega), \bbeta\in C_{\bbeta}(\Omega),\\
			\gamma\ge 0 \end{subarray}$}} \EE_{\pi}\left[f(\bX')\right] +  \gamma\left(\varepsilon - \int_{\Omega\times\Omega} \|\bx-\bx'\|_1\bpi(d\bx, d\bx')\right)\\
	& = \EE_{\pi}\left[f(\bX')\right].
	\end{align*}
	Thus, 
	\begin{equation}\label{proof.RHS}
	\sup_{\bpi \in \bPi} \inf_{\mbox{$\begin{subarray}{c}\balpha\in C_{\balpha}(\Omega), \bbeta\in C_{\bbeta}(\Omega),\\
			\gamma\ge 0\end{subarray}$}} \calL\left(\bpi,\gamma,\boldsymbol{\alpha}, \boldsymbol{\beta}\right) =  I_{\boldsymbol{\nu}}(\varepsilon).
	\end{equation}
	On the other hand, an alternative expression of $\calL$ is
	\begin{equation}\label{eq.RHS_Sion}
	\begin{aligned}
	\calL\left(\bpi,\gamma,\boldsymbol{\alpha}, \boldsymbol{\beta} \right) &= \gamma\varepsilon + \sum_{k=1}^d\int_{\Omega_k}\beta_k(x_k)\nu_k(dx_k)
	+ \int_{\Omega\times\Omega} \left[f\left(\bx'\right)- \sum_{k=1}^{d}\beta_k(x_k)\right.\\
	&\left. -\gamma\|\bx - \bx'\|_1 + \sum_{k=1}^{d-1}\alpha_k(\bx'_{1:k})(x_{k+1}'-x_k')\right]\bpi(d\bx, d\bx')
	\end{aligned}
	\end{equation}
	Therefore, when taking the supremum of $\calL$ over $\bpi\in\bPi$, the probability measure $\bpi$ will pick the supremum of the inner integral in RHS of \eqref{eq.RHS_Sion} over all the $\bx, \bx'\in\Omega$, which implies the following:
	\begin{align}
	\nonumber&\sup_{\bpi\in\bPi}\calL \left(\bpi,\gamma,\boldsymbol{\alpha}, \boldsymbol{\beta} \right) \\
	\nonumber&= \gamma\varepsilon + \sum_{k=1}^d\int_{\Omega_k}\beta_k(x_k)\nu_k(dx_k)
	+ \sup_{\bx, \bx'\in\Omega}\left\lbrace f\left(\bx'\right)- \sum_{k=1}^{d}\beta_k(x_k)-\gamma\|\bx-\bx'\| \right. \\
	\nonumber& \left. \qquad\qquad \qquad\qquad\qquad\qquad\qquad\qquad\qquad\qquad + \sum_{k=1}^{d-1}\alpha_k(\bx_{1:k}')(x_{k+1}'-x_k')\right\rbrace\\
	\nonumber&= \gamma\varepsilon +  \sum_{k=1}^{d-1}\int_{\Omega_k}\beta_k(x_k)\nu_k(dx_k)
	+ \int_{\Omega_d}\beta_d(x_d)\nu_d(dx_d) + \sup_{x_d\in\Omega_{d}}\left\lbrace F\left(x_d; \gamma, \balpha,\bbeta\right) - \beta_d(x_d)\right\rbrace\\
	\nonumber&=  \gamma\varepsilon +  \sum_{k=1}^{d-1}\int_{\Omega_k}\beta_k(x_k)\nu_k(dx_k)
	+ \int_{\Omega_d}\left[\beta_d(x_d) + \sup_{x_d\in\Omega_{d}}\left\lbrace F\left(x_d; \gamma, \balpha,\bbeta\right) - \beta_d(x_d)\right\rbrace \right]\nu_d(dx_d)\\
	&\geq \gamma\varepsilon +  \sum_{k=1}^{d-1}\int_{\Omega_k}\beta_k(x_k)\nu_k(dx_k)
	+ \int_{\Omega_d} F\left(x_d; \gamma, \balpha,\bbeta\right)\nu_d(dx_d), \label{eq.inf_beta_plus_F}
	\end{align}
	where
\begin{equation*}
\begin{aligned} F\left(x_d; \gamma, \boldsymbol{\alpha} ,\boldsymbol{\beta}\right) : = \sup_{ 
\bx_{1:d-1}\in\Omega_{1:d-1},   \mathbf{x}'\in\Omega}\left\lbrace f\left(\mathbf{x}'\right)- \sum_{k=1}^{d-1}\beta_k(x_k)-\gamma\sum_{k=1}^{d}|x_k-x_k'| + \right. \\ \left. \sum_{k=1}^{d-1}\alpha_k(\mathbf{x}_{1:k}')(x_{k+1}'-x_k')\right\rbrace.\end{aligned}
\end{equation*}
Let
{
\begin{align*}
\mathcal{S}_{\boldsymbol{\nu}_{1:d-1}} =\left\{  \left(
\gamma,\boldsymbol{\alpha} ,\boldsymbol{\beta}\right)  :\gamma\geq0,\alpha_{k}\in C(\Omega_{1:k}),\beta_{k}\in C(\Omega_{k}), \int_{\Omega_{k}}\beta_{k}(x)\nu_{k}(dx)=0,\forall 1\leq k\leq
d-1\right\}.
\end{align*}
}
Note that the RHS of equation \eqref{eq.inf_beta_plus_F} would not change if we shift $\bbeta$ by a constant vector. Then, the RHS of \eqref{eq.sion-minimax} is lower bounded by
	\begin{align}
	\nonumber&\inf_{\mbox{$\begin{subarray}{c}\balpha\in C_{\balpha}(\Omega), \bbeta\in C_{\bbeta}(\Omega),\\
			\gamma\ge 0\end{subarray}$}}\sup_{\bpi \in \bPi} \calL\left(\bpi,\gamma,\balpha, \bbeta\right) \\
	\nonumber &\geq \inf_{(\gamma,\balpha, \bbeta)\in\calS_{\nu_{1:d-1}}} \gamma\varepsilon +  \sum_{k=1}^{d-1}\int_{\Omega_k}\beta_k(x_k)\nu_k(dx_k)
	+ \int_{\Omega_d} F(x_d;\gamma,\balpha,\bbeta)\nu_d(dx_d)\\
	&\geq J_{\boldsymbol{\nu}}(\varepsilon), \label{strongdual_I>=J}
	\end{align}
	where in the last step we used the fact that 
	$
	    H(\gamma, \balpha, (\bbeta_{1:d-1}, F(\cdot ; \gamma, \balpha, \bbeta))) \leq 0
	$
	for all $\bx, \bx'\in\Omega$ when $(\gamma, \balpha, \bbeta) \in \calS_{\nu_{1:d-1}}(\varepsilon)$. Combining the above equation \eqref{strongdual_I>=J} with equation \eqref{proof.RHS}, we have $I_{\bnu}(\varepsilon) \geq  J_{\bnu}(\varepsilon)$. Since the weak duality ($I_{\bnu}(\varepsilon)\leq J_{\bnu}(\varepsilon)$) always hold, we get the desired strong duality result. Finally, it is straight forward to see that $\calM_{\boldsymbol{\nu}}(\varepsilon)$ is compact in $\bPi$ with respect to weak topology, thus there exist a primal optimizer $\pi^{\textrm{DRO}}\in\calM_{\boldsymbol{\nu}}(\varepsilon)$ for $I_{\boldsymbol{\nu}}(\varepsilon)$.
\end{proof}

\subsection{Proof of Proposition \ref{prop.cotinuity_epsilon}}\label{sect.proof_consistency}

\begin{proof}[Proof of Proposition \ref{prop.cotinuity_epsilon}]

	 First of all, note that $I_{\boldsymbol{\nu}}(\varepsilon) < \infty$ when  $\varepsilon\in (R_{\boldsymbol{\nu}}, +\infty)$. For any $\varepsilon_1,\varepsilon_2\in (R_{\boldsymbol{\nu}}, +\infty)$ and $t\in[0,1]$, we claim the following inequality holds
		\begin{equation}\label{eq.prop.stability_varepsilon_final}
		I_{\boldsymbol{\nu}}( t\varepsilon_1 + (1-t)\varepsilon_2) \geq tI_{\boldsymbol{\nu}}(\varepsilon_1) + (1-t)I_{\boldsymbol{\nu}}(\varepsilon_2).
		\end{equation}
		Pick any $\pi^{(1)}\in \calM_{\boldsymbol{\nu}}(\varepsilon_1)$ and  $\pi^{(2)}\in \calM_{\boldsymbol{\nu}}(\varepsilon_2)$.Then, there exist $\widetilde{\pi}^{(1)}, \widetilde{\pi}^{(2)}\in \Pi(\bnu)$, such that 
		\begin{equation*}
		\calW\left(\pi^{(1)}, \widetilde{\pi}^{(1)}\right)\leq\varepsilon_1, \quad  \calW\left(\pi^{(2)},\widetilde{\pi}^{(2)}\right)\leq\varepsilon_2.
		\end{equation*}
		Define a new probability measure $\pi^t := t\pi^{(1)} + (1-t)\pi^{(2)}$. Note that the probability measure $\widetilde{\pi}^t:= t\widetilde{\pi}^{(1)} + (1-t)\widetilde{\pi}^{(2)}\in \Pi(\bnu)$, by the dual representation of Wasserstein distance,  
		\begin{align*}
		    d\left(\pi^t, \Pi(\nu)\right)&\leq \sup_{f\in\textrm{Lip}_1}\left\lbrace \int fd\pi^t - \int f d\widetilde{\pi}^t\right\rbrace \\
		    &\leq t\sup_{f\in\textrm{Lip}_1}\left\lbrace \int fd\pi^{(1)} - \int f d\widetilde{\pi}^{(1)}\right\rbrace + (1-t)\sup_{f\in\textrm{Lip}_1}\left\lbrace \int fd\pi^{(2)} - \int f d\widetilde{\pi}^{(2)}\right\rbrace \\
		    &= t\calW\left(\pi^{(1)}, \widetilde{\pi}^{(1)}\right) + (1-t)\calW\left(\pi^{(2)}, \widetilde{\pi}^{(2)}\right)\\
		    &\leq t\varepsilon_1 + (1-t)\varepsilon_2.
		\end{align*}
		Next, we show that $\pi^t$ is a martingale measure. For any $1\leq k\leq d-1$, and any continuous function $\alpha_k(\bx_{1:k})\in C(\Omega_{1:k})$,
		\begin{align*}
		    & \EE_{\pi^t}\left[\alpha_k(\bX_{1:k})(X_{k+1}-X_k)\right] \\
		    =\quad &t\EE_{\pi^{(1)}}\left[\alpha_k(\bX_{1:k})(X_{k+1}-X_k)\right] + (1-t)\EE_{\pi^{(2)}}\left[\alpha_k(\bX_{1:k})(X_{k+1}-X_k)\right] \\
		    =\quad & 0.
		\end{align*}
		Hence, $\pi^t\in\calM(\Omega)$. To sum up, we have $\pi^t\in  \calM_{\boldsymbol{\nu}}(t\varepsilon_1 + (1-t)\varepsilon_2)$, and thus
		\begin{equation}\label{eq.prop.stability_varepsilon}
		I_{\boldsymbol{\nu}}( t\varepsilon_1 + (1-t)\varepsilon_2) \geq \EE_{\pi^t}[f(\bX)] = t\EE_{\pi^{(1)}}[f(\bX)] + (1-t) \EE_{\pi^{(2)}}[f(\bX)].
		\end{equation}
		Take the supremum of the RHS of \eqref{eq.prop.stability_varepsilon} over $\pi^{(1)}\in \calM_{\boldsymbol{\nu}}(\varepsilon_1)$ and  $\pi^{(2)}\in \calM_{\boldsymbol{\nu}}(\varepsilon_2)$, we prove the desired inequality in \eqref{eq.prop.stability_varepsilon_final}, and thus $I_{\boldsymbol{\nu}}( \varepsilon)$ is concave and continuous.
		
		Suppose now $\nu_1\preceq_c\cdots\preceq_c \nu_d$. It suffices to show that $I_{\boldsymbol{\nu}}(\varepsilon)$ is right continuous at 0. Take any sequence $\varepsilon_n\rightarrow0$, and $\pi^{n}\in\calM_{\boldsymbol{\nu}}(\varepsilon_n)$ that optimize $I_{\boldsymbol{\nu}}(\varepsilon_n)$. Observe that
		\begin{align*}
		\sup_{n\geq 1}\EE_{\pi^{n}}\|\bX\|_1 &= \sup_{n\geq1}\left\lbrace \sum_{i=1}^d \int|x|\pi_i^{n}(dx)\right\rbrace \\
		&\leq \sup_{n\geq1}\left\lbrace \sum_{i=1}^d \int|x|\nu_i(dx) + \sum_{i=1}^d\calW\left(\pi_i^{n},\nu_i\right)\right\rbrace\\
		&\leq \sup_{n\geq1}\left\lbrace \sum_{i=1}^d \int|x|\nu_i(dx) + d\left(\pi^n, \Pi(\bnu)\right)\right\rbrace\\
		&\leq \sum_{i=1}^dm_1(\nu_i) + \sup_{n\geq 1}\varepsilon_n <\infty.
		\end{align*}
	Hence, $\left\lbrace\pi^{(n)}\right\rbrace_{n\geq 1}$ is uniformly tight. Without loss of generality, we may assume $\pi^n$ converge weakly to a probability measure $\pi$. Then,
	\begin{equation*}
	    d\left(\pi, \Pi(\bnu)\right) \leq \calW(\pi, \pi^n) + d\left(\pi^n, \Pi(\bnu)\right) \rightarrow 0.
	\end{equation*}
	Thus, $\pi\in\Pi(\bnu)$. Moreover, for all $n\geq 1$, 
	\begin{equation*}
	\EE_{\pi^{n}}\left[\alpha_k\left(\bX_{1:k}\right)(X_{k+1}-X_k)\right] = 0
	\end{equation*}
	holds for all continuous function $\alpha_k(\cdot)\in C(\Omega_{1:k})$ ($1\leq k\leq d-1$). Follows from the Dominated Convergence Theorem, the above equations also hold for $\pi$, which implies $\pi\in\calM_{\bnu}(0)$. As a result,
	\begin{equation*}
	\limsup_{n\rightarrow\infty}I_{\boldsymbol{\nu}}(\varepsilon_n) = \limsup_{n\rightarrow\infty} \EE_{\pi^{n}}[f(\bX)] = \EE_{\pi}[f(\bX)]\leq I_{\boldsymbol{\nu}}(0).
	\end{equation*}
	The reverse inequality holds trivially from the definition.
\end{proof}

\section{Proof of Theorem \ref{thm.distance_n}}\label{sect.proof_theorem_distance_n}
We present the proof of Theorem \ref{thm.distance_n} in below. The key idea is to leverage the Wasserstein concentration of empirical measure and the dual formulation we get in Theorem \ref{thm.strongdual} and Corollary \ref{cor.dual_reformulation}. 
\begin{proof}[Proof of Theorem \ref{thm.distance_n}]
For simplicity, let us introduce the following notations:
\begin{align*}
    \bmu_0 &:= \left(\widehat{\mu}_1^{(n)},\widehat{\mu}_2^{(n)},\cdots,\widehat{\mu}_d^{(n)}\right),\\
    \bmu_d &:= (\mu_1,\cdots,\mu_d),\\
    \bmu_k &:= \left(\widehat{\mu}_{1}^{(n)},\cdots, \widehat{\mu}_{d-k}^{(n)}, \mu_{d-k+1},\cdots,\mu_d\right), \textrm{ for all } 1\leq k\leq d. 
\end{align*}
	By triangular inequality we have
	\begin{equation}\label{thm2.triangular}
	\left|I_{n}(\varepsilon) - I(\varepsilon)\right| = \left|I_{\boldsymbol{\mu}_0}(\varepsilon) - I_{\boldsymbol{\mu}_d}(\varepsilon)\right| \leq \sum_{k=0}^{d-1} \left|I_{\boldsymbol{\mu}_k}(\varepsilon) - I_{\boldsymbol{\mu}_{k+1}}(\varepsilon)\right|.
	\end{equation}
	Since for each $k$, $\boldsymbol{\mu}_k$ and $\boldsymbol{\mu}_{k+1}$ are only different from  one marginal, by the symmetricity that mentioned in Remark \ref{remark.symmetry}, we just need to bound one of the $d$ terms in \eqref{thm2.triangular}, say $\left| I_{\boldsymbol{\mu}_0}(\varepsilon) - I_{\boldsymbol{\mu}_1}(\varepsilon)\right|$.
	
	By Corollary \ref{cor.dual_reformulation} and Lemma  \ref{lemma.feasibility}, for all $n\geq N(\varepsilon,\delta)$, with probability at least $1-\delta$ we have
	\begin{equation*}\label{eq.I_mu0}
	I_{\boldsymbol{\mu}_0}(\varepsilon) =  \inf_{\left(  \gamma, \boldsymbol{\alpha},   \boldsymbol{\beta} \right) \in\calS_{\widehat{\bmu}^n_{1:d-1}}}\gamma\varepsilon + \int_{\Omega_d}
	F\left(x_d; \gamma, \balpha. \bbeta\right)\widehat{\mu}_d^{(n)}(dx_d).
	\end{equation*}
	Given $B>0$, define
	\begin{equation*}\label{eq.dual_of_I_leq_B}
	I_{\boldsymbol{\mu}_0}^{\leq B}(\varepsilon) := \inf_{\mbox{$\begin{subarray}{c}\left(  \gamma, \boldsymbol{\alpha},   \boldsymbol{\beta} \right) \in\calS_{\widehat{\bmu}^n_{1:d-1}},\\
			0\leq\gamma\leq B. \end{subarray}$}}\gamma\varepsilon + \int_{\Omega_d}
	F\left(x_d; \gamma, \balpha, \bbeta \right)\widehat\mu_d^{(n)}(dx_d)
	\end{equation*}
	and
	\begin{equation*}
	I_{\boldsymbol{\mu}_0}^{> B}(\varepsilon) := \inf_{\mbox{$\begin{subarray}{c}\left(  \gamma, \boldsymbol{\alpha},   \boldsymbol{\beta} \right) \in\calS_{\widehat{\bmu}^n_{1:d-1}},\\
			\gamma> B. \end{subarray}$}}\gamma\varepsilon + \int_{\Omega_d}
	F\left(x_d; \gamma, \balpha, \bbeta \right)\widehat{\mu}_d^{(n)}(dx_d).
	\end{equation*}
	It is straight forward to see that $I_{\boldsymbol{\mu}_0}^{\leq B}(\varepsilon)$ is upper bounded by $\sup_{\bX\in\Omega}f(\bX)$ (simply set all the parameters equal to zero).
	Moreover, by the definition of $F$ we see that
	\begin{equation*}
	\begin{aligned}
	I_{\boldsymbol{\mu}_0}^{> B}(\varepsilon) \geq \inf_{\mbox{$\begin{subarray}{c}\left(  \gamma, \boldsymbol{\alpha},   \boldsymbol{\beta} \right) \in\calS_{\widehat{\bmu}^n_{1:d-1}},\\
			\gamma> B \end{subarray}$}}\gamma\varepsilon + \int\left(\int\left( f\left(\bx'\right)- \sum_{k=1}^{d-1}\beta_k(x_k)-\gamma\sum_{k=1}^{d}|x_k-x_k'| + \right.\right. \\
	\left.\left. \sum_{k=1}^{d-1}\alpha_k(\bx_{1:k}')(x_{k+1}'-x_k')\right)\bpi\left(d\bx_{1:d-1}, d\bx' \mid x_d\right)\right)\widehat{\mu}_d^{(n)}(dx_d)
	\end{aligned}
	\end{equation*}
	holds for every conditional probability measure $\bpi(d\bx_{1:d-1},d\bx'\mid x_d)$. By the theory in the standard Martingale Optimal Transport, there exists an martingale measure $\pi\in\Pi(\bmu)$ that optimize $I(0)$. In particular, we can pick the $\bpi(d\bx_{1:d-1}, d\bx' \mid x_d)$ such that its projection on $\bx'$ is $\pi$, and its projection on $\bx_{1:d-1}$ is a $(d-1)$- dimensional probability measure with marginals $\widehat{\mu}_1^{(n)},\cdots,\widehat{\mu}_{d-1}^{(n)}$. Taking supremum over possible $\bpi$'s ($\mu_i$ and $\widehat{\mu}_i^{(n)}$ will be coupled optimally, for all $1\leq i\leq d$), we have
	\begingroup
	\allowdisplaybreaks
	\begin{align*}
	I_{\boldsymbol{\mu}_0}^{> B}(f;\varepsilon) &\geq \inf_{\mbox{$\begin{subarray}{c}\left(  \gamma, \boldsymbol{\alpha},   \boldsymbol{\beta} \right) \in\calS_{\widehat{\bmu}^n_{1:d-1}},\\
			\gamma> B \end{subarray}$}}\gamma\varepsilon + \EE_{\pi}\left[f(\bX')\right] - \sum_{k=1}^{d-1}\int_{\Omega_k}\beta_k(x_k)\widehat{\mu}_k^{(n)}(dx_k) + \\
	& \qquad\qquad\qquad\qquad\qquad \sum_{k=1}^{d-1}\EE_{\pi}\left[\alpha_k(\bX_{1:k}')(X_{k+1}'-X_k')\right] -\gamma\sum_{k=1}^d\calW\left(\mu_k,\widehat{\mu}_k^{(n)}\right)\\
	&\geq \inf_{\mbox{$\begin{subarray}{c}\left(  \gamma, \boldsymbol{\alpha},   \boldsymbol{\beta} \right) \in\calS_{\widehat{\bmu}^n_{1:d-1}},\\
			\gamma> B \end{subarray}$}}\gamma\varepsilon + \EE_{\pi}\left[f(\bX')\right] - \gamma\sum_{k=1}^d\calW\left(\mu_k,\widehat{\mu}_k^{(n)}\right)\\
	&> \EE_{\pi}\left[f(\bX')\right] + B\left(\varepsilon - \sum_{k=1}^d\calW\left(\mu_k,\widehat{\mu}_k^{(n)}\right)\right),
	\end{align*}
	\endgroup
	where the second inequality have used the fact that $\pi$ is a martingale measure.
	By Lemma \ref{lemma.feasibility}, for $n\geq N(\frac{\varepsilon}{2}, \frac{\delta}{2d})$, we see that with probability $1 - \frac{\delta}{2d}$,
	\begin{equation}\label{ieq.proof_thm_distance_n_concentration}
	\sum_{k=1}^d\calW\left(\mu_k,\widehat{\mu}_k^{(n)}\right) \leq \frac{\varepsilon}{2}.
	\end{equation}
	Taking $B = \frac{4D}{\varepsilon} \geq \frac{4}{\varepsilon}\sup_{\bX\in\Omega}\left|f(\bX)\right|$, we have
	\begin{align*}
	I_{\boldsymbol{\mu}_0}^{> B}(\varepsilon) > \EE_{\pi}f(\bX') + \frac{4}{\varepsilon}\sup_{\bX\in\Omega}\left|f(\bX)\right|\cdot\frac{\varepsilon}{2}>\sup_{\bX\in\Omega}f(\bX)\geq I_{\boldsymbol{\mu}_0}^{\leq B}(\varepsilon).
	\end{align*}
	Hence we have $\gamma$ is bounded by $B$ with high probability. Precisely, we have
	\begin{equation*}
	I_{\boldsymbol{\mu}_0}(\varepsilon) = I_{\boldsymbol{\mu}_0}^{\leq B}(\varepsilon).
	\end{equation*}
	holds with probability at least $1 - \frac{\delta}{2d}$ for $n\geq N(\frac{\varepsilon}{2}, \frac{\delta}{2d})$. Similar arguments can be applied to $I_{\boldsymbol{\mu}_1}(\varepsilon)$ and we get
	\begin{equation*}
	I_{\boldsymbol{\mu}_1}(\varepsilon) =  I_{\boldsymbol{\mu}_1}^{\leq B}(\varepsilon).
	\end{equation*}
	holds with probability at least $1-\frac{\delta}{2d}$ for $n\geq N\left(\frac{\varepsilon}{2},\frac{\delta}{2d}\right)$. By definition, for any $\delta >0$, there exists $\left(\gamma_{\delta}, \balpha_{\delta},\bbeta_{\delta} \right)\in\calS_{\widehat{\bmu}^n_{1:d-1}}(\varepsilon)$ such that
	\begin{equation*}
	\gamma_{\delta}\varepsilon + \int_{\Omega_d}
	F\left(x_d; \gamma_{\delta}, \balpha_{\delta}, \bbeta_{\delta}\right)\widehat\mu_d^{(n)}(dx_d)\leq I_{\boldsymbol{\mu}_0}(\varepsilon) + \delta.
	\end{equation*}
	Note that $F(x_d;\gamma, \balpha, \bbeta)$ is a $\gamma$-Lipschitz function for fixed $\gamma$, by the dual representation of Wasserstein distance,
	\begin{align*}
	I_{\boldsymbol{\mu}_1}(\varepsilon) - I_{\boldsymbol{\mu}_0}(\varepsilon) & = I_{\boldsymbol{\mu}_1}^{\leq B}(\varepsilon) - I_{\boldsymbol{\mu}_0}^{\leq B}(\varepsilon)\\
	&\leq \gamma_{\delta}\varepsilon + \int_{\Omega_d}
	F\left(x_d; \gamma_{\delta}, \balpha_{\delta}, \bbeta_{\delta}\right)\mu_d(dx_d) -\left(\gamma_{\delta}\varepsilon + \int_{\Omega_d}
	F\left(x_d; \gamma_{\delta}, \balpha_{\delta}, \bbeta_{\delta}\right)\widehat\mu_d^{(n)}(dx_d) - \delta\right)\\
	&\leq B\cdot\calW\left(\mu_d,\widehat{\mu}_d^{(n)}\right) + \delta.
	\end{align*}
	Similarly, we have
	$I_{\boldsymbol{\mu}_0}(\varepsilon) - I_{\boldsymbol{\mu}_1}(\varepsilon)\leq B\cdot\calW\left(\mu_d,\widehat{\mu}_d^{(n)}\right) + \delta$. Thus, $\delta \rightarrow 0$ yields that
	\begin{equation*}
	\left|I_{\boldsymbol{\mu}_0}(\varepsilon) - I_{\boldsymbol{\mu}_1}(\varepsilon) \right|\leq B\cdot\calW\left(\mu_d,\widehat{\mu}_d^{(n)}\right).
	\end{equation*}
	Recall from Lemma \ref{lemma.feasibility} that  $\calW\left(\mu_d,\widehat{\mu}_d^{(n)}\right)\leq\sqrt{\frac{\log(2C_1d/\delta)}{C_2n}}$ with probability at least $1-\frac{\delta}{2d}$. Then, the union bound yields
	\begin{equation}\label{eq.thm.stability_bound}
	\left|I_{\boldsymbol{\mu}_0}(\varepsilon) - I_{\boldsymbol{\mu}_1}(\varepsilon) \right|\leq B\sqrt{\frac{\log(2C_1d/\delta)}{C_2n}}
	\end{equation}
	holds with probability at least $1-\frac{\delta}{d}$ when $n\geq N(\frac{\varepsilon}{2}, \frac{\delta}{2d})$. Finally, similar arguments works for all $\left|I_{\boldsymbol{\mu}_k}(\varepsilon) - I_{\boldsymbol{\mu}_{k+1}}(\varepsilon) \right|$ and give the same bound as \eqref{eq.thm.stability_bound}.
	Thus, by union bound, with probability at least $1-\delta$ we have
	\begin{equation*}
	\left|I_n(\varepsilon) - I(\varepsilon)\right|\le Bd\sqrt{\frac{\log\left(2C_1d/\delta\right)}{C_2n}} = \frac{4Dd}{\varepsilon} \sqrt{\frac{\log\left(2C_1d/\delta\right)}{C_2n}}
	\end{equation*}
	holds for $n\geq N(\frac{\varepsilon}{2}, \frac{\delta}{2d})$, which is the desired result.
\end{proof}
Finally, we sketch the proof of Theorem \ref{thm.distance_n_general}. The proof is essentially the same as the proof of Theorem \ref{thm.distance_n}, while we need to consider a concentration result in general dimensions.
\begin{proof}[Proof of Theorem \ref{thm.distance_n_general}]
Firstly, assuming compact support, it is straight forward to check that the parallel results of Theorem \ref{thm.strongdual} and Corollary \ref{cor.dual_reformulation} hold when we lift each marginals to $k$-dimensional space. Next, by Theorem 2 in \cite{fournier2015rate}, we have for any $n\geq \frac{2^kd^k}{C_2'\varepsilon^k} \log\left(\frac{2d^2C_1'}{\delta}\right)$ that 
\begin{equation*}
	\PP\left(\sum_{i=1}^d\calW(\widehat{\mu}_i^{(n)},\mu_i)>\frac{\varepsilon}{2}\right)\leq\sum_{i=1}^d\PP\left(\calW(\widehat{\mu}_i^{(n)},\mu_i)>\frac{\varepsilon}{2d}\right)\leq dC_1'e^{-C_2'n\varepsilon^k/(2d)^k}<\frac{\delta}{2d}.
	\end{equation*}
Thus, with probability at least $1-\frac{\delta}{2d}$,
\begin{equation}\label{ieq.proof_distance_n_general_concentration}
    \sum_{i=1}^d\calW\left(\mu_i,\widehat\mu_i^{(n)}\right)\leq\frac{\varepsilon}{2}.
\end{equation}
We can replace \eqref{ieq.proof_thm_distance_n_concentration} in the proof of Theorem \ref{thm.distance_n} by \eqref{ieq.proof_distance_n_general_concentration}, and repeat the same arguments to get the desired upper bound in Theorem \ref{thm.distance_n_general}.
\end{proof}

\subsection{Proof of Technical Result Supporting Theorem \ref{thm.distance_n}} \label{sect.proof_feasiblity}

In this section we present the proof of Lemma \ref{lemma.feasibility}.

\begin{proof}[Proof of Lemma \ref{lemma.feasibility}]
	$I_{n}(\varepsilon)$ is feasible if and only if $\calM_{n}(\varepsilon)\neq \varnothing$. 
	Due to the compactness assumption, each marginal measure $\mu_i$ has a finite square-exponential moment, that is, $\sup_{1\leq i\leq d}\EE_{\mu_i}[e^{\gamma |X|^2}] < \infty$ for some $\gamma >0$. Thus, we can apply Theorem 2 in \cite{fournier2015rate} to conclude: there exist universal constants $C_1, C_2 >0$, such that 
	\begin{equation*}\label{eq.concentrateEm}
	\PP\left(\sum_{i=1}^d\calW(\widehat{\mu}_i^{(n)},\mu_i)>\varepsilon\right)\leq\sum_{i=1}^d\PP\left(\calW(\widehat{\mu}_i^{(n)},\mu_i)>\frac{\varepsilon}{d}\right)\leq dC_1e^{-C_2n\varepsilon^2/d^2}<\delta
	\end{equation*}
	holds for $n\geq N(\varepsilon,\delta):= \frac{d^{2}}{C_{2}}\log\left(\frac{d C_{1}}{\delta}\right)\varepsilon^{-2}$. Hence, with probability at least $1-\delta$, we have 
	\begin{equation*}
	    \sum_{i=1}^d\calW\left(\widehat{\mu}_i^{(n)},\mu_i\right)\leq\varepsilon.
	\end{equation*}
	Let $\pi^*\in\Pi(\bmu)$ be the optimal martingale measure for the MOT problem \eqref{def.standard_MOT}. Suppose that for any $1\leq i\leq d$, $\widetilde{\pi}_i\in \Pi\left(\widehat{\mu}_i^{(n)},\mu_i\right)$ is the coupling such that that $\EE_{(X,X')\sim \widetilde{\pi}_i} |X-X'| = \calW \left(\widehat{\mu}_i^{(n)},\mu_i\right)$. Then,  
	their exists a probability measure $\pi'\in\Pi(\widehat{\bmu}^n)$ and a joint measure $\bpi\in\Pi(\pi^*, \pi')$, such that $\pi^*_i$ and $\pi'_i$ are coupled exactly by $\widetilde{\pi}_i$ for all $1\leq i\leq d$. Thus,
	\begin{equation*}
	    d\left(\pi^*, \Pi(\widehat{\bmu}^n)\right) \leq \calW(\pi^*, \pi') \leq \EE_{\bpi}\|\bX-\bX'\|_1=\sum_{i=1}^d\EE_{\bpi}|X_i-X_i'|=\sum_{i=1}^d\calW\left(\widehat{\mu}_i^{(n)},\mu_i\right)\leq \varepsilon.
	\end{equation*}
	In other words, $\pi^*\in\calM_n(\varepsilon)$ with probability at least $1-\delta$, which yields the desired result.
\end{proof}

\section{Proof of Theorem \ref{thm.discretize_approximation}}\label{sect.proof_thm_discrete_approx}

In this section, we give the full proof of Theorem \ref{thm.discretize_approximation}.

\begin{proof}[Proof of Theorem \ref{thm.discretize_approximation}]
	We first show that  $\calM_{n, N}\left(\varepsilon + \frac{ld^{1/2}}{N}; \frac{l}{N}\right)\neq \varnothing$ with probability at least $1-\delta$ for $n\geq N(\varepsilon,\delta)$, which ensure the optimization problems in LHS of \eqref{eq.discretize_approximation} are feasible.
	
	For any $\pi \in\calP_n(\varepsilon)$, by Lemma \ref{lemma.discretize} there exists a $l/N-$martingale measure $\pi'\in\calP\left(\Omega^{(N)}\right)$ with $\calW(\pi',\pi)\leq\frac{l\sqrt{d}}{N}$. Thus,
	\begin{equation*}
	    d\left(\pi', \Pi(\widehat{\bmu}^n)\right) \leq \calW(\pi', \pi) + d\left(\pi, \Pi(\widehat{\bmu}^n)\right) \leq \varepsilon + \frac{ld^{1/2}}{N}.
	\end{equation*}
	In other words, we have $\pi'\in\calM_{n, N}\left(\varepsilon + \frac{ld^{1/2}}{N}; \frac{l}{N}\right)\neq\varnothing$. Furthermore,
	\begin{equation*}
	\begin{aligned}
	\EE_{\pi}\left[f(\bX)\right] - I_{n,N}\left(\varepsilon + \frac{ld^{1/2}}{N},  \frac{l}{N}\right) &\leq \EE_{\pi}\left[f(\bX)\right] - \EE_{\pi'}\left[f(\bX)\right] \\
	&\leq L \calW(\pi',\pi)\leq\frac{Lld^{1/2}}{N}.
	\end{aligned}
	\end{equation*}
	Taking the supremum over $\pi \in\calM_n(\varepsilon)$ yields
	\begin{equation}\label{eq.RHS_discretize_approximation}
	I_n(\varepsilon) -  I^{(N)}_{n, N}\left(\varepsilon + \frac{ld^{1/2}}{N},  \frac{l}{N}\right)  \leq \frac{Lld^{1/2}}{N}.
	\end{equation}
	Next, we lower bound the LHS of \eqref{eq.RHS_discretize_approximation}. Recall from Corollary \ref{cor.dual_reformulation} and Theorem \ref{thm.distance_n} that
	\begin{equation*}
	I_n(\varepsilon) =
	\inf_{\mbox{$\begin{subarray}{c} \left(  \gamma, \boldsymbol{\alpha},   \boldsymbol{\beta} \right) \in\calS_{\widehat{\bmu}^n_{1:d-1}},\\
			\gamma \leq \frac{4D}{\varepsilon}\end{subarray}$}}\gamma\varepsilon + \int_{\Omega_d}
	F\left(x_d;\gamma. \balpha. \bbeta\right)\widehat{\mu}_d^{(n)}(dx_d)
	\end{equation*}
	holds with probability at least $1-\delta$ for $n\geq N\left(\varepsilon/2, \delta \right)$, where 
	\begin{equation}\label{eq.thm.discretize_F(x_d)}
    \begin{aligned} F\left(x_d; \gamma, \boldsymbol{\alpha} ,\boldsymbol{\beta}\right) : = \sup_{ \bx_{1:d-1}\in\Omega_{1:d-1},   \mathbf{x}'\in\Omega}\left\lbrace f\left(\mathbf{x}'\right)- \sum_{k=1}^{d-1}\beta_k(x_k)-\gamma\sum_{k=1}^{d}|x_k-x_k'| + \right. \\ \left. \sum_{k=1}^{d-1}\alpha_k(\mathbf{x}_{1:k}')(x_{k+1}'-x_k')\right\rbrace.
    \end{aligned}   
    \end{equation}
	Note that
	\begin{align*}
	F(x_d;\gamma. \balpha. \bbeta)&\geq \sup_{\mbox{$\begin{subarray}{c} \bx_{1:d-1}\in\Omega_{1:d-1},   \mathbf{x}'\in\Omega\end{subarray}$}}\left\lbrace f\left(\bx'\right)- \sum_{k=1}^{d-1}\beta_k(x_k)+ \sum_{k=1}^{d-1}\alpha_k(\bx_{1:k}')(x_{k+1}'-x_k')\right\rbrace - \gamma \sum_{i=1}^dl_i\\
	&= \sup_{\bx'\in\Omega}\left\lbrace  \sum_{k=1}^{d-1}\alpha_k(\bx_{1:k}')(x_{k+1}'-x_k')\right\rbrace + \inf_{\bx'\in\Omega}f\left(\bx'\right) - \sum_{k=1}^{d-1}\inf_{x_k\in\Omega_k}\beta_k(x_k) - \gamma ld\\
	&\geq \sup_{\bx'\in\Omega}\left\lbrace  \sum_{k=1}^{d-1}\alpha_k(\bx_{1:k}')(x_{k+1}'-x_k')\right\rbrace + \inf_{\bx'\in\Omega}f\left(\bx'\right) - \gamma ld.
	\end{align*}
	In the last step we've used the fact that  $\inf_{x_k\in\Omega_k}\beta_k(x_k)\leq\int_{\Omega_k}\beta_k(x_k)\widehat\mu^{(n)}_k(dx_k) = 0$ $(1\leq k\leq d)$ for $\beta_{k}$'s in $\calS_{\widehat{\bmu}^n_{1:d-1}}$. By letting $x_i' = x_2'$ for all $2\leq i\leq d$, we have
	\begin{align*}
	F(x_d;\gamma, \balpha, \bbeta)\geq \sup_{(x_1', x_2')\in\Omega_1\times\Omega_2}\left\lbrace \alpha_1(x_1')(x_2' - x_1')\right\rbrace +  \inf_{\bx'\in\Omega}f\left(\bx'\right) - \gamma ld.
	\end{align*}
	Suppose that $\|\alpha_1\|_{\infty}:=\sup_{x_1'\in\Omega_1}|\alpha_1(x_1')|> B_1$ for some positive number $B_1$. If there exists a $x_1^*\in\Omega_1$ such that $\alpha_1(x_1^*)\geq B_1$, by Assumption \ref{aspn.disperse} we get
	\begin{equation*}
	F(x_d;\gamma, \balpha, \bbeta)\geq \alpha_1(x_1^*)(b_2 - x_1^*) + \inf_{\bx'\in\Omega}f\left(\bx'\right) - \gamma ld> \tau B_1 + \inf_{\bx'\in\Omega}f\left(\bx'\right) - \gamma ld.
	\end{equation*}
	Similarly, if there exits an $x_1^*\in\Omega_1$ such that $\alpha_1(x_1^*)\leq -B_1$, then
	\begin{equation*}
	F(x_d;\gamma, \balpha, \bbeta)\geq \alpha_1(x_1^*)(a_2 - x_1^*) + \inf_{\bx'\in\Omega}f\left(\bx'\right) - \gamma ld> \tau B_1+ \inf_{\bx'\in\Omega}f\left(\bx'\right) - \gamma ld.
	\end{equation*}
	Combining with the bound of $\gamma$ shown in Theorem \ref{thm.distance_n}, we have
	\begin{align*}
	\inf_{\mbox{$\begin{subarray}{c} \left(  \gamma, \boldsymbol{\alpha},   \boldsymbol{\beta} \right) \in\calS_{\widehat{\bmu}^n_{1:d-1}},\\
			\gamma \leq \frac{4D}{\varepsilon}, \|\alpha\|_{\infty}\geq B_1\end{subarray}$}}\gamma\varepsilon + \int_{\Omega_d}
	F\left(x_d;\gamma, \balpha, \bbeta\right)\widehat{\mu}_d^{(n)}(dx_d) &\geq \gamma\varepsilon + \int_{\Omega_d}\left[\tau B_1+ \inf_{\bx'\in\Omega}f\left(\bx'\right) - \gamma dl\right]\widehat{\mu}_d^{(n)}(dx_d)\\
	&\geq \tau B_1+ \inf_{\bx'\in\Omega}f\left(\bx'\right) -  \frac{4}{\varepsilon}\sup_{\bx\in\Omega}f(\bx)\cdot ld\\
	&> \sup_{\bx\in\Omega}f(\bx)\\
	&\geq I_n(\varepsilon).
	\end{align*}
	where $B_1 = \frac{\left[\frac{4ld}{\varepsilon} + 2\right]D}{\tau} \geq \frac{\left[\frac{4ld}{\varepsilon} + 2\right]\sup_{\bx\in\Omega}|f(\bx)|}{\tau}$.
	As a result,
	\begin{equation*}
	I_n(\varepsilon) = \inf_{\mbox{$\begin{subarray}{c} \left(  \gamma, \boldsymbol{\alpha},   \boldsymbol{\beta} \right) \in\calS_{\widehat{\bmu}^n_{1:d-1}},\\
			\gamma \leq \frac{4D}{\varepsilon}, \|\alpha\|_{\infty}\leq B_1\end{subarray}$}}\gamma\varepsilon + \int_{\Omega_d}
	F\left(x_d;\gamma,\alpha, \bbeta\right)\widehat{\mu}_d^{(n)}(dx_d).
	\end{equation*}
	Next, we use an induction argument to show that each $\alpha_i$ has bounded supremum norm. Suppose we have
	\begin{equation*}
	I_n(\varepsilon) = \inf_{\mbox{$\begin{subarray}{c} \left(  \gamma, \boldsymbol{\alpha},   \boldsymbol{\beta} \right) \in\calS_{\widehat{\bmu}^n_{1:d-1}},\\
			\gamma \leq \frac{4D}{\varepsilon}, \|\alpha_i\|_{\infty}\leq B_i, (1\leq i\leq m) \end{subarray}$}}\gamma\varepsilon + \int_{\Omega_d}
	F\left(x_d\right)\widehat{\mu}_d^{(n)}(dx_d)
	\end{equation*}
	holds for some $1\leq m\leq d-2$. Conditioning on  $\gamma \leq \frac{4D}{\varepsilon}$ and $\|\alpha_i\|_{\infty}\leq B_i, 1\leq i\leq m$, we have
	\begin{align*}
	F(x_d;\gamma, \balpha, \bbeta)&\geq \sup_{\bx'\in\Omega}\left\lbrace  \sum_{k=1}^{d-1}\alpha_k(\bx_{1:k}')(x_{k+1}'-x_k')\right\rbrace + \inf_{\bx'\in\Omega}f\left(\bx'\right) - \gamma ld\\
	&\geq \sup_{\bx'\in\Omega}\left\lbrace \sum_{k=m+1}^{d-1}\alpha_k(\bx_{1:k}')(x_{k+1}'-x_k')\right\rbrace -\sum_{k=1}^m B_il_i + \inf_{\bx'\in\Omega}f\left(\bx'\right) - \gamma ld.
	\end{align*}
	Taking $x_k' = x_{m+2}$ for all $k\geq m+2$, conditioning on $\|\alpha_{m+1}\|_{\infty}> B_{m+1}$, we can lower bounded $F(x_d;\gamma, \balpha, \bbeta)$ by
	\begin{align*}
	& F(x_d;\gamma, \balpha, \bbeta)\\
	&\geq \sup_{(x_{m+1}', x_{m+2}')\in\Omega_{m+1}\times\Omega_{m+2}}\left\lbrace \alpha_{m+1}\left(\bx'_{1:m+1}\right)(x'_{m+2}-x'_{m+1})\right\rbrace - l\sum_{k=1}^mB_i+\inf_{\bx'\in\Omega}f\left(\bx'\right) - \gamma ld\\
	&\geq \inf_{\bx'\in\Omega}f\left(\bx'\right) - \frac{4Dld}{\varepsilon} + \tau B_{m+1}  - l\sum_{k=1}^mB_i.
	\end{align*}
	Hence
	{\small
		\begin{align*}
		\inf_{\mbox{$\begin{subarray}{c} \left(  \gamma, \boldsymbol{\alpha},   \boldsymbol{\beta} \right) \in\calS_{\widehat{\bmu}^n_{1:d-1}}, \gamma \leq \frac{4D}{\varepsilon}\\
				\|\alpha_i\|_{\infty}\leq B_i, 1\leq i\leq m, \|\alpha_{m+1}\|_{\infty}> B_{m+1}. \end{subarray}$}}\gamma\varepsilon + \int_{\Omega_d}
		F\left(x_d;\gamma, \balpha, \bbeta\right)\widehat{\mu}_d^{(n)}(dx_d) &\geq  \inf_{\bx'\in\Omega}f\left(\bx'\right) - \frac{4Dld}{\varepsilon} + \tau B_{m+1}  - l\sum_{k=1}^mB_i\\
		&> \sup_{\bx\in\Omega}f(\bx)\geq I_n(\varepsilon),
		\end{align*}
	}
	where
	\begin{equation}\label{eq.iteration_Bi}
	B_{m+1}:= \frac{\left[\frac{4ld}{\varepsilon} + 2\right]D}{\tau}+\frac{l}{\tau}\sum_{k=1}^mB_k = B_1 + \frac{l}{\tau}\sum_{k=1}^mB_k.
	\end{equation}
	Now, we come to the following result,
	\begin{equation*}
	I_n(\varepsilon) = \inf_{\mbox{$\begin{subarray}{c} \left(  \gamma, \boldsymbol{\alpha},   \boldsymbol{\beta} \right) \in\calS_{\widehat{\bmu}^n_{1:d-1}},\\
			\gamma\leq\frac{4D}{\varepsilon}, \|\alpha_i\|_{\infty}\leq B_i, (1\leq i\leq m+1)\end{subarray}$}}\gamma\varepsilon + \int_{\Omega_d}
	F\left(x_d;\gamma, \balpha, \bbeta\right)\widehat{\mu}_d^{(n)}(dx_d),
	\end{equation*}
	which completes the induction. Moreover, note that \eqref{eq.iteration_Bi} holds for all $1\leq m\leq d-1$, we have $B_{m}\leq \left(1+\frac{l}{\tau}\right)^{m-1}B_1$ via a simple induction argument. To sum up, we have 
	\begin{equation}\label{eq.In_B1_to_B_d}
	I_n(\varepsilon) = \inf_{\mbox{$\begin{subarray}{c} \left(  \gamma, \boldsymbol{\alpha},   \boldsymbol{\beta} \right) \in\calS_{\widehat{\bmu}^n_{1:d-1}},\\
			\gamma\leq\frac{4D}{\varepsilon}, \|\alpha_i\|_{\infty}\leq B_i, (1\leq i\leq d)\end{subarray}$}}\gamma\varepsilon + \int_{\Omega_d}
	F\left(x_d;\gamma, \balpha, \bbeta\right)\widehat{\mu}_d^{(n)}(dx_d).
	\end{equation}
	Now we are able to give a lower bound of the LHS of \eqref{eq.RHS_discretize_approximation}. Define
	\begin{equation*}
    I_{n}(\varepsilon,\delta) := \sup_{\pi\in\mathcal{\calM}_{n}(\varepsilon,\delta
)}\mathbb{E}_{\pi}\left[  f(\bX)\right],
\end{equation*}
where 
\begin{equation*}
\mathcal{M}_{n}( \varepsilon, \delta) :=
\left\lbrace \pi\in\mathcal{M}\left(\Omega
;\delta\right)  :
d\left(\pi, \Pi\left(\boldsymbol{\widehat\mu}^{n}\right)\right)\leq\varepsilon\right\rbrace.
\end{equation*}
Directly from Theorem \ref{thm.strongdual}, one can show the strong duality also holds for  $I_n\left(\varepsilon + \frac{ld^{1/2}}{N},  \frac{l}{N}\right)$. Since the proof can be obtained by simply repeating the arguments in the proof of Theorem \ref{thm.strongdual}, we omit the details and just state the result as following
	\begin{equation*}
	I_n\left(\varepsilon + \frac{ld^{1/2}}{N},  \frac{l}{N}\right) = \inf_{\left(  \gamma, \boldsymbol{\alpha},   \boldsymbol{\beta} \right) \in\calS_{\widehat{\bmu}^n_{1:d-1}}}\gamma\left(\varepsilon + \frac{ld^{3/2}}{N}\right)+ \int_{\Omega_d}
	\widetilde{F}\left(x_d;\gamma, \balpha, \bbeta\right)\widehat{\mu}_d^{(n)}(dx_d),
	\end{equation*}
	where
	\begin{equation*}
	\begin{aligned}
	\widetilde{F}\left(x_d;\gamma, \balpha, \bbeta\right) : = \sup_{\mbox{$\begin{subarray}{c} \bx_{1:d-1}\in\Omega_{1:d-1}, \bx'\in\Omega\end{subarray}$}}\left\lbrace f\left(\bx'\right) + \sum_{k=1}^{d-1}\frac{l}{N}|\alpha_k(\bx'_{1:k})| + \sum_{k=1}^{d-1}\alpha_k(\bx_{1:k}')(x_{k+1}'-x_k') \right.\\
	\left. + \sum_{k=1}^{d-1}\beta_k(x_k)-\gamma\sum_{k=1}^{d}|x_k-x_k'| \right\rbrace.
	\end{aligned}
	\end{equation*}
	By equation \eqref{eq.In_B1_to_B_d}, for any $\delta >0$, we can pick
	\begin{equation*}
	\left(\gamma_{\delta},\balpha_{\delta}, \bbeta_{\delta}\right)\in\calS_{\widehat{\bmu}^n_{1:d-1}}\cap\left\lbrace \gamma \leq \frac{4D}{\varepsilon},\quad  \|\alpha_i\|_{\infty}\leq B_i,    (1\leq i\leq d-1)\right\rbrace,
    \end{equation*}
	such that
	\begin{equation*}
	\gamma_{\delta}\varepsilon + \int_{\Omega_d}
	F\left(x_d;\gamma_{\delta}, \balpha_{\delta}, \bbeta_{\delta}\right)\widehat\mu_d^{(n)}(dx_d)\leq I_n( \varepsilon)  + \delta.
	\end{equation*}
	Thus, 
	\begin{align*}
	& I_{n,N}\left(\varepsilon + \frac{ld^{1/2}}{N},  \frac{l}{N}\right)  - I_n(\varepsilon)\\
	\leq \quad &I_n\left(\varepsilon + \frac{ld^{1/2}}{N},  \frac{l}{N}\right)  - I_n(\varepsilon) \\
	\leq \quad& \gamma_{\delta}\left(\varepsilon + \frac{ld^{3/2}}{N}\right)+ \int_{\Omega_d} \widetilde{F}(x_d;\gamma_{\delta}, \balpha_{\delta}, \bbeta_{\delta})\widehat{\mu}_d^{(n)}(dx_d) - \left(\gamma_{\delta}\varepsilon + \int_{\Omega_d}
	F\left(x_d;\gamma_{\delta}, \balpha_{\delta}, \bbeta_{\delta}\right)\widehat{\mu}_d^{(n)}(dx_d)\right) + \delta\\
	\leq \quad& \frac{4D}{\varepsilon}\cdot \frac{ld^{1/2}}{N} + \int_{\Omega_d}\left(\widetilde{F}\left(x_d;\gamma_{\delta}, \balpha_{\delta}, \bbeta_{\delta}\right) - F\left(x_d;\gamma_{\delta}, \balpha_{\delta}, \bbeta_{\delta}\right)\right)\widehat{\mu}_d^{(n)}(dx_d) + \delta\\
	\leq \quad& \frac{4D}{\varepsilon}\cdot \frac{ld^{1/2}}{N} + \frac{l}{N}\sum_{i=1}^{d-1}B_i + \delta\\
	\leq \quad& \frac{1}{N}\left[\frac{4Dld^{1/2}}{\varepsilon} + \left(1+\frac{l}{\tau}\right)^{d-1}\left(\frac{4dl}{\varepsilon} + 2\right)D\right] + \delta.
	\end{align*}
	Sending $\delta\rightarrow 0$, and combining with \eqref{eq.RHS_discretize_approximation}, we have
	\begin{equation*}
	\left|I_n(\varepsilon) - I_{n, N}\left(\varepsilon + \frac{ld^{1/2}}{N},  \frac{l}{N}\right)\right|\leq\frac{C(f,\varepsilon, d, \Omega)}{N},
	\end{equation*}
	where 
	\begin{equation*}
	 C(f,\varepsilon, d, \Omega):= \max\left\lbrace \left[\frac{4ld^{1/2}}{\varepsilon} + \left(1+\frac{l}{\tau}\right)^{d-1}\left(\frac{4ld}{\varepsilon} + 2\right)\right]D, Lld^{1/2}\right\rbrace.   
	\end{equation*}

\end{proof}

\subsection{Proof of Technical Results Supporting Theorem
\ref{thm.discretize_approximation}} \label{sect.proof_discretize_lemma}

In this section, we provide the full proof of Lemma \ref{lemma.discretize}.

\begin{proof}[Proof of Lemma \ref{lemma.discretize}]
	Let us  introduce some useful notations. For any $1\leq i\leq d$, $k_i\in \ZZ$, we define
	\begin{equation*}
	I_{i,k_i}^{(-1)} := \left[a_i + \frac{(k_i - 1)l_i}{N}, a_i + \frac{k_il_i}{N}\right], \qquad I_{i,k_i}^{(+1)} := \left[a_i + \frac{k_il_i}{N}, a_i + \frac{(k_i + 1)l_i}{N}\right].
	\end{equation*}
	And
	\begin{equation*}
	g_{i,k_i}^{(-1)}(x_i):= \frac{1}{l_i}\left(Nx_i - Na_i - (k_i - 1)l_i\right), \qquad g_{i,k_i}^{(+1)}(x_i) := \frac{1}{l_i}\left(Na_i + (k_i + 1)l_i - Nx_i\right).
	\end{equation*}
	Note that $g_{i,k_i}^{(-1)}(x_i)$ and $g_{i,k_i}^{(+1)}(x_i)$ are non-negative, and we also have $g_{i,k_i}^{(-1)}(x_i) + g_{i,k_i-1}^{(+1)}(x_i) =1$. Let $\pi$ be a martingale probability measure on $\Omega$, we define a probability measure $\pi^{(N)}$ supported on $\Omega$ as follows:
	\begin{equation}\label{eq.discrete_pi}
	\pi^{(N)}\left(a_1+\frac{k_1l_1}{N},\cdots, a_d+\frac{k_dl_d}{N}\right) := \sum_{(t_1,\cdots, t_d)\in\left\lbrace -1, +1\right\rbrace^d}\int_{\prod\limits_{i=1}^dI_{i,k_i}^{(t_i)}}\prod\limits_{i=1}^dg_{i,k_i}^{(t_i)}(x_i)\pi(d\bx_{1:d}), 
	\end{equation}
	for all $0\leq k_i\leq N$, $1\leq i\leq d$. We first verify that $\pi^{(N)}$ is indeed a probability measure. Fix $k_2,\cdots, k_d$ and sum over $k_1$ yields (note that $\pi \equiv 0$ on $\Omega^c$)
	\begin{align*}
	&\sum_{0\leq k_1 \leq N} \pi^{(N)}\left(a_1+\frac{k_1l_1}{N},\cdots, a_d+\frac{k_dl_d}{N}\right) \\
	= & \sum_{(t_2,\cdots, t_d)\in\left\lbrace -1, +1\right\rbrace^{d-1}}\int_{\prod\limits_{i=2}^dI_{i,k_i}^{(t_i)}}\prod\limits_{i=2}^dg_{i,k_i}^{(t_i)}\left[\sum_{ k_1}\left(\int_{I_{1,k_1}^{(-1)}}g_{1,k_1}^{(-1)}(x_1)+ \int_{I_{1,k_1}^{(+1)}}g_{1,k_1}^{(+1)}(x_1)\right)\pi(d\bx_{1:d})\right]\\
	=& \sum_{(t_2,\cdots, t_d)\in\left\lbrace -1, +1\right\rbrace^{d-1}}\int_{\prod\limits_{i=2}^dI_{i,k_i}^{(t_i)}}\prod\limits_{i=2}^dg_{i,k_i}^{(t_i)}\left[\sum_{k_1}\int_{I_{1,k_1}^{(-1)}}\left(g_{1,k_1}^{(-1)}(x_1) + g_{1,k_1-1}^{(+1)}(x_i)\right)\pi(d\bx_{1:d})\right]\\
	=& \sum_{(t_2,\cdots, t_d)\in\left\lbrace -1, +1\right\rbrace^{d-1}}\int_{\prod\limits_{i=2}^dI_{i,k_i}^{(t_i)}}\prod\limits_{i=2}^dg_{i,k_i}^{(t_i)}\left[\sum_{k_1}\int_{I_{1,k_1}^{(-1)}}\pi(d\bx_{1:d})\right]\\
	=& \sum_{(t_2,\cdots, t_d)\in\left\lbrace -1, +1\right\rbrace^{d-1}}\int_{\prod\limits_{i=2}^dI_{i,k_i}^{(t_i)}}\prod\limits_{i=2}^dg_{i,k_i}^{(t_i)}(x_i)\pi_{2:d}(d\bx_{2:d}),
	\end{align*}
	where $\pi_{2:d}(d\bx_{2:d}) = \int_{\Omega_1}\pi(d\bx_{1:d})$ is also a probability measure.
	Therefore we can sum over $k_i$ $(1\leq i\leq d)$
	iteratively and get
	\begin{equation*}
	\sum_{1\leq i\leq d, 0\leq k_i \leq N} \pi^{(N)}\left(a_1+\frac{k_1l_1}{N},\cdots, a_d+\frac{k_dl_d}{N}\right) = 1.
	\end{equation*}
	Furthermore, from the above arguments we have the $i$th $(1\leq i\leq d)$ marginal of $\pi^{(N)}$ is
	\begin{equation*}
	\pi_i^{(N)}\left(a_i + \frac{k_il_i}{N}\right) = \int_{I_{i,k_i}^{(-1)}}g_{i,k_i}^{(-1)}(x_i)\pi_i(dx_i) + \int_{I_{i,k_i}^{(+1)}}g_{i,k_i}^{(+1)}(x_i)\pi_i(dx_i), \qquad 0\leq k_i\leq N,
	\end{equation*}
	where $\pi_i$ is the $i$th marginal of $\pi$. Next we show that $\pi$ and $\pi^{(N)}$ are close in Wasserstein distance. Note that for any 1$-$Lipschitz function $\phi$, we have
	{\small{
			\begin{align*}
			&\int \phi d\pi^{(N)} - \int\phi d\pi \\
			= &\sum_{1\leq i\leq d, 0\leq k_i \leq N} \phi\left(a_1+\frac{k_1l_1}{N}, \cdots, a_d+\frac{k_dl_d}{N}\right)\pi^{(N)}\left(a_1+\frac{k_1l_1}{N},\cdots, a_d+\frac{k_dl_d}{N}\right) - \int\phi d\pi\\
			= & \sum_{k_1,\cdots, k_d}\int_{\prod\limits_{i=1}^dI_{i,k_i}^{(-1)}}\sum_{(s_1,\cdots,s_d)\in\{0,1\}^d}\prod\limits_{i=1}^dg_{i,k_i-s_i}^{(2s_i-1)}(x_i)\left[\phi\left(\left\lbrace a_i+\frac{(k_i-s_i)l_i}{N}\right\rbrace_{i=1}^d\right) -\phi\left(\left\lbrace x_i\right\rbrace_{i=1}^d\right)\right]\pi(d\bx_{1:d}),
			\end{align*}
	}} 
    where $\sum_{(s_1,\cdots,s_d)\in\{0,1\}^d}\prod\limits_{i=1}^dg_{i,k_i-s_i}^{(2s_i-1)}(x_i) = \prod_{i=1}^d\left(g_{i,k_i}^{(-1)}(x_i) + g_{i,k_i-1}^{(+1)}(x_i)\right) = 1$ is used in the last step.
	
	\noindent For any $(x_1,\cdots, x_d)\in \prod\limits_{i=1}^dI_{i,k_i}^{(-1)}$ and $(s_1,\cdots,s_d)\in\{0,1\}^d$, by Lipschitz property of $\phi$ we get
	\begin{equation*}
	\left|\phi\left(\left\lbrace a_i+\frac{(k_i-s_i)l_i}{N}\right\rbrace_{i=1}^d\right) -\phi\left(\left\lbrace x_i\right\rbrace_{i=1}^d\right)\right|\leq \frac{l\sqrt{d}}{N},
	\end{equation*}
	where $l = l_d = \max_{1\leq i\leq d}l_i$.
	Consequently
	\begin{align*}
	\calW\left(\pi^{(N)},\pi\right) &= \inf_{\phi\in\textrm{Lip}_1}\left\lbrace \int \phi d\pi^{(N)} - \int\phi d\pi \right\rbrace\\
	&\leq \frac{l\sqrt{d}}{N}\sum_{k_1,\cdots, k_d}\int_{\prod\limits_{i=1}^dI_{i,k_i}^{(-1)}}\sum_{(s_1,\cdots,s_d)\in\{0,1\}^d}\prod\limits_{i=1}^dg_{i,k_i-s_i}^{(2s_i-1)}(x_i)\pi(d\bx_{1:d})\\
	&\leq \frac{l\sqrt{d}}{N}.
	\end{align*}
	Finally, we check that $\pi^{(N)}$ is a $l/N-$ martingale measure, that is, for all $1\leq j\leq d-1$, $0\leq k_j\leq N$,
	{\small
		\begin{equation}\label{eq.delta_martingale}
		\left|\sum_{0\leq k_{j+1}\leq N}\left[\left(a_{j+1} + \frac{k_{j+1}l_{j+1}}{N}\right) - \left(a_j + \frac{k_jl_j}{N}\right)\right]\frac{\pi^{(N)}_{1:j+1}\left(a_1+\frac{k_1l_1}{N},\cdots, a_{j+1}+\frac{k_{j+1}l_{j+1}}{N}\right)}{\pi^{(N)}_{1:j}\left(a_1+\frac{k_1l_1}{N},\cdots, a_{j}+\frac{k_{j}l_{j}}{N}\right)}\right|\leq \frac{l}{N}.
		\end{equation}
	}
	To prove \eqref{eq.delta_martingale}, observe that
	\begingroup
	\allowdisplaybreaks
	\begin{align*}
	    &\sum_{0\leq k_{j+1}\leq N}\left(a_{j+1} + \frac{k_{j+1}l_{j+1}}{N}\right) \pi^{(N)}_{1:j+1}\left(a_1+\frac{k_1l_1}{N},\cdots, a_{j+1}+\frac{k_{j+1}l_{j+1}}{N}\right)\\
		= \quad & \sum_{(t_1,\cdots, t_j)\in\left\lbrace -1, +1\right\rbrace^{j}}\int_{\prod\limits_{i=1}^jI_{i,k_i}^{(t_i)}}\prod\limits_{i=1}^jg_{i,k_i}^{(t_i)}\left[\sum_{0\leq k_{j+1}\leq N}\left(\int_{I_{j+1,k_{j+1}}^{(-1)}}g_{j+1,k_{j+1}}^{(-1)}(x_{j+1})\left(a_{j+1} + \frac{k_{j+1}l_{j+1}}{N}\right) +\right.\right.\\
		& \qquad\qquad\qquad\qquad\quad\left.\left. \int_{I_{j+1,k_{j+1}}^{(+1)}}g_{j+1,k_{j+1}}^{(+1)}(x_{j+1})\left(a_{j+1} + \frac{k_{j+1}l_{j+1}}{N}\right)\right)\pi_{1:j+1}(d\bx_{1:j+1})\right]\\
		= \quad & \sum_{(t_1,\cdots, t_j)\in\left\lbrace -1, +1\right\rbrace^{j}}\int_{\prod\limits_{i=1}^jI_{i,k_i}^{(t_i)}}\prod\limits_{i=1}^jg_{i,k_i}^{(t_i)}\left[\sum_{k_{j+1}}\int_{I_{j+1,k_{j+1}}^{(-1)}}\left(g_{j+1,k_{j+1}}^{(-1)}(x_{j+1})\left(a_{j+1} + \frac{k_{j+1}l_{j+1}}{N}\right) +\right.\right.\\
		& \qquad\qquad\qquad\qquad\quad\left.\left. g_{j+1,k_{j+1}-1}^{(+1)}(x_{j+1})\left(a_{j+1} + \frac{(k_{j+1}-1)l_{j+1}}{N}\right)\right)\pi_{1:j+1}(d\bx_{1:j+1})\right]\\
		= \quad & \sum_{(t_1,\cdots, t_j)\in\left\lbrace -1, +1\right\rbrace^{j}}\int_{\prod\limits_{i=1}^jI_{i,k_i}^{(t_i)}}\prod\limits_{i=1}^jg_{i,k_i}^{(t_i)}\left[\sum_{k_{j+1}}\int_{I_{j+1,k_{j+1}}^{(-1)}}x_{j+1}\pi_{1:j+1}(d\bx_{1:j+1})\right]\\
		= \quad & \sum_{(t_1,\cdots, t_j)\in\left\lbrace -1, +1\right\rbrace^{j}}\int_{\prod\limits_{i=1}^jI_{i,k_i}^{(t_i)}}\prod\limits_{i=1}^jg_{i,k_i}^{(t_i)}(x_i)\left(\int_{\Omega_{j+1}}x_{j+1}\pi_{1:j+1}(d\bx_{1:j+1})\right).
	\end{align*}
	\endgroup
	Since $\pi$ is a martingale measure, we have $\EE_{\pi}\left[\alpha(\bX_{1:j})(X_{j+1}-X_j)\right] = 0$ holds for all $\calF_{j}-$measurable function $\alpha$. In particular, we have
	\begin{equation*}
	\sum_{(t_1,\cdots, t_j)\in\left\lbrace -1, +1\right\rbrace^{j}}\int_{\prod\limits_{i=1}^jI_{i,k_i}^{(t_i)}\times\Omega_{j+1}}\prod\limits_{i=1}^jg_{i,k_i}^{(t_i)}(x_i)\left(x_{j+1} - x_{j}\right)\pi_{1:j+1}(d\bx_{1:j+1}) = 0.
	\end{equation*}
	Hence,
	\begin{align*}
	&\left|\sum_{0\leq k_{j+1}\leq N}\left[\left(a_{j+1} + \frac{k_{j+1}l_{j+1}}{N}\right) - \left(a_j + \frac{k_jl_j}{N}\right)\right]\pi^{(N)}_{1:j+1}\left(a_1+\frac{k_1l_1}{N},\cdots, a_{j+1}+\frac{k_{j+1}l_{j+1}}{N}\right)\right|\\
	=\quad&\left|\sum_{(t_1,\cdots, t_j)\in\left\lbrace -1, +1\right\rbrace^{j}}\int_{\prod\limits_{i=1}^jI_{i,k_i}^{(t_i)}}\prod\limits_{i=1}^jg_{i,k_i}^{(t_i)}(x_i)\left(x_j- \left(a_j + \frac{k_jl_j}{N}\right)\right)\pi_{1:j}(d\bx_{1:j})\right|\\
	\leq\quad& \frac{l_j}{N} \sum_{(t_1,\cdots, t_j)\in\left\lbrace -1, +1\right\rbrace^{j}}\int_{\prod\limits_{i=1}^jI_{i,k_i}^{(t_i)}}\prod\limits_{i=1}^jg_{i,k_i}^{(t_i)}(x_i)\pi_{1:j}(d\bx_{1:j})\\
	\leq\quad& \frac{l_j}{N} \pi^{(N)}_{1:j}\left(a_1+\frac{k_1l_1}{N},\cdots, a_{j}+\frac{k_{j}l_{j}}{N}\right),
	\end{align*}
	which implies \eqref{eq.delta_martingale}.
\end{proof}

\section{Proofs of Theorem \ref{thm.strongdual_general}, \ref{thm.distance_n_general} and \ref{thm.non_compact_final_step}}\label{sect.proofs_non_compact}

\subsection{Proof of Lemma \ref{lemma.skorokhod_embedding} and  \ref{lemma.reduce_to_compact}}\label{proof.skorokhod_embedding}

We first prove Lemma \ref{lemma.skorokhod_embedding},  which is based on the Skorokhod's embedding.

\begin{proof}[Proof of Lemma \ref{lemma.skorokhod_embedding}]
We first construct the stopping time $T$ via the standard arguments in Skorokhod Embedding. Let $(U,V)$ be a random vector independent of the Brownian motion $(B_t)_{t\geq 0}$, satisfying:
\begin{equation*}
dF_{U,V}(u,v) = \frac{2}{\EE |X|}(v-u)dF(u)dF(v), \qquad u\leq 0< v,
\end{equation*}
where $F$ is the cdf of $X$. Now we define the stopping time $T = \inf\left\lbrace t>0: B_t\notin(U,V)\right\rbrace$, it is straight forward to see that  $B_T$ follows the same law as $X$. For any $k>0$, we define the the following stopping time $T'$:
\begin{equation*}
T' = \left\{\begin{aligned}
\inf\left\lbrace t>0: B_t\notin(U,V) \right\rbrace, \quad & \textrm{If } U>-k, V< k\\
\inf\left\lbrace t>0: B_t\notin(U,-U) \right\rbrace, \quad &\textrm{If } U>-k, V\geq k\\
\inf\left\lbrace t>0: B_t\notin(-V,V) \right\rbrace, \quad &\textrm{If } U\leq -k, V< k\\
\inf\left\lbrace t>0: B_t\notin(-k,k) \right\rbrace, \quad &\textrm{If } U\leq-k, V\geq k 
\end{aligned}\right.
\end{equation*}
By the construction of $T'$, $|B_{T'}|$ is always bounded by the prefixed $k$. Moreover, 
we have $B_{T'} = B_T$ conditioning on the event $\left\lbrace U > -k, V<k\right\rbrace $, and $|B_{T'}|\leq\min\left\lbrace |U|, |V|\right\rbrace \leq |B_T| $ on $\left\lbrace U > -k, V<k\right\rbrace^c $. Hence $|B_{T'}|\leq |B_T|\wedge k$. Next we upper bound the $\EE|B_T-B_{T'}|$.

Observe that $B_{T'} - B_T = 0$ on the event $\left\lbrace U > -k, V<k\right\rbrace $, and $|B_T-B_{T'}|\leq V-U $ on $\left\lbrace U > -k, V<k\right\rbrace^c $. Hence,
\begin{align*}
\EE |B_T- B_{T'}| &= \EE \left[\EE\left[ |B_T-B_{T'}|  U, V\right]\right]\\
&= \frac{2}{\EE|X|}\left[\int_{-\infty}^0dF(u)\int_k^{\infty}(v-u)^2dF(v) + \int_{-\infty}^{-k}dF(u)\int_0^{\infty}(v-u)^2dF(v)\right]\\
&\leq \frac{4}{\EE|X|}\left[\int_{-\infty}^0\left[\EE[X^2\mathbbm{1}_{\left\lbrace X\geq k\right\rbrace }] + u^2\PP(X\geq k)\right]dF(u) +\right.\\
& \qquad \qquad \qquad  \qquad \qquad \qquad \left.\int_0^{\infty}\left[\EE[X^2\mathbbm{1}_{\left\lbrace X\leq -k\right\rbrace }] + v^2\PP(X\leq -k)\right]dF(v)\right]\\
&\leq \frac{4}{\EE|X|} \left[\EE[X^2\mathbbm{1}_{\left\lbrace |X|\geq k\right\rbrace }] + \PP(|X|\geq k)\EE[X^2] \right]\\
&\leq \frac{4}{\EE|X|}\left[\frac{2}{t}\EE\left[e^{\frac{t}{2}|X|^2}\mathbbm{1}_{\{|X|\geq k\}}\right] + \frac{1}{e^{tk^2}}\EE\left[e^{t|X|^2}\right]\cdot\frac{1}{t}\EE\left[e^{t|X|^2}\right]\right] \\
&\leq \frac{4}{c}\left(\frac{2C}{te^{tk^2/2}} + \frac{C^2}{te^{tk^2}}\right).
\end{align*}
Hence, by picking $k:= \sqrt{\frac{2}{t}\log\left(\frac{8C^2}{ct\delta}\right)}>0$, we have $\EE|B_T - B_{T'}|\leq \delta$. 
\end{proof}
Now, we can provide the complete proof of Lemma  \ref{lemma.reduce_to_compact}.
\begin{proof}[Proof of Lemma \ref{lemma.reduce_to_compact}]
For any $\pi\in\calM(\varepsilon,K)$, let $(X_1,\cdots, X_d)$ be the martingale process which the underlying measure is $\pi$. We first recover the construction of the Skorokhod Embedding of martingale, that is, there exists some sequence of increasing stopping time $T_1\leq \cdots\leq T_d$  for a Brownian Motion $(B_t)_{t\geq 0}$ such that:
\begin{equation}\label{eq.strassen_embedding}
(X_1,\cdots, X_d) \stackrel{d}{=} \left(B_{T_1},\cdots, B_{T_d}\right).
\end{equation}

We construct the sequence of stopping time $\{T_k\}_{k=1}^d$ recursively. By the construction of Lemma \ref{lemma.skorokhod_embedding}, there exist a stopping time $T_1$ such that $B_{T_1}\stackrel{d}{=}X_1$. 
Suppose now we have $(X_1,\cdots,X_k)\stackrel{d}{=} \left(B_{T_1},\cdots, B_{T_k}\right)$ for some $k\geq 1$. The strong Markov property implies that $\left(B_t^{(T_k)}:= B_{T_k+t} - B_{T_k}\right)_{t\geq 0}$ is a Brownian Motion that is independent of $\calF_{T_k}$. Let $\mu_k(X_1,\cdots, X_k;\cdot)$ be the regular conditional distribution of $X_{k+1} - X_k$ given $X_i, i\leq k$. The mean of $\mu_k(X_1,\cdots, X_k;\cdot)$ equals zero almost surely by the definition of martingale. Applying Lemma \ref{lemma.skorokhod_embedding} to $\mu_k(X_1,\cdots, X_k;\cdot)$ we see that there is a stopping time $\tau_{k+1}$ such that
\begin{equation*}
B_{\tau_{k+1}}^{(T_{k})} \stackrel{d}{=} \mu_k(X_1,\cdots, X_k;\cdot).
\end{equation*}
Thus by taking the stopping time $T_{k+1} = T_k + \tau_{k+1}$ we have $(X_1,\cdots, X_{k+1}) \stackrel{d}{=} \left(B_{T_1},\cdots, B_{T_{k+1}}\right)$ and the result in \eqref{eq.strassen_embedding} follows by induction.

Under the same Brownian Motion, we show that there exist a sequence of increasing stopping time $T_1'\leq \cdots \leq T_d'$, such that 
\begin{enumerate}
\item [(i)] $\left(B_{T_1'},\cdots, B_{T_d'}\right)$ is a martingale, and each marginal has bounded support. 
\item [(ii)] $\EE |B_{T_k} - B_{T_k'}|\leq \frac{\delta}{d}$, for all $1\leq k\leq d$. 
\end{enumerate}
Notice that $\left(B_{T_1},\cdots, B_{T_d}\right)\sim\pi$, and so $B_{T_i}\sim \pi_i$. Since $\pi\in\calMM(\varepsilon)$, we have $\sum_{i=1}^d\calW(\mu_i,\pi_i)\leq\varepsilon$. Note that the map $x\mapsto |x|$ is 1-Lipschitz. Thus, 
\begin{equation*}
\big|\EE_{\pi_i}|X| - \EE_{\mu_i}|X|\big|\leq \calW(\mu_i,\pi_i)\leq \varepsilon.
\end{equation*}
For $\varepsilon>0$ sufficiently small, we have $\EE_{\pi_i}|X|$ ($1\leq i\leq d$) are uniformly lower bounded by $\min_{1\leq i\leq d} \EE_{\mu_i}|X| - \varepsilon>0$. 
Hence, by the Lemma \ref{lemma.skorokhod_embedding}, for any $\delta >0$, there exist a sequence of stopping times  $\left\lbrace T_k'\right\rbrace_{k=1}^d $, random variables $\left\lbrace \tau_k' \right\rbrace_{k=2}^{d}$, and constant $C= C'\sqrt{\log(1/\delta)}$ (where $C'$ only depends on $\bmu, \gamma$) such that $T_{k+1}' = T_{k}' + \tau_{k+1}'$ ($1\leq k\leq d-1$), and 
\begin{equation*}
|B_{T_1'}|\leq |B_{T_1}|\wedge C, \qquad \EE|B_{T_1} - B_{T_1'}|\leq \delta/d^2.
\end{equation*}
and 
\begin{equation*}
\big|B_{\tau_{k+1}'}^{(T_k')}\big|\leq \big|B_{\tau_{k+1}}^{(T_k)}\big|\wedge C, \qquad \EE\big|B_{\tau_{k+1}}^{(T_k)} - B_{\tau_{k+1}'}^{(T_k')}\big|\leq \delta/d^2, \qquad \textrm{for all } 1\leq k\leq d-1.
\end{equation*}
By the construction in Lemma \ref{lemma.skorokhod_embedding}, we have for each $k$ that 
\begin{equation*}
|B_{T_k'\wedge t}| \leq |B_{T_1'\wedge t}| + \sum_{i=1}^{k-1}\big|B_{\tau_{i+1}'\wedge (t_i')}^{(T_i')}\big|\leq kC.
\end{equation*}
where $t_i':= (t-T_i')\vee 0$. Hence $\left(B_{T_k'\wedge t}\right)$ is uniformly bounded, by the Optional Stopping Theorem we have  $\left(B_{T_1'},\cdots, B_{T_d'}\right)$ is a martingale, and each marginal has bounded support (since $|B_{T_k}|\leq kC$). Furthermore, we have 
\begin{equation*}
\EE|B_{T_k} - B_{T_k'}|\leq \EE|B_{T_1} - B_{T_1'}| + \sum_{i=1}^{k-1} \EE\big|B_{\tau_{i+1}}^{(T_i)} - B_{\tau_{i+1}'}^{(T_i')}\big|\leq \frac{k\delta}{d^2}\leq \frac{\delta}{d}
\end{equation*}
for all $1\leq k\leq d$. To sum up, $\left(B_{T_1'},\cdots, B_{T_d'}\right)\sim\hat{\pi}$ satisfying (i), (ii) and (iii). Recalled that $\left(B_{T_1},\cdots, B_{T_d}\right)\sim\pi$, since the underlying Brownian Motion are the same, we automatically obtain the coupling of $\pi$ and $\hat{\pi}$, and thus by (ii) we have $\calW(\pi,\hat{\pi})\leq \delta$. Finally, by triangular inequality, 
\begin{equation*}
    d\left(\hat\pi, \Pi(\bmu)\right)\leq \calW(\hat\pi, \pi) + d\left(\pi, \Pi(\bmu)\right)\leq \varepsilon + \delta.
\end{equation*}
Thus, we can pick $\Lambda^{\delta} = \bigtimes_{i=1}^d\left[-iC, iC\right]$ to have $\widehat{\pi}\in \calMM^{\Lambda^{\delta}}(\varepsilon+\delta)$ and $\calW(\pi,\widehat{\pi})\leq\delta$.
\end{proof}
\subsection{Proof of Theorem \ref{thm.strongdual_general}} \label{sect.strong_dual_general_proof}
In this section, we sketch the proof of Theorem \ref{thm.strongdual_general} based on the proof of Theorem \ref{thm.strongdual}.
\begin{proof}[Proof of Theorem \ref{thm.strongdual_general}]
Similar to Theorem \ref{thm.strongdual}, consider the following functional 
	\begin{align*}
	\calL\left(\bpi,\gamma,\eta, \boldsymbol{\alpha} \right) &=\int_{\RR^d\times\Lambda} \left[f\left(\bx'\right)+\sum_{k=1}^{d-1}\alpha_k(\bx'_{1:k})(x_{k+1}'-x_k')+\eta x_1'\right]\bpi(d\bx, d\bx') \\
	&\qquad\qquad\qquad\qquad\qquad +\gamma\left(\varepsilon - \int_{\RR^d\times\Lambda} \|\bx - \bx'\|_1\bpi(d\bx, d\bx')\right)
	\end{align*}
	with
	\begin{align*}
	&\bpi \in \bPi := \{\text{Borel probability measures on $\RR^d\times\Lambda$ with the first $d$ marginals are exactly $\bnu$}\};\\
	&\gamma \geq 0;\\
	&\eta \in \RR;\\
	&\boldsymbol{\alpha}= (\alpha_1,\cdots,\alpha_{d-1}), \textrm{ where } \alpha_k(\cdot)\in C\left(\Lambda_{1:k}\right)\textrm{ for } 1\leq k\leq d-1.
	\end{align*}
	The main difference compare to Theorem \ref{thm.strongdual} is the definition of $\bPi$. Note that for every probability measure $\bpi\in\bPi$, we have $\bpi_i = \bnu_i$, and $\bpi_{d+i}\in\calP(\Lambda_i)$ has compact support, for all $1\leq i\leq d$. As a consequence of Prokhorov's theorem, $\bPi$ is a compact convex subset of the space of Borel probability measures on $\RR^d\times \Lambda$ equipped with the weak topology induced by $C_b(\RR^d\times\Lambda)$. It is also straight forward to check that $\calL$ is both continuous and affine in $\left(\bpi,\gamma, \eta, \balpha \right)$, now we can invoke the Sion's minimax theorem (Lemma \ref{sion}) to obtain
	\begin{align}
	\sup_{\bpi \in \bPi} \inf_{\mbox{$\begin{subarray}{c}\balpha\in C_{\balpha}(\Lambda), \gamma\geq 0, \eta\in\RR\end{subarray}$}} \calL\left(\bpi,\gamma,\eta, \balpha \right) =  \inf_{\mbox{$\begin{subarray}{c}\balpha\in C_{\balpha}(\Lambda), \gamma\geq 0, \eta\in\RR \end{subarray}$}}\sup_{\bpi \in \bPi} \calL\left(\bpi,\gamma,\eta, \balpha \right),
	\end{align}
	where ${C}_{\boldsymbol{\alpha}}(\Lambda) = \bigtimes_{k=1}^{d-1} C(\Lambda_{1:k})$. Follows from the same arguments in the proof of Theorem \ref{thm.strongdual}, we have 
	\begin{equation*}
	    \sup_{\bpi \in \bPi} \inf_{\mbox{$\begin{subarray}{c}\balpha\in C_{\balpha}(\Lambda), \gamma\geq 0\end{subarray}$}} \calL\left(\bpi,\gamma, \eta, \balpha \right) = \bI^{\Lambda}_{\bnu}(\varepsilon).
	\end{equation*}
	For $\bpi\in\bPi$, let $\pi_{1:d}$ be the projection of $\bpi$ onto its first $d$ marginals. Then, $\pi_{1:d}\in\Pi(\bnu)$. Let $\bpi(d\bx' | \bx)$ be the regular conditional probability measure. Next, we rewrite $\calL(\bpi,\gamma, \eta, \balpha)$ by
	\begin{align*}
	    \calL(\bpi, \gamma, \eta, \balpha) = \gamma\varepsilon + \int_{\RR^d}\left[ \int_{\Lambda} \left[f\left(\bx'\right)+\sum_{k=1}^{d-1}\alpha_k(\bx'_{1:k})(x_{k+1}'-x_k')+\eta x_1' - \right.\right. \\
	    \left.\left.\gamma \|\bx-\bx'\|_1\right]\bpi(d\bx'|\bx)\right]\bpi_{1:d}(d\bx).
	\end{align*}
By Proposition 2.1 in \cite{model_indep_bound_Beiglbock2013}, for any Lipschitz function $\phi:\RR^d\rightarrow \RR$, we have 
\begin{equation}\label{eq.Kantorovich_dual}
    \sup_{\pi\in\Pi(\bnu)} \int_{\RR^d}\phi(\bx)\pi(d\bx) = \inf_{\beta_k\in C_b(\RR), 1\leq k\leq d}\left\lbrace \sum_{k=1}^d\beta_k(x_k)\bnu_k(dx_k): \sum_{k=1}^d\beta_k(x_k)\geq \phi(\bx)\right\rbrace.
\end{equation}
Thus, 
\begin{align}
    \nonumber&\sup_{\bpi\in\bPi} \calL(\bpi, \gamma, \eta, \balpha) \\
    \nonumber=\quad &\sup_{\bpi\in\bPi} \gamma \varepsilon + \int_{\RR^d}\sup_{\bx'\in\Lambda}\left\lbrace f\left(\bx'\right)+\sum_{k=1}^{d-1}\alpha_k(\bx'_{1:k})(x_{k+1}'-x_k')-\gamma\|\bx-\bx'\|_1+\eta x_1' \right\rbrace\bpi_{1:d}(d\bx) \\
    =\quad &\sup_{\pi\in\Pi(\bnu)}\gamma\varepsilon + \int_{\RR^d}\sup_{\bx'\in\Lambda}\left\lbrace f\left(\bx'\right)+\sum_{k=1}^{d-1}\alpha_k(\bx'_{1:k})(x_{k+1}'-x_k')-\gamma\|\bx-\bx'\|_1+\eta x_1' \right\rbrace\pi(d\bx) \label{eq.sup_over_Lambda}\\
    =\quad &\inf_{\bbeta\in C_{\bbeta}(\RR^d)}\left\lbrace \gamma\varepsilon + \sum_{k=1}^d\int_{\RR}\beta_k(x_k)\nu_k(d x_k): \widetilde{H}(\gamma, \eta, \balpha, \bbeta)(\bx, \bx')\leq 0, \text{ for all } \bx\in\RR^d,  \bx'\in\Lambda\right\rbrace. \label{eq.inf_over_beta}
\end{align}
Here $\bbeta = (\beta_1,\cdots, \beta_d)$, $C_{\bbeta}(\RR^d) = \bigtimes_{i=1}^d C_b(\RR)$, where $\beta_k(x_k)\in C_b(\RR)$ for all $1\leq k\leq d$. We have also applied equation \eqref{eq.Kantorovich_dual} to connect the equality between \eqref{eq.sup_over_Lambda} and \eqref{eq.inf_over_beta}. Finally, 
\begin{align*}
    &\inf_{\mbox{$\begin{subarray}{c}\balpha\in C_{\balpha}(\Lambda), \gamma\geq 0, \eta\in\RR \end{subarray}$}}\sup_{\bpi \in \bPi} \calL\left(\bpi,\gamma,\eta, \balpha \right) \\
    =\quad & \inf_{\mbox{$\begin{subarray}{c}\balpha\in C_{\balpha}(\Lambda), \bbeta\in C_{\bbeta}(\RR^d),\\ \gamma\geq 0, \eta\in\RR \end{subarray}$}}\left\lbrace \gamma\varepsilon + \sum_{k=1}^d\int_{\RR}\beta_k(x_k)\nu_k(d x_k): \widetilde{H}(\gamma, \eta, \balpha, \bbeta)(\bx, \bx')\leq 0, \text{ for all } \bx\in\RR^d,  \bx'\in\Lambda\right\rbrace\\
    =\quad & \bJ^{\Lambda}_{\bnu}(\varepsilon),
\end{align*}
which completes the proof.
\end{proof}
\subsection{Proof of Theorem \ref{thm.distance_n_general}}\label{sect.distance_n_general}
Before we stating the proof of Theorem \ref{thm.distance_n_general}, let's prove the feasibility result in Lemma \ref{lemma.feasibility_non_compact} first.
\begin{proof}[Proof of Lemma \ref{lemma.feasibility_non_compact}]
	It suffices to check that $\calM^{\Lambda^{\delta}}_n(\varepsilon, K)\neq \varnothing$ with high probability. Since 
	\begin{equation*}
	    \sup_{1\leq i\leq d}\EE_{\mu_i}[e^{\gamma |X|^2}] < \infty
	\end{equation*}
	for some $\gamma >0 $,  by Theorem 2 in \cite{fournier2015rate}, there exist universal constant $C_3, C_4 > 0$, such that for any 
	$n\geq \frac{4d^{2}}{C_{4}}\log\left(\frac{d C_{3}}{\delta'}\right)\varepsilon^{-2}(=: N'(\varepsilon, \delta'))$, we have 
	\begin{equation*}
	    \sum_{i=1}^d\calW\left(\widehat{\mu}_i^{(n)},\mu_i\right)\leq\varepsilon/2
	\end{equation*}
	hold with probability at least $1-\delta$. Let  $\pi^*\in\Pi(\bmu)$ be the optimal martingale measure of the MOT problem \eqref{def.standard_MOT}. Similar to the arguments in the proof of Lemma \ref{lemma.feasibility}, there exists a probability measure $\pi'\in\Pi(\widehat\bmu^n)$, such that $d\left(\pi^*, \Pi(\widehat\bmu^n)\right)\leq \calW(\pi^*,\pi') = \sum_{i=1}^d\calW\left(\widehat{\mu}_i^{(n)},\mu_i\right)\leq\varepsilon/2$. Note that by Lemma \ref{lemma.reduce_to_compact}, there exists a $\widehat\pi\in\calM^{\Lambda^{\delta}}(\delta, K)\subset \calM^{\Lambda^{\delta}}(\varepsilon, K)$, such that $\calW(\pi^*, \widehat\pi)\leq \delta$. Thus,
	\begin{equation*}
	    d\left(\widehat\pi, \Pi(\widehat\bmu^n)\right) \leq \calW(\widehat\pi, \pi^*) + d\left(\pi^*, \Pi(\widehat\bmu^n)\right)\leq \delta + \varepsilon/2 \leq \varepsilon.
	\end{equation*}
	In other words, $\widehat\pi\in\calM^{\Lambda^{\delta}}_n(\varepsilon, K)$.
\end{proof}
Now, we can prove Theorem \ref{thm.distance_n_general} as follows.
\begin{proof}[Proof of Theorem \ref{thm.distance_n_general}]
Using the same strategy as Theorem \ref{thm.distance_n}, it suffices to show that
\begin{align*}
&\left|\bI_{\bmu_0}^{\Lambda^{\delta}}(\varepsilon) - \bI_{\bmu_1}^{\Lambda^{\delta}}(\varepsilon)\right|\\
\leq\quad& \frac{4}{\varepsilon}\left(|f(\textbf{0})| + \frac{5}{4}\varepsilon L + L\sum_{i=1}^dm_1(\mu_i)\right)\sqrt{\frac{\log\left(2C_3d/\delta'\right)}{C_4n}},
\end{align*}
where $\bmu_0:= \left(\widehat{\mu}_1^{(n)},\cdots,\widehat{\mu}_d^{(n)}\right)$ and $\bmu_1 := \left(\widehat{\mu}_1^{(n)},\cdots,\widehat{\mu}_{d-1}^{(n)}, \mu_d\right)$. Then, the desired result follows from a triangular inequality. 

By Corollary \ref{cor.dual_representation_general}, 

\begin{align*}
\bI_{\bmu_0}^{\Lambda^{\delta}}(\varepsilon) = 
\inf_{\left(\gamma,\eta,  \boldsymbol{\alpha}, \boldsymbol{\beta}\right)\in\calS^{\Lambda^{\delta}}_{\widehat\bmu^n_{1:d-1}}}\gamma\varepsilon + \int_{\RR}
F\left(x_d; \gamma, \eta, \balpha, \bbeta\right)\widehat{\mu}_d^{(n)}(dx_d).
\end{align*}
Note that $F'$ is $\gamma-$Lipschitz in $x_d$ for any fixed $\gamma$. Similar to Theorem \ref{thm.distance_n}, it suffices to show that $\gamma$ is appropriately bounded. For a constant $B>0$, we define 

\begin{align*}
\bI_{\bmu_0}^{\Lambda^{\delta}, >B}(\varepsilon) :=  \inf_{ \left(\gamma, \eta,  \boldsymbol{\alpha}, \boldsymbol{\beta}\right)\in\calS^{\Lambda^{\delta}}_{\widehat\bmu^n_{1:d-1}}, \gamma > B}\gamma\varepsilon + \int_{\RR}
F\left(x_d; \gamma,\eta,\balpha, \bbeta \right)\widehat{\mu}_d^{(n)}(dx_d).
\end{align*}
We claim that by a proper choice of $B$, we have 
\begin{equation*}\label{eq.bound_gamma_general}
\bI_{\bmu_0}^{\Lambda^{\delta}, >B}(\varepsilon) > \bI_{\bmu_0}^{\Lambda^{\delta}}(\varepsilon).
\end{equation*}
For any $\pi\in\calMM_{\bmu_0}^{\Lambda^{\delta}}\left(\varepsilon\right)$, there exists $\pi'\in\Pi(\bmu_0)$, such that
\begin{equation*}
    \calW(\pi, \pi') = d\left(\pi, \Pi(\bmu_0)\right)\leq\varepsilon. 
\end{equation*}
Follows from the proof of Lemma \ref{lemma.feasibility_non_compact}, for $n\geq N'\left(\frac{\varepsilon}{2}, \frac{\delta}{2d}\right)$, with probability at least $1-\frac{\delta}{2d}$,
\begin{equation*}
    \sum_{i=1}^d\calW\left(\mu_i, \widehat{\mu}_i^{(n)}\right) \leq \frac{\varepsilon}{4}.
\end{equation*}
Thus, for each $1\leq i\leq d$,
\begin{equation*}
    \calW(\pi_i,\mu_i) \leq \calW\left(\pi_i,\widehat{\mu}_i^{(n)}\right)+\calW\left(\widehat{\mu}_i^{(n)},\mu_i\right)=\calW\left(\pi_i,\pi_i'\right)+\calW\left(\widehat{\mu}_i^{(n)},\mu_i\right)\leq \calW(\pi, \pi') + \frac{\varepsilon}{4}\leq \frac{5}{4}\varepsilon.
\end{equation*}
Since $f$ is $L-$Lipschitz, it satisfies a linear growth condition, namely,
\begin{equation*}
|f(\bx)|\leq |f(\textbf{0})| + L\|\bx\|_1.
\end{equation*}
Hence,  
\begin{align*}
\EE_{\pi} \left[f(\bX)\right] &\leq |f(\textbf{0})| + L\sum_{i=1}^d\int_{\RR}|x|\pi_i(dx)\\
&\leq  |f(\textbf{0})|  + L\sum_{i=1}^d\left(\int_{\RR}|x|\mu_i(dx)+\calW\left(\pi_i,\mu_i\right)\right)\\
&\leq |f(\textbf{0})| + L\sum_{i=1}^dm_1(\mu_i) + L\sum_{i=1}^d\left(\calW\left(\pi_i,\widehat{\mu}_i^{(n)}\right)+\calW\left(\widehat{\mu}_i^{(n)},\mu_i\right)\right)\\
&\leq |f(\textbf{0})| + L\sum_{i=1}^dm_1(\mu_i) + \frac{5}{4}\varepsilon L.
\end{align*}
Take the supremum over $\pi$ yields
\begin{equation*}
\left|\bI_{\bmu_0}^{\Lambda^{\delta}}(\varepsilon)\right|\leq  |f(\textbf{0})| + \frac{5}{4}\varepsilon L + L\sum_{i=1}^d m_1(\mu_i) .
\end{equation*}
On the other hand, by the definition of $F$ we have 
\begin{equation*}
\begin{aligned}
\bI_{\bmu_0}^{\Lambda^{\delta}, >B}(\varepsilon) \geq \inf_{ \left(\gamma,\eta,  \boldsymbol{\alpha}, \boldsymbol{\beta}\right)\in\calS^{\Lambda^{\delta}}_{\widehat\bmu^n_{1:d-1}}, \textrm{ }\gamma >B}\gamma\varepsilon  +  \int\left(\int\left( f\left(\bx'\right)- \sum_{k=1}^{d-1}\beta_k(x_k)-\right.\right. \\ 
\left.\left. \gamma\sum_{k=1}^{d}|x_k-x_k'| +  \sum_{k=1}^{d-1}\alpha_k(\bx_{1:k}')(x_{k+1}'-x_k') -\right.\right.\\
\left.\left.  +  \eta x_1'\right)\bpi\left(d\bx_{1:d-1}, d\bx' \mid x_d\right)\right)\widehat{\mu}_d^{(n)}(dx_d)
\end{aligned}
\end{equation*}
hold for every conditional probability measure $\bpi(d\bx_{1:d-1},d\bx'\mid x_d)$ (supported on $\RR^{d-1}\times\Lambda^{\delta}$). Recalled from the Lemma \ref{lemma.reduce_to_compact} that there exists a probability measure $\pi\in\calM_0(\RR^d)$  supported on $\Lambda^{\delta}$ that satisfying  $\sum_{i=1}^d\calW\left(\pi_i,\mu_i\right)\leq \delta\leq \frac{\varepsilon}{4}$. Now we pick the $\bpi(d\bx_{1:d-1}, d\bx' \mid x_d)$ such that its projection on $\bx'$ is $\pi$, and its projection on $\bx_{1:d-1}$ is a $(d-1)-$ dimensional probability measure with marginals $\widehat{\mu}_1^{(n)},\cdots,\widehat{\mu}_{d-1}^{(n)}$. Taking the supremum over these $\bpi$'s, and note that $\pi$ is a martingale measure, we have 

\begin{align*}
\nonumber \bI_{\bmu_0}^{\Lambda^{\delta}, >B}(\varepsilon)&\geq \inf_{\left(\gamma,\eta,  \boldsymbol{\alpha}, \boldsymbol{\beta}\right)\in\calS^{\Lambda^{\delta}}_{\widehat\bmu^n_{1:d-1}}, \textrm{ }\gamma>B}\gamma\varepsilon + \EE_{\pi}\left[f(\bX')\right] - \sum_{k=1}^{d-1}\int_{\RR}\beta_k(x_k)\widehat{\mu}_k^{(n)}(dx_k) + \\ \nonumber
& \sum_{k=1}^{d-1}\EE_{\pi}\left[\alpha_k(\bX_{1:k}')(X_{k+1}'-X_k')\right]  + \eta \int_{\RR}x_1'\pi_1(dx_1')  -\gamma\sum_{k=1}^d\calW\left(\pi_k,\widehat{\mu}_k^{(n)}\right)\\ \nonumber
&\geq \inf_{\left(\gamma, \eta,  \boldsymbol{\alpha}, \boldsymbol{\beta}\right)\in\calS^{\Lambda}_{\widehat\bmu^n_{1:d-1}}, \textrm{ }\gamma>B}\gamma\varepsilon +  \EE_{\pi}\left[f(\bX')\right] - \gamma\sum_{k=1}^d\left(\calW\left(\mu_k,\widehat{\mu}_k^{(n)}\right) + \calW\left(\pi_k,\mu_k\right)\right)\\ \nonumber
&> \EE_{\pi}\left[f(\bX')\right] + \frac{\varepsilon B}{2} . \label{eq.bounded_parameter}
\end{align*}
Therefore, by taking 
\begin{equation*}
    B = \frac{4}{\varepsilon}\left(|f(\textbf{0})| + \frac{5}{4}\varepsilon L + L\sum_{i=1}^dm_1(\mu_i)\right),
\end{equation*}
we have the equation \eqref{eq.bound_gamma_general} hold. In other words, 
\begin{align*}
&\bI_{\bmu_0}^{\Lambda^{\delta}}(\varepsilon) = \inf_{\left(\gamma, \eta,  \boldsymbol{\alpha}, \boldsymbol{\beta}\right)\in\calS^{\Lambda^{\delta}}_{\widehat\bmu^n_{1:d-1}},\textrm{ }\gamma \leq B} \gamma\varepsilon + \int_{\RR}
F'\left(x_d; \gamma, \eta, \balpha, \bbeta \right)\widehat{\mu}_d^{(n)}(dx_d).
\end{align*}
Using the same argument on $I_{\bmu_1}^{\Lambda^{\delta}}(\varepsilon,K)$ gives us 
\begin{align*}
&\bI_{\bmu_1}^{\Lambda^{\delta}}(\varepsilon) = \inf_{\left(\gamma,\eta,  \boldsymbol{\alpha}, \boldsymbol{\beta}\right)\in\calS^{\Lambda^{\delta}}_{\widehat\bmu^n_{1:d-1}},\textrm{ }\gamma\leq B}\gamma\varepsilon + \int_{\RR}
F'\left(x_d; \gamma,\eta, \balpha, \bbeta\right)\mu_d(dx_d).
\end{align*}
The rest follows exactly the same arguments in the proof of Theorem \ref{thm.distance_n}.
\end{proof}
Finally, we prove Proposition \ref{prop.Lip_general_support}.
\begin{proof}[Proof of Proposition \ref{prop.Lip_general_support}]
By Corollary \ref{cor.dual_representation_general}, we have 
\begin{align*}
&\bI^{\Lambda}(\varepsilon) = 
\inf_{\left(\gamma,\eta,  \boldsymbol{\alpha}, \boldsymbol{\beta}\right)\in\calS^{\Lambda}_{\boldsymbol{\mu}_{1:d-1}}}\gamma\varepsilon + \int_{\RR}
F\left(x_d; \gamma,\eta, \boldsymbol{\alpha}, \boldsymbol{\beta}\right)\mu_d(dx_d). 
\end{align*}
Pick $\pi\in\calM_0(\bmu)$. Consider the diagonal coupling $\bpi\in\Pi(\pi, \pi)$, take $B = 2\sup_{\bx\in\Lambda}f(\bx)/R$, we have 
\begin{equation*}
\begin{aligned}
&\quad  \inf_{\left(\gamma,\eta,  \boldsymbol{\alpha}, \boldsymbol{\beta}\right)\in\calS^{\Lambda}_{\boldsymbol{\mu}_{1:d-1}}, \textrm{ }\gamma > B}\gamma\varepsilon  + \int_{\RR}
F\left(x_d; \gamma, \eta, \boldsymbol{\alpha}, \boldsymbol{\beta}\right)\mu_d(dx_d)\\
&\geq \inf_{ \left(\gamma,\eta,  \boldsymbol{\alpha}, \boldsymbol{\beta}\right)\in\calS^{\Lambda}_{\bmu_{1:d-1}}, \textrm{ }\gamma >  B}\gamma\varepsilon  +  \int\Bigg(\int\Bigg( f\left(\bx'\right)- \sum_{k=1}^{d-1}\beta_k(x_k)- \gamma\sum_{k=1}^{d}|x_k-x_k'| -  \sum_{k=1}^{d-1}\alpha_k(\bx_{1:k}')(x_{k+1}'-x_k') \\
&\qquad\qquad\qquad\qquad\qquad\qquad\qquad\qquad\qquad\qquad\qquad\qquad\qquad  +  \eta x_1'\Bigg)\bpi\left(d\bx_{1:d-1}, d\bx' \mid x_d\right)\Bigg)\mu_d(d x_d)\\
&\geq \inf_{ \left(\gamma,\eta,  \boldsymbol{\alpha}, \boldsymbol{\beta}\right)\in\calS^{\Lambda}_{\bmu_{1:d-1}}, \textrm{ }\gamma > B}\gamma\varepsilon + \EE_{\pi}\left[f(\bX')\right] - \gamma \calW(\pi, \pi) + \sum_{k=1}^{d-1}\EE_{\pi}\left[\alpha(\bX'_{1:k})(X'_{k+1} - X'_k)\right] + \EE_{\mu_1}[X]\\
&\geq \varepsilon B + \EE_{\pi}\left[f(\bX)\right] > \sup_{\bx\in\Lambda} f(\bx) \geq \bI^{\Lambda}(\varepsilon). 
\end{aligned}
\end{equation*}
Thus, for any $\varepsilon >R$, 
\begin{equation*}
\bI^{\Lambda}(\varepsilon) = 
\inf_{\left(\gamma,\eta,  \boldsymbol{\alpha}, \boldsymbol{\beta}\right)\in\calS^{\Lambda}_{\boldsymbol{\mu}_{1:d-1}}, \gamma \leq B}\gamma\varepsilon + \int_{\RR}
F\left(x_d; \gamma,\eta, \boldsymbol{\alpha}, \boldsymbol{\beta}\right)\mu_d(dx_d). 
\end{equation*}
As a result, for any $\varepsilon, \varepsilon' > R$, 
\begin{equation*}
\left|\bI^{\Lambda}(\varepsilon) - \bI^{\Lambda}(\varepsilon')\right|  \leq B|\varepsilon - \varepsilon'| = \frac{2\sup_{\bx\in\Lambda}f(\bx)}{R}|\varepsilon - \varepsilon'|.    
\end{equation*}

\end{proof}

\subsection{Proof of Theorem \ref{thm.non_compact_final_step}} \label{sect.proof_discretization_approximation_general}
We first prove an analogy of lemma \ref{lemma.discretize}. 

\begin{lemma}\label{lemma.discretize_general}
Let Assumption \ref{aspn.convexorder} hold. Let $\Lambda = \bigtimes_{i=1}^d[-iC, iC]$ for some constant $C>0$. Given $\varepsilon>
0$, and $\bnu\in\bigtimes_{i=1}^d\calP(\RR)$. Then, for any $\pi\in\calMM_{\bnu}^{\Lambda}(\varepsilon)$ and integer $N$, there exists a $2dC/N-$martingale measure $\pi^{(N)}$ such that 
\begin{enumerate}
\item [\emph{(1)}] For each $1\leq i\leq d$, $\pi_i^{(N)}$ is mean zero measure that supported on $\Lambda_{i}^N:= \{-iC + 2iC/N: 0\leq i\leq N\}$.
\item [\emph{(2)}] $\calW\left(\pi,\pi^{(N)}\right)\leq \frac{2d^{3/2}C}{N}$. 
\end{enumerate}
\end{lemma}
\begin{proof}[Proof of Lemma \ref{lemma.discretize_general}]
Adapting the definitions in the proof of lemma \ref{lemma.discretize}. Given any martingale measure $\pi\in\calMM_{\bnu}^{\Lambda}(\varepsilon)$, let $\pi^{(N)}$ be the probability measure defined in \eqref{eq.discrete_pi} (with $a_i = -iC$, $b_i= iC$). By lemma \ref{lemma.discretize}, there exists a $2dC/N$-martingale measure  $\pi_i^{(N)}$ that is supported on $\Lambda_{i}^N$ and $\calW\left(\pi,\pi^{(N)}\right)\leq \frac{2d^{3/2}C}{N}$. Follows from straight forward calculation, we have the $i$th marginal of $\pi^{(N)}$ satisfies:
\begin{equation*}
\pi_i^{(N)}\left(-iC+\frac{2k_iiC}{N}\right)= \int_{I_{i,k_i}^{(-1)}}g_{i,k_i}^{(-1)}(x_i)\pi_i(dx_i) + \int_{I_{i,k_i}^{(+1)}}g_{i,k_i}^{(+1)}(x_i)\pi_i(dx_i), \qquad 0\leq k_i\leq N,
\end{equation*}
where the definition of $I_{i,k_i}^{(\pm1)}$ and $g_{i,k_i}^{(\pm1)}(x_i)$ follows from the proof of lemma \ref{lemma.discretize}. Finally, we check that the expectation of each $\pi_i^{(N)}$ are zero. By definition, we have 
\begin{align*}
\EE_{\pi_i^{(N)}}\left[X\right]&=\sum_{k_i=0}^N\left(-iC+\frac{2k_iiC}{N}\right)\pi_i^{(N)}\left(-iC+\frac{2k_iiC}{N}\right)\\
&= \sum_{k_i\in\ZZ}\int_{I_{i,k_i}^{(-1)}} \left[g_{i,k_i}^{(-1)}(x)\left(-iC+\frac{2k_iiC}{N}\right)+g_{i,k_i-1}^{(+1)}(x)\left(-iC+\frac{2(k_i-1)iC}{N}\right)\right]\pi_i(dx)\\
&=\sum_{k_i\in\ZZ}\int_{I_{i,k_i}^{(-1)}}x\pi_i(dx) \\
&=\EE_{\pi_i}\left[X\right]= 0.
\end{align*}
\end{proof}
\begin{proof}[Proof of Theorem \ref{thm.non_compact_final_step}]
Let $C = C'\sqrt{\log(1/\delta)}$. Then, $\Lambda^{\delta} = \bigtimes_{i=1}^d[-iC, iC]$. For any $\pi\in \calMM^{\Lambda^{\delta}}_{n}( \varepsilon) $, by Lemma \ref{lemma.discretize_general} we can construct a $2dC/N-$ martingale measure $\pi'$ that supported on $\Lambda^{
\delta, N}$, such that 
\begin{equation*}
\calW\left(\pi,\pi'\right)\leq \frac{2d^{3/2}C}{N}.
\end{equation*}
Note that
\begin{equation*}
d\left(\pi', \Pi(\widehat{\bmu}^n)\right) \leq \calW(\pi', \pi) + d\left(\pi, \Pi(\widehat{\bmu}^n)\right) \leq \varepsilon + \frac{2d^{3/2}C}{N}.
\end{equation*}
Thus, $\pi'\in \calMM_{n,N}^{\Lambda^{\delta}}\left( \varepsilon + \frac{2d^{3/2}C}{N},  \frac{2dC}{N}\right)$. Furthermore,
\begin{equation*}
\begin{aligned}
\EE_{\pi}\left[f(\bX)\right] - \bI_{n,N}^{\Lambda^{\delta}}\left( \varepsilon + \frac{2d^{3/2}C}{N},  \frac{2dC}{N}\right) &\leq \EE_{\pi}\left[f(\bX)\right] - \EE_{\pi'}\left[f(\bX)\right] \\
&\leq L \calW(\pi',\pi)\leq\frac{2d^{3/2}CL}{N}.
\end{aligned}
\end{equation*}
Taking the supremum over $\pi$ yields 
\begin{equation}\label{eq.discrete_approximation_general}
\bI_{n}^{\Lambda^{\delta}}(\varepsilon, K) - \bI_{n,N}^{\Lambda^{\delta}}\left( \varepsilon + \frac{2d^{3/2}C}{N},  \frac{2dC}{N}\right)\leq \frac{2d^{3/2}CL}{N}.
\end{equation}
Next we lower bounded the LHS of \eqref{eq.discrete_approximation_general}. Given $\delta'>0$, by the Corollary \ref{cor.dual_representation_general} and the proof of Theorem \ref{thm.distance_n_general}, for any $n\geq N'(\varepsilon, \delta')$, we have with probability ay least $1-\delta'$ that
\begin{equation}\label{eq.I_n_Omega_delta_bounded_dual_parameter}
\begin{aligned}
&\bI_{n}^{\Lambda^{\delta}}(\varepsilon) \\
&=  \inf_{\mbox{$\begin{subarray}{c} \left(\gamma,\eta,  \boldsymbol{\alpha}, \boldsymbol{\beta}\right)\in\calS^{\Lambda^{\delta}}_{\widehat\bmu^n_{1:d-1}},\\
\gamma\leq B^*\end{subarray}$} }\gamma\varepsilon + \int_{\RR}
F'\left(x_d; \gamma,\eta, \balpha, \bbeta\right)\widehat{\mu}_d^{(n)}(dx_d),
\end{aligned} 
\end{equation}
where, for notational convenience, we define
\begin{align*}
B^* &= \frac{4}{\varepsilon}\left(|f(\textbf{0})| + \frac{5}{4}\varepsilon L + L\sum_{i=1}^dm_1(\mu_i)\right) :=  \frac{4}{\varepsilon}\cdot C(\bmu,\varepsilon).   
\end{align*}
It is also useful to recall that $\bI^{
\Lambda^{\delta}}_n(\varepsilon) \leq C(\bmu, \varepsilon)$. In the rest of the proof, our analysis will concentrate on the event that equation \eqref{eq.I_n_Omega_delta_bounded_dual_parameter} hold. Note that on this event we have 
\begin{equation*}
    \sum_{k=1}^d\calW\left(\mu_k,\widehat\mu^{(n)}_k\right)\leq \varepsilon/2.
\end{equation*}
Recalled that $\Lambda^{\delta}_i = [-iC, iC]$ for $1\leq i\leq d$. We have 

\begin{align}
\nonumber&F'\left(x_d; \gamma,\eta, \balpha, \bbeta \right)\\ 
\nonumber&\geq \sup_{\bx_{1:d-1}\in\RR^{d-1}} \left\lbrace \sup_{\bx'\in\Lambda^{\delta}}\left\lbrace  \sum_{k=1}^{d-1}\alpha_k(\bx_{1:k}')(x_{k+1}'-x_k') + \eta x_1'\right\rbrace   + \inf_{\bx'\in\Lambda}f\left(\bx'\right) -\right.\\
\nonumber& \qquad\qquad\qquad\qquad\qquad\qquad\qquad\qquad\qquad\qquad \left. - \sum_{k=1}^{d-1}\beta_k(x_k) - \gamma\sum_{k=1}^d \sup_{x_k'\in\Lambda_k^{\delta}}|x_k-x_k'|\right\rbrace\\
\nonumber&\geq \sup_{\bx_{1:d-1}\in\RR^{d-1}} \left\lbrace \sup_{\bx'\in\Lambda^{\delta}}\left\lbrace  \sum_{k=1}^{d-1}\alpha_k(\bx_{1:k}')(x_{k+1}'-x_k') + \eta x_1'\right\rbrace  - \sum_{k=1}^{d-1}\beta_k(x_k) - B^*\sum_{k=1}^d |x_k| \right\rbrace-C(\bmu,\varepsilon)\\
\nonumber &\geq \int \left( \sup_{\bx'\in\Lambda^{\delta}}\left\lbrace  \sum_{k=1}^{d-1}\alpha_k(\bx_{1:k}')(x_{k+1}'-x_k') + \eta x_1'\right\rbrace  - \sum_{k=1}^{d-1}\beta_k(x_k) - B^*\sum_{k=1}^d |x_k| \right)\pi(d\bx_{1:d-1})-C(\bmu,\varepsilon)\\
&= \sup_{\bx'\in\Lambda^{\delta}}\left\lbrace  \sum_{k=1}^{d-1}\alpha_k(\bx_k')(x_{k+1}'-x_k') + \eta x_1'\right\rbrace -B^*\sum_{k=1}^{d-1}m_1\left(\widehat{\mu}_k^{(n)}\right) -B^*|x_d|-C(\bmu,\varepsilon), \label{eq.bounded_alpha_eta'}
\end{align}
where $\pi\in\Pi\left(\widehat\bmu^n\right)$ and we have used the fact that $\beta_{k}$'s are in $\calS^{\Lambda^{\delta}}_{\widehat\bmu^n_{1:d-1}}$. Suppose $|\eta|\geq B'$ for some $B'>0$, taking $x_1'=\cdots=x_d'$ yields  
\begin{align*}
&\sup_{\bx'\in\Lambda^{\delta}}\left\lbrace  \sum_{k=1}^{d-1}\alpha_k(\bx_k')(x_{k+1}'-x_k') + \eta x_1'\right\rbrace\geq\sup_{x_1'\in\Lambda^{\delta}_1}\eta x_1'\geq B'C.
\end{align*}
Now we arrive at
\begin{align*}
&\inf_{\mbox{$\begin{subarray}{c} \left(\gamma,\eta, \boldsymbol{\alpha}, \boldsymbol{\beta}\right)\in\calS^{\Lambda^{\delta}}_{\widehat\bmu^n_{1:d-1}},\\
\gamma\leq B^*,|\eta|\geq B'\end{subarray}$} }\gamma\varepsilon + \int_{\RR}
F'\left(x_d; \gamma,\eta, \balpha, \bbeta \right)\widehat{\mu}_d^{(n)}(dx_d)\\
&\geq B'C - B^*\sum_{k=1}^{d}m_1\left(\widehat{\mu}_k^{(n)}\right)-C(\bmu, \varepsilon)\\
&\geq B'C - B^*\left[\sum_{k=1}^dm_1(\mu_k) + \frac{\varepsilon}{2}\right]-C(\bmu, \varepsilon).
\end{align*}
Where we have used the fact that $\sum_{k=1}^d\calW\left(\mu_k,\widehat\mu^{(n)}_k\right)\leq \varepsilon/2$ and $x\rightarrow |x|$ is a 1-Lipschitz map. For $B'$ sufficiently large, say 
$$B' = \frac{1}{C}\left[2C(\bmu,\varepsilon)+ B^*\left(\sum_{k=1}^dm_1(\mu_k) + \varepsilon\right)\right], $$
we have 
\begin{equation*}
\inf_{\mbox{$\begin{subarray}{c} \left(\gamma,\eta,\eta, \boldsymbol{\alpha}, \boldsymbol{\beta}\right)\in\calS^{\Lambda^{\delta}}_{\widehat\bmu^n_{1:d-1}}(\varepsilon),\\
0\leq\gamma,\eta\leq B^*,|\eta|\geq B'\end{subarray}$} }\gamma\varepsilon +\eta K + \int_{\RR}
F'\left(x_d; \gamma,\eta,\eta, \balpha, \bbeta \right)\widehat{\mu}_d^{(n)}(dx_d) > C(\bmu,\varepsilon) > I^{\Lambda^{\delta}}_n(\varepsilon, K),
\end{equation*}
which implies the optimal $\eta$ is essentially bounded by $B'$. That is,
\begin{equation*}
\begin{aligned}
&\bI_{n}^{\Lambda^{\delta}}(\varepsilon) \\
&=  \inf_{\mbox{$\begin{subarray}{c} \left(\gamma,\eta,  \boldsymbol{\alpha}, \boldsymbol{\beta}\right)\in\calS^{\Lambda^{\delta}}_{\widehat\bmu^n_{1:d-1}},\\
\gamma\leq B^*, |\eta|\leq B'\end{subarray}$} }\gamma\varepsilon + \int_{\RR}
F'\left(x_d; \gamma,\eta, \balpha, \bbeta\right)\widehat{\mu}_d^{(n)}(dx_d),
\end{aligned} 
\end{equation*}
Deduced from \eqref{eq.bounded_alpha_eta'}, we have 
\begin{align*}
&\int_{\RR} F'\left(x_d; \gamma,\eta, \balpha, \bbeta \right)\widehat\mu^{(n)}_d(d x_d)\\
&\geq  \sup_{\bx'\in\Lambda^{\delta}}\left\lbrace  \sum_{k=1}^{d-1}\alpha_k(\bx_{1:k}')(x_{k+1}'-x_k') \right\rbrace -B'C-B^*\sum_{k=1}^{d}m_1\left(\widehat{\mu}_k^{(n)}\right)
-C(\bmu,\varepsilon)\\
&\geq \sup_{\bx'\in\Lambda^{\delta}}\left\lbrace  \sum_{k=1}^{d-1}\alpha_k(\bx_{1:k}')(x_{k+1}'-x_k') \right\rbrace -B'C-B^*\left[\sum_{k=1}^{d}m_1\left(\mu_k\right)+ \varepsilon/2\right]-C(\bmu,\varepsilon)\\
&\geq \sup_{\bx'\in\Lambda^{\delta}}\left\lbrace  \sum_{k=1}^{d-1}\alpha_k(\bx_{1:k}')(x_{k+1}'-x_k') \right\rbrace - 2B'C.
\end{align*}
Now, by the same induction arguments in the proof of Theorem \ref{thm.discretize_approximation}, we can show that 
\begin{align*}
&\bI_{n}^{\Lambda^{\delta}}(\varepsilon) \\
&=  \inf_{\mbox{$\begin{subarray}{c} \left(\gamma,\eta,  \boldsymbol{\alpha}, \boldsymbol{\beta}\right)\in\calS^{\Lambda^{\delta}}_{\widehat\bmu^n_{1:d-1}},\\
\gamma\leq B^*,|\eta|\leq B', \|\alpha_k\|_{\infty}\leq B_k,1\leq k\leq d-1.\end{subarray}$} }\gamma\varepsilon + \int_{\RR}
F'\left(x_d; \gamma,\eta, \balpha, \bbeta \right)\widehat{\mu}_d^{(n)}(dx_d).
\end{align*}
where the constants
\begin{equation*}
B_1 = \frac{1}{C}\left(2B'C+C(\bmu,\varepsilon)\right), \quad B_k\leq (1+2d)^{k-1}B_1 \textrm{ for all } 1\leq k\leq d-1.
\end{equation*}
Define
\begin{equation*}
    \bI^{\Lambda}_n(\varepsilon,\delta) = \sup_{\pi\in\calMM^{\Lambda}_n(\varepsilon,\delta)}\EE_{\pi} [f(\bX)],
\end{equation*}
where 
\begin{equation*}
\calMM^{\Lambda}_{n}( \varepsilon, 
\delta) := \left\lbrace \pi\in\calM_0\left(\Lambda;\delta\right):  d\left(\pi, \Pi\left(\boldsymbol{\widehat\mu}^{(n)}\right)\right)\leq\varepsilon\right\rbrace.
\end{equation*}
By the same argument in the proof of Theorem \ref{thm.discretize_approximation}, we also have the following strong duality 
\begin{align*}
&\bI^{\Lambda^{\delta}}_{n}\left(\varepsilon + \frac{2d^{3/2}C}{N},  \frac{2dC}{N}\right)= \inf_{\mbox{$\begin{subarray}{c} \left(\gamma,\eta,  \boldsymbol{\alpha}, \boldsymbol{\beta}\right)\in\calS^{\Lambda^{\delta}}_{\widehat\bmu^n_{1:d-1}}\end{subarray}$} } \gamma\left(\varepsilon + \frac{2d^{3/2}C}{N}\right)  +  \int_{\RR}
\widetilde{F}'\left(x_d; \gamma,\eta,\balpha, \bbeta\right)\widehat{\mu}_d^{(n)}(dx_d),
\end{align*}
where 
\begin{equation*}
\begin{aligned}
\widetilde{F}'\left(x_d; \gamma,\eta, \balpha, \bbeta\right) : = \sup_{\mbox{$\begin{subarray}{c} \bx\in\RR^{d-1}, \bx'\in\Lambda^{\delta}
\end{subarray}$}}\left\lbrace f\left(\bx'\right)- \sum_{k=1}^{d-1}\beta_k(x_k)-\gamma\sum_{k=1}^{d}|x_k-x_k'| + \right. \\ 
\left.  \sum_{k=1}^{d-1}\frac{2dC}{N}|\alpha_k(\bx_{1:k}')|+\sum_{k=1}^{d-1}\alpha_k(\bx_{1:k}')(x_{k+1}'-x_k') + \eta x_1'\right\rbrace.
\end{aligned}
\end{equation*}
Finally, we apply the same trick in the proof of Theorem \ref{thm.discretize_approximation}. For any $\delta'>0$, we can pick 
\begin{equation*}
    \left(\gamma_{\delta'}, \eta_{\delta'},\balpha_{\delta'}, \bbeta_{\delta'}\right)\in \calS^{\Lambda^{\delta'}}_{\widehat\bmu^n_{1:d-1}}(\varepsilon) \cap \left\lbrace \gamma, \leq B^*, |\eta|\leq B', \|\alpha_i\|_{\infty}\leq B_i, (1\leq i\leq d-1)\right\rbrace,
\end{equation*}
such that
\begin{equation*}
\gamma_{\delta'}\varepsilon + \int_{\RR}
F'\left(x_d; \gamma_{\delta'},\eta_{\delta'}', \balpha_{\delta'}, \bbeta_{\delta'}\right)\widehat{\mu}_d^{(n)}(dx_d)\leq \bI_{n}^{\Lambda^{\delta}}(\varepsilon)+\delta'.
\end{equation*}
Hence, 
\begin{align*}
&\bI^{\Lambda^{\delta}}_{n, N}\left(\varepsilon + \frac{2d^{3/2}C}{N},  \frac{2dC}{N}\right) - \bI_{n}^{\Lambda^{\delta}}(\varepsilon)\\
\leq\quad & \bI^{\Lambda^{\delta}}_{n}\left(\varepsilon + \frac{2d^{3/2}C}{N},  \frac{2dC}{N}\right) - \bI_{n}^{\Lambda^{\delta}}(\varepsilon)\\
\leq\quad & \gamma_{\delta'} \frac{2d^{3/2}C}{N} + \int_{\RR}\left(\widetilde{F}'(x_d;\gamma_{\delta'},\eta_{\delta'},\balpha_{\delta'}, \bbeta_{\delta'}) - F'(x_d;\gamma_{\delta'},\eta_{\delta'},\balpha_{\delta'}, \bbeta_{\delta'})\right)\widehat{\mu}_d^{(n)}(dx_d) + \delta'\\
\leq\quad & B^*\frac{2d^{3/2}C}{N} + \frac{2dC}{N}\sum_{k=1}^{d-1}(1+2d)^{k-1}B_1 +\delta'\\
\leq\quad & \frac{1}{N}\left[B^*2d^{3/2}C + (2B'C+C(\bmu, \varepsilon))(1+2d)^{d-1}\right] + \delta'.
\end{align*}
Letting $\delta'$ goes to zero, together with \eqref{eq.discrete_approximation_general}, we get

\begin{small}
\begin{equation*}
\left|\bI_{n}^{\Lambda^{\delta}}(\varepsilon) - \bI^{\Lambda^{\delta}}_{n, N}\left(\varepsilon + \frac{2d^{3/2}C}{N},  \frac{2dC}{N}\right)\right|\leq\frac{C(\bmu, L,\gamma, d)}{\varepsilon}
\cdot \frac{\sqrt{\log(1/\delta)}}{N}.
\end{equation*}
\end{small}
where $C(\bmu, L, \gamma, d)$ is a constant only depends on $\bmu$,  $L$, $\gamma$ and $d$.
\end{proof}
\clearpage
\bibliography{martDRO.bib}

\end{document}